\documentclass[twoside]{article}
\long\def\symbolfootnote[#1]#2{\begingroup%
\def\thefootnote{\fnsymbol{footnote}}\footnote[#1]{#2}\endgroup}

\usepackage{amsfonts, amssymb, amsmath, amsthm}
\usepackage[colorlinks=true, pdfstartview=FitV, linkcolor=blue,
            citecolor=blue, urlcolor=blue]{hyperref}
\usepackage[usenames]{color}
\definecolor{Red}{rgb}{0.7,0,0.1}
\definecolor{Green}{rgb}{0,0.7,0}
\usepackage{accents}
\usepackage{endnotes}

\title{Time Discrete Approximation of Weak Solutions for Stochastic Equations of Geophysical Fluid
  Dynamics and Applications}
\author{Nathan Glatt-Holtz$^{\sharp}$,  Roger Temam${^\flat}$,  Chuntian Wang${^\flat}$\\
 \small{$^{\sharp}$ {Department} of Mathematics, Virginia Polytechnic and State University}\\
 \small{Blacksburg, VA 24061}\\
 \small{${^\flat}$ Department of Mathematics and The Institute for Scientific Computing and Applied Mathematics}\\
  \small{Indiana University, Bloomington, IN 47405}\\
 \small{emails:
   \url{negh@vt.edu}, \url{temam@indiana.edu}, \url{wang211@indiana.edu}}}

\numberwithin{equation}{section}

\newtheorem{Thm}{Theorem}[section]
\newtheorem{Lem}{Lemma}[section]
\newtheorem{Prop}{Proposition}[section]
\newtheorem{Cor}{Corollary}[section]

\newtheorem{Def}{Definition}[section]
\newtheorem{Rmk}{Remark}[section]

\newcommand{\pd}[1]{\partial_{#1}}
\newcommand{\indFn}[1]{1 \! \! 1_{#1}}
\newcommand{\E}{\mathbb{E}}
\newcommand{\Prb}{\mathbb{P}}
\newcommand{\vel}{\mathbf{v}}

\newcommand{\spX}{\mathbf{x}}
\newcommand{\RR}{\mathbb{R}}
\newcommand{\LP}{\Pi}

\newcommand{\SDT}{\xi}
\newcommand{\SDvar}{s}

\newcommand{\todo}[1]
{{\textcolor{Red}{{\bf [TO DO]}}}}

\newcommand{\homework}[1]{}
\newcommand{\nohomework}[1]{}

\newcommand{\deltaN}{\Delta t}
\newcommand{\mathfrakV}{U^{\sharp}}

\begin{document}

\maketitle

\markboth{Time Discrete Approximation of Weak Solutions for Stochastic Equations of GFD and Applications}
  {Time Discrete Approximation of Weak Solutions for Stochastic Equations of GFD and Applications}

\begin{abstract}
\noindent    As a first step towards the numerical analysis of the stochastic primitive equations of the atmosphere and 
oceans, we study their time discretization by an implicit Euler scheme.
From
deterministic 
 viewpoint the  3D
Primitive Equations are studied 
with 
physically
realistic boundary conditions.
From 
 probabilistic viewpoint we consider a wide class of nonlinear, state dependent,
white noise forcings.
The proof of convergence of the Euler scheme
covers the equations for the oceans,
atmosphere, 
coupled oceanic-atmospheric system
and other
geophysical equations.  We obtain the existence of solutions
 weak in 
PDE and probabilistic sense, 
a result 
which is 
new 
by itself 
to the best of our knowledge.
\end{abstract}

{\noindent \small
  }

\tableofcontents

\newpage

\section{Introduction}
\label{sec:introduction}

The primitive equations of the oceans and atmosphere (PEs) are a fundamental model for the
large scale fluid flows forming the analytical core of the most advanced general circulation models (GCMs)
in use today.  In recent years these systems have been a subject of considerable interest in the mathematical community
not only because of their wide significance in geophysical applications but also for their delicate nonlinear, nonlocal,
anisotropic structure and as a cousin to the other basic equations of mathematical fluid dynamics, namely  the incompressible
Navier-Stokes and Euler equations.

In this work we study a stochastic version of the PEs and develop  techniques which may be viewed
as a first step toward their numerical analysis.  From the point of view of applications, this work is motivated
by a plea from the geophysical community to further develop
the theory of nonlinear Stochastic Partial Differential Equations (SPDEs) in a large scale fluid dynamics context
and in general, \cite{RozovskiiTemamTribbia_Conf}. Indeed, in view of the many sources of uncertainty
both physical and numerical which are typically encountered by the modeler, stochastic techniques are playing an increasingly central role
in the study of geophysical fluid dynamics. See e.g. \cite{Hasselmann1, Rose1, LeslieQuarini1,MasonThomson,
PenlandSardeshmukh, PenlandEwald1,EwaldPenland2,
BernerShuttsLeutbecherPalmer, ZidikheriFrederiksen} and also \cite{GlattTemamTribbia1} for a small
sampling of this vast literature.

The primitive equations trace their origins to the beginning of the
20th century with the seminal works of V. Bjerknes  and L. F. Richardson
(\cite{Bjerknes1,Richardson1}) and have played
a central role in the development of climate modeling and weather prediction since that time,
\cite{Pedlosky}. To the best of our knowledge, the  development of the mathematical theory for the deterministic PEs
began in the early 1990's with a series of articles by J. L. Loins, R. Temam and S. Wang,
 \cite{LionsTemamWang1,LionsTemamWang2,
LionsTemamWang3}.   This direction in mathematical geophysics is now a fairly well developed
subject with results guaranteeing the global existence of
weak solutions which are bounded in $L^{2}_{\spX}$, \cite{LionsTemamWang1} and the global
existence and uniqueness of strong solutions, i.e. solutions evolving continuously in $H^{1}_{\spX}$, \cite{Kobelkov,Kobelkov2007,CaoTiti,ZianeKukavica}.
Of course, these latter developments stand in striking contrast to the current state of the art for
the Navier-Stokes equations as proving the global existence and uniqueness of strong solutions is
tantamount to solving the famous Clay problem.
For further background on the deterministic mathematical theory see the recent surveys
\cite{PetcuTemamZiane,RousseauTemamTribbia}.

Recently, significant efforts have been made to establish suitable analogues of the
above (deterministic) mathematical results in a stochastic setting.
In a series of works, \cite{EwaldPetcuTemam,GlattZiane1,GuoHuang2008,GlattTemam1,
GlattTemam2,DebusscheGlattTemam1,DebusscheGlattTemamZiane1},
the mathematical theory of strong, pathwise\footnote{Here pathwise refers to the fact that
solutions are found relative to a prescribed driving noise.  In this article we will use the
terms `pathwise' and `martingale' as opposed to the alternate terminology of `weak' and `strong' solutions to avoid confusion with the typical
PDE terminology for which weak solutions are, roughly speaking, those in $L^\infty_t (L^2_x)$ and strong solutions are those in $L_t^\infty (H_x^1)$.}
solutions has been developed.
These recent works more or less bring this aspect of the subject
to the state of the art, that is they establish, in increasingly physically realistic
settings,  the global existence and uniqueness of solutions evolving continuously in $H^{1}_{\spX}$.

Notwithstanding the above cited body of works, many aspects of the stochastic theory still need
further consideration. In this article we develop existence results for \emph{weak solutions}, that is solutions
which remain bounded in time only in $L^{2}_{\spX}$.
This is a direction which, to the best of our knowledge, remained unaddressed previously.  Since such
`weak solutions' are not expected to be unique, even in the deterministic setting, it is natural to work within the framework
of \emph{martingale solutions}. In other words we consider below solutions which are weak in both the sense of
PDE theory and stochastic analysis.

One particular advantage of this weak-martingale setting is that it allows us to
consider physical situations unattainable so far in the above cited works on strong (or strong-pathwise) solutions.
From the deterministic point of view we obtain results for the case of inhomogenous, physically realistic boundary
conditions.  On the other  hand, from the stochastic viewpoint our results cover a very general
class of state-dependent (multiplicative) noise structures.  In particular these
noise terms may be interpreted in either the It\=o or Stratonovich sense.  The
later Stratonovich interpretation of noise is important as it may be more realistic in
geophysical settings.  See e.g. \cite{HorsthemkeLefever1}, \cite{Penland2003} for
further details.
Note that we develop our analysis in a slightly abstract setting which at once allows
us to treat the PEs of the oceans, the atmosphere and the coupled
oceanic/atmospheric system.\footnote{We have previously taken such an abstract approach
in other work on the stochastic primitive equations, \cite{DebusscheGlattTemam1}.
There however our focus was on the local existence of strong, \emph{pathwise} solutions
and that framework was, by necessity, more restrictive with respect to domains, noise structures,
etc.}

While the results established here take an important further step in the development of
the analytical theory for the PEs  we believe the main contribution of this article relates to numerical considerations.
The approach below centers on an implicit Euler (i.e. time discrete) scheme
and we choose this set-up mainly because it may be seen as a mathematical setting suitable for the development
of tools needed for the \emph{numerical analysis} of the stochastic PEs and other nonlinear SPDEs arising in fluid
dynamics.  Note that while discrete time approximation has been previously employed
in \cite{DeBouardDebussche2004, DebusschePrintems2006},
these works treat hyperbolic type systems and only address the case of an additive noise.
As such, a number of the techniques developed here,
play a crucial role in a work related to the stability
and consistency of a class of numerical schemes (both explicit and semi-implicit)
for the 2D and 3D stochastic Navier-Stokes equations, \cite{GlattTemam2011}.

Let us now finally turn to sketch some of the main technical challenges and contributions of the article.
In fact the first main difficulty is to justify the validity of the implicit scheme on which our analysis
centers.   While classical arguments involving the Brouwer fixed point theorem can be used
to establish the existence of sequences satisfying the implicit scheme, we crucially need that these
sequences are \emph{adapted} to the driving noise.  To address this concern we rely on  {a specifically chosen filtration and} a suitable measurable selection theorem
from  \cite{BensoussanTemam} (see also \cite{KuratowskiRyllNardzewski1965}, \cite{Castaing1967}).

With suitable solutions to the semi-implicit scheme in hand, basic uniform estimates proceed
analogously to the continuous time case with the use of martingale inequalities, etc.
In contrast to previous works on Martingale solutions  (see e.g.
\cite{Bensoussan1,FlandoliGatarek1,MenaldiSritharan,
DebusscheGlattTemam1,GlattVicol2011}) we circumvent  the need for higher moments
with suitable stopping time arguments.
Another difficulty related to the concern that solutions be adapted appears when we associate
continuous time processes with the discrete time schemes in pursuit of compactness and the passage to the limit.
In contrast to the deterministic case, \cite{Temam1},  \cite{MarionTemam1} we must introduce processes
which are lagged by a time step.  While these processes are indeed adapted, we obtain a time
evolution equation with troublesome error terms.  In turn these error terms prevent us from addressing
compactness directly from the equations and  force us to carry out the compactness arguments for a
series of interrelated processes.

\vskip 1.2 mm

\noindent{\textbf{{Organization of the Article}}}

\vskip 1.2 mm

The exposition is organized as follows.  In Section~\ref{sec:setup}
we outline an abstract, functional-analytic framework for the stochastic Primitive
Equations (and related evolution systems) which may be seen as an
``axiomatic''; basis for the rest of the work.   The section concludes
by recalling the basic notion of Martingale solutions within the context of this framework.
In Section~\ref{sec:DiscreteScheme}
we introduce an implicit Euler scheme which discretizes the equations
in time.  The details of the existence of suitable solutions ({adapted to the specific  filtration})  of this implicit scheme
along with associated
uniform estimates are given in Propositions~\ref{thm:ExistenceSelectionThm}
and \ref{thm:UniformBnds} respectively.  In Section~\ref{sec:ContTimeProcesses}
we study some continuous time processes associated with the implicit
Euler scheme introduced in Section~\ref{sec:DiscreteScheme}.
Section~\ref{sec:TightnessAndLimit} then outlines the compactness
(tightness) arguments that allow us to pass to the limit
and derive the existence of solutions from these approximating
continuous time processes.  Finally, Section~\ref{sec:GeophysicsExample}
provides extended details connecting the abstract results that we just derived with
the concrete example of the primitive equations of the oceans.  In this section
we also provide a number of examples of possible types of nonlinear state dependent
noises covered under the main abstract results.
In the interest of making the manuscript
as self-contained as possible an Appendix (Section~\ref{sec:Appendix}) collects various
technical tools used in the course of our analysis.

\section{The Abstract Problem Set-Up}
\label{sec:setup}

We begin by describing the setting for the abstract evolution equation that
we will study below (cf. \eqref{eq:PEsAbstractFormulation} at the end of this Section).
As we noted in the introduction, we take this point of view in order to systematically
treat the existence of weak solutions
for a class of geophysical fluids equations including but not limited to
the example \eqref{eq:PE3DBasic}--\eqref{eq:3dPEBCTop}
developed below in Section~\ref{sec:GeophysicsExample}.
For further details about how to cast other related equations
of geophysical fluid dynamics in the following abstract formulation we refer the reader
to \cite{PetcuTemamZiane} and the references therein.

Throughout what follows we fix a Gelfand-Lions inclusion of Hilbert spaces
\begin{align}
V_{(3)} \subset V_{(2)} \subset V \subset H \subset V' \subset V_{(2)}' \subset V_{(3)}'.
\label{eq:BasicSpaces}
\end{align}
Each space is densely, continuously and compactly embedded in the next one.
We will
denote the norms for $H$ and $V$ by $| \cdot |$ and $\| \cdot \|$
and the remaining spaces simply by e.g. $\| \cdot \|_{V_{(2)}'}$.
When the context is clear, we will denote the dual pairing
between $V', V$, $V_{(2)}', V_{(2)}$ or $V_{(3)}',V_{(3)}$ by $\langle \cdot , \cdot \rangle$.

\subsection{Basic Operators}
\label{sec:BasicOperators}

We now outline the main elements, a collection of abstract operators, which we use to build
the stochastic evolution \eqref{eq:PEsAbstractFormulation} below.
We suppose we are given:
\begin{itemize}
\item A linear continuous operator $A: V \mapsto V'$ which defines a bilinear continuous form
$a(U, {\mathfrakV}) := \langle AU,  {\mathfrakV} \rangle_{V', V}$ on $V$.
We assume that $a$ is \emph{coercive}, i.e.
\begin{align}
	a(U,U) \geq c_{1} \|U\|^{2} \quad \textrm{ for all } U \in V.
	\label{eq:AvarDef}
\end{align}
This term will typically capture the \emph{diffusive terms} in the concrete equations: molecular and eddy viscosity,
diffusion of heat, salt, humidity etc.\footnote{In previous works on the Stochastic PEs, \cite{GlattTemam1,
GlattTemam2, DebusscheGlattTemam1} we required that this $a$ be symmetric.  In particular such a symmetry was
strongly used in these previous works so that we could apply the spectral theorem to the inverse of an associated operator $A^{-1}$. This is not needed for the arguments presented here and we therefore  revert to the more general weak formulation of the PEs given in \cite{PetcuTemamZiane}.}

\item A second linear operator $E$ continuous on both $H$ and $V$;  $E$ defines a bilinear continuous form
$e(U,  {\mathfrakV}) := (E U,  {\mathfrakV})$ on $H$ (which is also continuous on $V$). We suppose furthermore that $e$ is antisymmetric, that is
\begin{align}
	e(U,U) = 0 \quad \textrm{ for all } U \in H.
	\label{eq:eCancelation}
\end{align}
This term $E$ appears in applications to account for the \emph{Coriolis (rotational)} forces coming from the rotation of the earth.

\item A bilinear form $B$ which continuously maps $V \times V$ into $V_{(2)}'$;
$B$ gives rise to an associated trilinear form
$b(U,  {U^\flat},  {\mathfrakV}) := \langle B(U,  {U^\flat}),  {\mathfrakV}\rangle$
which satisfies the estimates
\begin{align}
|b(U, {U^\flat}, {\mathfrakV})| \leq
c_{2}
\| U \| |  {U^\flat} |^{1/2} \|  {U^\flat} \|^{1/2}  \|  {\mathfrakV} \|_{V_{(2)}}&
\quad \textrm{ for all }
U,  {U^\flat} \in V,  {\mathfrakV} \in V_{(2)}.
\label{eq:SizeEstimatesForB}
\end{align}
Moreover we assume the antisymmetry property
\begin{align}
b(U,\tilde{U}, \tilde{U})  = 0
 \quad \textrm{ for all } U \in V,  \tilde{U} \in V_{(2)}.
 \label{eq:BCan}
\end{align}
Note that, in particular, we may infer from \eqref{eq:SizeEstimatesForB} that
\begin{align}
	\|B(U)\|_{V_{(2)}'} \leq c_{2} |U|^{1/2} \|U\|^{3/2}
	\quad  \textrm{ for any } U \in V.
	\label{eq:BSizeV2Prime}
\end{align}
Furthermore, we infer from \eqref{eq:SizeEstimatesForB}, \eqref{eq:BCan}
we may assume that $B$ is continuous from $V \times V_{(2)}$ into $V'$
and satisfies
\begin{align}
 	\| B(U) \|_{V'} \leq c_{2} \| U\| \|U \|_{V_{(2)}}
	\quad \text{ for all } U \in V_{(2)}.
	\label{eq:BSizeHNorm}
\end{align}
Finally we impose some additional technical convergence conditions
on $b$.  Firstly we suppose that
when $U_{k}$ converges weakly to $U$ in $V$
then, up to a subsequence $k'$,
\begin{align}
b(U_{k'}, U_{k'},  {\mathfrakV}) \rightarrow b(U,U, {\mathfrakV})
\quad \textrm{ for each }  {\mathfrakV} \in V_{(2)}.
\label{eq:WeakLimCond}
\end{align}
Similarly we assume that if, for some $T > 0$,
\begin{align*}
	U_k \rightarrow U& \quad \textrm{weakly in } L^2( 0,T ;V)
	\textrm{ and strongly in } L^2( 0,T;H),
\end{align*}
then, again up to a subsequence $k'$,
\begin{align}
\int_0^Tb(U_{k'}, U_{k'}, {\mathfrakV}) dt \rightarrow \int_0^Tb(U,U,{\mathfrakV}) dt
\quad \textrm{ for each } {\mathfrakV} \in L^\infty(0,T;V_{(3)}).
\label{eq:WeakLimCondTimeInt}
\end{align}
$B$ accounts for the main \emph{nonlinear (convective)} terms in the equations.
\item An externally given element $\ell$. We consider $\ell$ to be random
in general; it is specified only as a probability distribution on $L^{2}_{loc}(0,\infty;V')$ subject to
the second moment condition \eqref{eq:DataMomentAssumptions} given below.
This term $\ell$ captures various inhomogeneous elements i.e. externally determined body forcings,
boundary forcings etc.
\end{itemize}

\noindent In order to define the operators involving the `stochastic
terms' in the equations we consider an auxiliary space $\mathfrak{U}$, on
which the underlying driving noise, a cylindrical Brownian motion
$W$ evolves (see Section~\ref{sec:SomeStochasticAnal}
below).
We suppose $\mathfrak{U}$ is a separable Hilbert space and use
$L_2(\mathfrak{U}, X)$ to denote the space of Hilbert-Schmidt operators
from $\mathfrak{U}$ into $X$, where, for example $X = H, V$ or  $\mathbb{R}$.  Sometimes
we will abbreviate and write $L_{2}:= L_{2}(\mathfrak{U},\mathbb{R})$.

Returning to the list of operators we suppose we have defined:
\begin{itemize}
\item A (possibly nonlinear) continuous map $\sigma: [0,\infty) \times H  \mapsto L_2(\mathfrak{U}, H)$.
We suppose  that $\sigma$ is \emph{uniformly sublinear}, i.e.
\begin{align}
  | \sigma(t,U) |_{L_2(\mathfrak{U}, H)} \leq c_3 (1 + |U|), \quad \textrm{ for every } U \in H
  \textrm{ and } t \in \RR^{+},
  \label{eq:SublinearCondH}
\end{align}
where the constant $c_{3} >0$ is independent of $t \in [0,\infty)$.
For economy of notation we will frequently drop the dependence
on $t$ in the exposition below. We define  $g : [0,\infty) \times H \times H \mapsto L_{2}$ according
to $g(t, U, {\mathfrakV}) = (\sigma(t,U), {\mathfrakV})$ for $U, {\mathfrakV}
\in H$. The element $\sigma$
determines the structure of the (volumic) stochastic forcing applied to the
equations.  These stochastic terms typically appear to
account for various sources of physical, empirical and numerical
uncertainty as we described in the introduction.
\item A continuous map $\SDT: [0,\infty) \times H \mapsto H$
which is subject to the uniform sublinear condition
\begin{align}
	|\SDT(t,U)| \leq c_{4} (1 + |U|), \quad \textrm{ for every } U \in H
	\textrm{ and } t \in \RR^{+},
	\label{eq:SublinearConditionsOnSDTerm}
\end{align}
where $c_{4} > 0$ does not depend on $t \geq 0$.  We define
$\SDvar: [0,\infty) \times H \times H \mapsto \RR$
by
\begin{equation}\label{def:s}
\SDvar(t,U, {\mathfrakV}) = (\SDT(t,U),{\mathfrakV})
\end{equation}
for $U, {\mathfrakV} \in H$.
We include $\SDT$ in the abstract formulation to allow, in particular,
for the treatment of a class of \emph{Stratonovich noises}; $\SDT$ arises
when we convert from a Stratonovich into an It\=o type noise.  This term
$S$ therefore allows us to carry out the forthcoming analysis entirely within the
It\=o framework.  See
Remarks~\ref{rmk:LetsBoastABit},~\ref{rmk:StratFromalSoFar} below.
\end{itemize}

With the above abstract framework now in place we may reduce
the problem  \eqref{eq:PE3DBasic}--\eqref{eq:3dPEBCTop} below
(and related equations)
to studying the following abstract stochastic  evolution equation in $V_{(2)}'$,
namely,
\begin{align}
  dU + (AU + B(U) + EU)\,dt = (\ell + \SDT(U))\, dt + \sigma(U)dW, \quad U(0) = U^0.
  \label{eq:PEsAbstractFormulation}
\end{align}
This system is to be interpreted in the It\={o} sense which we
recall immediately below in Subsection~\ref{sec:SomeStochasticAnal}.
{}

Note that $U^{0}$ and $\ell$   in \eqref{eq:PE3DBasic}
are considered to be random in general.
Indeed, since we are studying \emph{Martingale Solutions} of
\eqref{eq:PEsAbstractFormulation} where the underlying stochastic
elements in the problem are considered as unknowns,
we will specify $U^0$ and $\ell$   only as probability distributions
on $H$ and $L^2(0, T;V')$.  See Definition~\ref{def:MGSol}
and the Remark~\ref{rmk:LetsBoastABit} following.
Note also that, for brevity of notation, we will sometimes write
\begin{align}
\mathcal{N}(t,U) := - (A U + B(U) + E U - \SDT(t,U)),
\label{eq:TheDriftPartTheWholePotatoe}
\end{align}
in the course of the exposition below.  When the context is clear we will sometimes drop
the dependence in $t$ and simply write $\mathcal{N}(U)$.

\subsection{Some Elements of Stochastic Analysis and Abstract Probability Theory}
\label{sec:SomeStochasticAnal}
Of course, \eqref{eq:PEsAbstractFormulation} is
understood  relative to a \emph{stochastic basis} $\mathcal{S}
:= (\Omega, \mathcal{F}, \{\mathcal{F}_t\}_{t \geq 0}, \mathbb{P},
\{W^k\}_{k \geq 1})$,
that is a filtered probability space with $\{W^k\}_{k
  \geq 1}$ a sequence of independent standard $1$-d Brownian motions
relative to $\mathcal{F}_t$.
 Here we may define $W$ on $\mathfrak{U}$
by considering an associated orthonormal basis $\{e_k\}_{k \geq 1}$ of $\mathfrak{U}$ and
taking  $W = \sum_k W_k e_k$; $W$ is thus a `cylindrical Brownian'
motion evolving over $\mathfrak{U}$.

Actually, this sum $W = \sum_k W_k e_k$ is
only formal; it does not generally converge in $\mathfrak{U}$.  For this reason
we will occasionally make use of a larger space $\mathfrak{U}_0 \supset \mathfrak{U}$
which we define according to
\begin{align}
  \mathfrak{U}_0
  := \left\{ v = \sum_{k \geq 0} \alpha_{k} e_{k} :
  | v |_{\mathfrak{U}_{0}}^{2}
  < \infty \right\},
  \quad \textrm{ where }   | v |_{\mathfrak{U}}^{2} := \sum_{k}  {\alpha^{2}_{k}}       \mbox{ and }
  | v |_{\mathfrak{U}_{0}}^{2} := \sum_{k} \frac{\alpha^{2}_{k}}{k^{2}}.
  \label{eq:AuxSpace}
\end{align}
Note that the embedding of $\mathfrak{U} \subset \mathfrak{U}_0$
is Hilbert-Schmidt. Moreover, using standard martingale arguments
with the fact that each $W_{k}$ is almost surely continuous
 we have that, for almost every $\omega \in \Omega$,
  $W(\omega) \in \mathcal C([0,T], \mathfrak{U}_0)$.

Since, \eqref{eq:PEsAbstractFormulation} is
actually short hand for a stochastic integral equation
we next briefly recall some elements of the theory of It\=o stochastic integration
in infinite dimensional spaces.
We choose an arbitrary   Hilbert space $X$ and, as above, we use  $L_2(\mathfrak{U}, X)$
to denote the collection of Hilbert-Schmidt operators from $\mathfrak{U}$ into $X$.
Given an $X$-valued predictable\footnote{For a given  stochastic basis
$\mathcal{S}$, let $\Phi = \Omega \times [0,\infty)$ and take
$\mathcal{G}$ to be the sigma algebra generated by the sets of the
form
\begin{displaymath}
    (s,t] \times F, \textrm{ with }0 \leq s< t< \infty \textrm{ and } F \in \mathcal{F}_s;
    \quad \quad
    \{0\} \times F; \quad F \in \mathcal{F}_0.
\end{displaymath}
Recall that an $X$ valued process $U$ is called predictable (with
respect to the stochastic basis $\mathcal{S}$) if it is measurable
from $(\Phi,\mathcal{G})$ into $(X, \mathcal{B}(X))$ where
$\mathcal{B}(X)$ denotes the family of Borelian subsets of $X$.}
process
$G \in L^{2}(\Omega; L^{2}_{loc}
(0, \infty,L_{2}(\mathfrak{U}, X)))$
the (It\={o}) stochastic integral
\begin{displaymath}
   M_{t} := \int_{0}^{t} G dW = \sum_k \int_0^t G_k dW_k,
   \quad \textrm{ where } G_{k} = G e_{k},
\end{displaymath}
is defined as an element in $\mathcal{M}^2_X$, the space of all
$X$-valued square integrable martingales (see \cite[Section
2.2, 2.3]{PrevotRockner}).    For further details on the general theory of infinite-dimensional
stochastic integration and stochastic evolution equations
we refer the reader to e.g. \cite{ZabczykDaPrato1,PrevotRockner}.

  Since we will be working in the setting of \emph{Martingale solutions},
  where the data in the problem \eqref{eq:PEsAbstractFormulation}
  is specified only as a probability distribution (over an appropriate function
  space), it is convenient to introduce some further notations around
  Borel probability measures. Let $(\mathcal{H},\rho)$ be a complete metric space
  and denote the family of Borel probability measures on
  $\mathcal{H}$ by $Pr(\mathcal{H})$.  Given a Borel measurable function $f: \mathcal{H} \mapsto \mathbb{R}$
  and an element $\mu \in Pr(\mathcal{H})$
  we will sometime write $\mu(f)$ for $\int_{\mathcal{H}} f(x) d\mu(x)$
  when the associated integral makes sense.  In particular we will write
  \begin{align}
	\mu(|f|) < \infty  \iff  \int_{\mathcal{H}} |f(x)| d \mu(x) < \infty.
\end{align}
We will review some basic properties related to convergence and compactness
of subsets of $Pr(\mathcal{H})$ in the Appendix, Section~\ref{sec:LetsConvergeToaMeasure},
below.  We refer the reader to e.g. \cite{Billingsley1} for an extended treatment of the general theory
of probability measures on Polish spaces which include Hilbert spaces such as $H$ and $V$.

\subsection{Definition of Martingale Solutions and Statement of the Main Result}
\label{sec:DefSoln}

We turn now to give a rigorous meaning for the so-called \emph{weak-martingale solutions} of \eqref{eq:PEsAbstractFormulation}
which are defined as follows:
\begin{Def}\label{def:MGSol}[Weak-Martingale Solutions]
Fix $\mu_{U^{0}}$, $\mu_{\ell}$ Borel measures respectively on $H$
and $L^{2}_{loc}$ $(0,\infty;V')$ with
\begin{align}
\mu_{U^{0}}( | \cdot |_{H}^{2} ) < \infty \quad \textrm{ and } \quad \mu_{\ell}( \| \cdot \|_{L^{2}(0, T;V')}^{2}) < \infty,
\textrm{ for any } T > 0.
\label{eq:DataMomentAssumptions}
\end{align}
A \emph{\textbf{weak-martingale solution}}
$({\tilde{\mathcal{S}}}, {\tilde{ {U}}},{\tilde{\mathcal{\ell}}})$ of  \eqref{eq:PEsAbstractFormulation}
consists of a stochastic basis ${\tilde{\mathcal{S}} = (\tilde{\Omega}, \tilde {\mathcal{{F}}}, \{\tilde{\mathcal{{F}}}_t\}_{t \geq
    0},  \tilde{\mathbb{P}}, \tilde{W})}$ and processes ${\tilde U}$ and  ${\tilde{\ell}}$ (defined relative to ${\tilde{\mathcal{S}}}$) adapted to $\{{\tilde{\mathcal{F}}}_{t}\}_{t \geq
    0}$. This triple $({\tilde{\mathcal{S}}}, {\tilde{ {U}}},{\tilde{\mathcal{\ell}}})$ will enjoy the following properties
    \begin{itemize}
    \item[(i)] for every $T > 0$
    \begin{gather}
 {\tilde{ U}} \in L^2({\tilde{\Omega}}; L^\infty(0, T;H)\cap L^{2}(0, T;V)), \quad
   {\tilde{ U}}\textrm{ is a.s. weakly continuous in } H,\\
 {\tilde{\ell}} \in L^{2}( {\tilde{\Omega}}; L^{2}(0, T; V').
    \notag
    \label{eq:solRegularity}
    \end{gather}
    \item[(ii)] For every $t > 0$ and each test function ${{\mathfrakV}} \in V_{(2)}$,
    \begin{align}
    ({\tilde{U}}(t), {{\mathfrakV}}) + \int_{0}^{t} (a({\tilde{U}},{{\mathfrakV}}) &+ b({\tilde{U}},{\tilde{U}}, {{\mathfrakV}}) + e(
    {\tilde{U}}, {{\mathfrakV}})
                                               )ds
                                               \notag\\
        =& ({\tilde{U}}(0), {{\mathfrakV}}) + \int_{0}^{t} (\ell({{\mathfrakV}})+ \SDvar({\tilde{U}}, {
        {\mathfrakV}})) dt +
        \int_{0}^{t} g({\tilde{U}}, {{\mathfrakV}}) d{\tilde{W}},
       \label{eq:solEqWeakForm}
    \end{align}
    almost surely.
   \item[(iii)] Finally,
   ${\tilde{U}}(0)$ and ${\tilde{\ell}}$ have the same laws as $\mu_{U^{0}}$, $\mu_{\ell}$, i.e.
    \begin{align}
       \tilde{\Prb}({\tilde{U}}(0) \in \cdot) = \mu_{ U ^{0}}(\cdot)
       \textrm{ and }  \tilde{\Prb}({\tilde{\ell}} \in \cdot) = \mu_{ {\ell} }(\cdot).
       \label{eq:InitialDataDist}
    \end{align}
\end{itemize}
\end{Def}
With this definition in hand we  now state one of  the main results of the work as follows.
\begin{Thm}\label{thm:MainExistenceMGSol}
 Let $\mu_{U^{0}}$, $\mu_{\ell}$ be
a given pair of Borel measures on
respectively $H$ and $L^2_{loc}(0,\infty; V')$ which
satisfy the moment conditions
\eqref{eq:DataMomentAssumptions}.
Then, relative to this data, there
exists a \textbf{martingale solution} $({\tilde{\mathcal{S}}}, {\tilde{U}}, {\tilde{\ell}})$
of \eqref{eq:PEsAbstractFormulation} in the sense of Definition~\ref{def:MGSol}.
\end{Thm}

\begin{Rmk}\label{rmk:LetsBoastABit}
Depending on the structure of $\sigma$ the application of noise
leads to a variety of different  effects on the behavior of the solutions.
In particular $\sigma$ can be
chosen so that the noise either provides a damping or an exciting effect.
It is therefore unsurprising that the structure of the stochastic
terms in e.g. \eqref{eq:PE3DBasic} remains a subject of ongoing
debate among physicists and
applied modelers. In any case,
viewed as a proxy for physical and numerical uncertainty,
the structure of the noise would be expected to vary by application.
With this debate in mind we have therefore sought to treat a very general
class of state-dependent
noise structures in $\sigma$ requiring only the sublinear condition
\eqref{eq:SublinearCondH}.  We have illustrated some interesting examples
covered under this condition in Section~\ref{sec:StochasticForcings}
below.

Actually, the Stratonovich interpretation of white noise driven forcing may often be more
appropriate for applications in geophysics.  See e.g. \cite{HorsthemkeLefever1}, \cite{Penland2003}
for extended discussions on this connection. Note that although the equations \eqref{eq:PEsAbstractFormulation} are
considered in  an It\=o sense, an additional, state dependent drift term
$\SDT$ has been added to the equations which allows us to treat
a class of Stratonovich noises with \eqref{eq:PEsAbstractFormulation} via the standard `conversion
formula' between It\=o and Stratonovich evolutions.  See e.g.
\cite{Arnaud1974} and also Section~\ref{sec:StochasticForcings}
where we present one such example of Stratonovich forcing in detail.
\end{Rmk}

\section{A Discrete Time Approximation Scheme}
\label{sec:DiscreteScheme}

We now describe in detail the semi-implicit Euler scheme, \eqref{eq:EulerSemiFnRep},
which we use to approximate \eqref{eq:PEsAbstractFormulation}. 
This system is given rigorous meaning in Definition~\ref{def:AdminEulerSchemes}.  We then recall a specific 
stochastic basis in Section~\ref{sec:specificfiltration}  and establish the existence of solutions of \eqref{eq:EulerSemiFnRep} in Proposition~\ref{thm:ExistenceSelectionThm}
relative to this basis.
We conclude this section by providing certain uniform bounds (energy estimates) independent of the time step of the discretization in Proposition~\ref{thm:UniformBnds}.

\subsection{The Implicit Scheme}

Fix a stochastic basis $\mathcal{S}= (\Omega, \mathcal{F}, \{\mathcal{F}_t\}_{t \geq 0}, \mathbb{P},
\{W^k\}_{k \geq 1})$ and elements $\ell  \in L^2 (\Omega; L^2 _{loc} (0, \infty; V^\prime))$, $U^{0}\in L^2 (\Omega; H)$ whose distributions
correspond to the externally given $\mu_{\ell}$, $\mu_{U^{0}}$.
For a given $T > 0$ and any integer $N$, let
\begin{align}
\deltaN = T/N, \quad t^{n} =  t^n_N = n \deltaN, \quad \textrm{ for } n = 0, 1, \ldots, N,
\label{eq:DetTimeIncDef}
\end{align}
along with the associated stochastic increments
\begin{align}
\eta^{n} = \eta_N^n = W(t_n) - W(t_{n-1}),  \quad \textrm{ for } n = 1, \ldots, N.
\label{eq:StochIncDef}
\end{align}
Using an implicit Euler time discretization scheme
we would then like to approximate \eqref{eq:PEsAbstractFormulation}
by considering sequences $\{U^n_N\}_{n =1}^N$ satisfying
\begin{equation}\label{eq:EulerSemiFnRep}
   \frac{U^{n}_N - U^{n-1}_N}{\deltaN}
   + AU^n_N
   + B(U^n_N)
   + EU^n_N
   = \ell^n_N    + \SDT(t^{n}, U^{n}_{N})
       + \sigma_{N}(t^{n-1}, U^{n-1}_{N}) \frac{\eta_N^n}{\deltaN},
\end{equation}
in $V_{(2)}'$ for $n = 1, \ldots N$.
For how to choose $U_N^0$, see Remark \ref{rmk:SomeCrazyApproximation}. The terms $\ell^n_N$ are given by
\begin{align}
  		\ell^n_N({{\mathfrakV}})
  		= \frac{1}{\deltaN} \int_{(n-1)\deltaN}^{n\deltaN}
    		 \ell(t, {{\mathfrakV}}) dt \quad \textrm{ for } n =1, 2, \ldots, N,
   			  \label{eq:EllIncrement}
\end{align}
and the operator $\sigma_{N}: [0,\infty) \times H \rightarrow L_{2}(\mathfrak{U},V)$
is any approximation of $\sigma$ which satisfies
\begin{align}
	\|\sigma_{N}(t,U)\|_{L_{2}(\mathfrak{U}, V)}^{2} &\leq N |\sigma(t,U)|_{L_{2}(\mathfrak{U}, H)}^{2},
	\label{eq:SigmaApproxCond1}\\
	|\sigma_{N}(t,U)|_{L_{2}(\mathfrak{U}, H)}^{2} &\leq |\sigma(t,U)|_{L_{2}(\mathfrak{U}, H)}^{2},
	\label{eq:SigmaApproxCond2}
\end{align}
for every $t \geq 0$ and every $U \in H$.
 Additionally we suppose that, for any $t \geq 0$,
\begin{align}
	\lim_{N \rightarrow \infty} \sigma_{N}(t,U_N) &= \sigma(t,U),
	\quad \textrm{ whenever } U_N \rightarrow U \textrm{ in } H.
	\label{eq:SigmaApproxCond3}
\end{align}	
For the existence of such $\sigma_N$, see Remark \ref{rmk:SomeCrazyApproximation}.
We write $g_{N}(t,U, {{\mathfrakV}}) = (\sigma_{N}(t,U), {{\mathfrakV}})$.\footnote{The
choice of a ``time explicit" term in
$\sigma_{N}(t^{n-1},U^{n-1}_N)$ is needed to obtain the correct (It\=o) stochastic
integral in the limit as $\deltaN \rightarrow 0$.  Actually, this adaptivity (measurability) concern also
leads us to introduce the approximations of $\sigma$ in \eqref{eq:EulerSemiFnRep};
see Remark~\ref{rmk:SomeCrazyApproximation}
and \eqref{eq:InEQPlusError}, \eqref{eq:StocErrorL2V} below.
Note that, as explained in this Remark approximations of $\sigma$
satisfying \eqref{eq:SigmaApproxCond1}--\eqref{eq:SigmaApproxCond3} can always
be found via an elementary functional-analytic construction.}

We make the notion of suitable solutions of \eqref{eq:EulerSemiFnRep} precise in the following definition.
\begin{Def}\label{def:AdminEulerSchemes}
    We consider a stochastic basis  $\mathcal{S} = (\Omega, \mathcal{F}, \{\mathcal{F}_t\}_{t \geq
    0}, \mathbb{P}, \{W^k\}_{k \geq 1})$. Given  $N \geq 1$ and an element $U_{N}^{0} \in L^2(\Omega,H)$  which is
    $({{\mathcal{{F}}}}_{0}, \mathcal{B}(H))$ measurable
    and a   process $\ell =\ell (t) \in L^{2}(\Omega; L^{2}(0,T; V'))$ adapted to $\{ \mathcal{F}_t\}_{t \geq
    0}
    $,
    we say that a sequence $\{U^{n}_{N}\}_{n = 0}^{N}$
    is an \emph{admissible solution of the Euler Scheme} \eqref{eq:EulerSemiFnRep}, if
    \begin{itemize}
     \item[(i)]  For each $n = 1, \ldots, N$,  $U^{n}_{N} \in L^{2}(\Omega; V)$
     and $U_N^{n}$ is
     ${{\mathcal{{F}}}}_{n}$ adapted,
     where ${{\mathcal{{F}}}}_{n} := {{\mathcal{{F}}}}_{t^n}$,
     $n = 0, \ldots, N$.
     \item[(ii)] Every pair $U^n_{N},U^{n-1}_N$, $n = 1, \ldots, N$,
     satisfies
     \begin{align}
	   (U^{n}_N - U^{n-1}_N, {{\mathfrakV}})
	   +& \left(a(U^n_N, {{\mathfrakV}})
	   + b(U^n_N, U^n_N, {{\mathfrakV}})
	   + e(U^n_N, {{\mathfrakV}})\right) \deltaN
	   \notag\\
	   &\quad \quad \quad = \left(\ell^n_{N}({{\mathfrakV}})  + \SDvar(t^{n},U^{n}_{N}, {{\mathfrakV}})\right)\deltaN
	       + g_{N}(t^{n-1},U^{n-1}_N, {{\mathfrakV}}) \eta_N^n,
	       \label{eq:EulerSemi}
	\end{align}
	almost surely for all ${{\mathfrakV}} \in V_{(2)}$.
	\item[(iii)]  For each $n = 1, \ldots, N$, $U_n^{N}$ and $U_{n-1}^N$
     satisfy  the `energy inequality', almost surely on $\Omega$:
\begin{align}
 (U^n_N -U^{n-1}_N, U^n_N)
 +  \deltaN c_{1} \| U^n_N \|^{2}
 \leq  \Bigl( \ell^n_N(U^n_N) + \SDvar(t^{n},U^{n}_{N}, U_{N}^{n}) \Bigr) \deltaN+
        g_{N}(U^{n-1}_N, U^n_N) \eta_N^n,
        \label{eq:DiscreteInequality1}
\end{align}
 for $n = 1, 2, \ldots, N$ and where $c_1$ is the constant from \eqref{eq:AvarDef}.
    \end{itemize}
\end{Def}

\begin{Rmk}\label{rmk:SomeCrazyApproximation}
At first glance the dependence on $N$ in both the initial condition and in the
noise term involving $\sigma$ may seem strange.
Indeed, in the deterministic setting, when we approximate
\eqref{eq:PEsAbstractFormulation} with \eqref{eq:EulerSemiFnRep}, we would simply
take $U_N^0$ to be  equal to the
initially given $U^0$ for all $N$.  Similarly if we were to add
deterministic sublinear terms analogous to $\sigma$ to the
governing equations no approximation as in \eqref{eq:SigmaApproxCond1}--\eqref{eq:SigmaApproxCond3}
would be necessary; however, the situation is, in general, more complicated in the stochastic setting as we shall see
in detail later on in Section~\ref{sec:ContTimeProcesses}, Proposition~\ref{prop:BasicPropertiesOfConProcesses}.
This is essentially because we must construct continuous time processes from the $U^n_N$'s
which are adapted to a given filtration.   See \eqref{eq:InEQPlusError}, \eqref{eq:ErrorTerm1}--\eqref{eq:ErrorTerm2}
\eqref{eq:DetErrorDecalLinfH} and \eqref{eq:StocErrorL2V} for specific details.

For now let us describe how we can achieve suitable
approximations in the $U^0_N$ and $\sigma_N$'s.
\begin{itemize}
\item
For a given initial probability distributions $\mu_{U^{0}}$,
on $H$  (with $\mu_{U^{0}}(|\cdot|^{2}_{H}) <\infty$)
and having fixed a suitable stochastic basis and an
element $U^{0} \in L^{2}(\Omega; H)$, ${{\mathcal{{F}}}}_{0}$-measurable,
with distribution
$\mu_{U^{0}}$.
We then pick a sequence $U^{0}_{N} \in L^{2}(\Omega;V_{(2)})$
such that
$U_{N}^{0} \rightarrow U^{0}$ \textrm{  as  }  $N \rightarrow \infty $ \textrm{ in } $L^{2}(\Omega;H)$
but subject to the restriction given in \eqref{eq:InitialDataBnd} below.  Such a sequence
can be found with a simple density argument.  Indeed, since $V_{(2)}$ is dense in $H$, we may initially
approximate $U^{0}$ in $L^{2}(\Omega,H)$ with a sequence $\bar{U}_{M}^{0} \in L^{\infty}(\Omega; V_{(2)})$.
We then define $M(N) = \max\{ M \geq 1: \| \bar{U}^{0}_{M}\|_{L^{\infty}(\Omega; V_{(2)})} \leq N^{1/2}\} \wedge N$
and define $U_{N}^{0} = \bar{U}_{M(N)}^{0}$.  Since $M(N) \rightarrow \infty$ as $N \rightarrow \infty$,
$U_{N}^{0}$ approximates $U^{0}$ in $L^{2}(\Omega;H)$ while maintaining the constraint \eqref{eq:InitialDataBnd}.
\item We may construct elements $\sigma_N$ from $\sigma$ satisfying \eqref{eq:SigmaApproxCond1}--\eqref{eq:SigmaApproxCond3}
according to the following general functional analytic construction.  For any $U \in H$, via Lax-Milgram we
define $\Psi (U)$ to be the unique solution in $V$ of $((\Psi (U), U^\sharp)) = (U, U^\sharp)$ for all $U^\sharp \in V$.
Classically  $\Psi$ is a compact, self-adjoint and injective linear operator on $H$.  Thus, by the Spectral Theorem, we may find
a complete orthonormal basis for $H$ $\{\Phi_j\}_{j \geq 1}$ which is made up of eigenfunctions of $\Psi$ with a corresponding sequence of
eigenvalues $\{\gamma_j\}_{j \geq 1}$ decreasing to zero.  For any integer $m$ we let $P_m$ to be the
projection onto $H_m := \mbox{span} \{ \Phi_1, \ldots, \Phi_m \}$.  Now choose a sequence
$m_N$ increasing to infinity  but so that $\gamma_{m_N}^{-1} \leq N$.  It is not hard to see that defined in this way
$\sigma_N(\cdot) =  P_{m_N} \sigma(\cdot)$
satisfies the requirements given in \eqref{eq:SigmaApproxCond1}--\eqref{eq:SigmaApproxCond3}.
\end{itemize}
\end{Rmk}

\subsection{Existence of the $U^n_N$'s}
\label{sec:ExistenceForImpEuler}

While the existence for a.e. $\omega \in \Omega$  of solutions to  \eqref{eq:EulerSemiFnRep} satisfying \eqref{eq:DiscreteInequality1}
follows along arguments similar to those found in \cite[Lemma 2.3]{PetcuTemamZiane},
some care is required to demonstrate the existence of sequences  $\{U_N^n\}_{n =0}^N$
which are \emph{adapted} to the underlying stochastic basis.
For this complication we will make use of a `measurable selection theorem' (Theorem~\ref{thm:MeasureSel} below in the Appendix Section~\ref{sec:LetsJustBeinMeasurableOK?})
from \cite{BensoussanTemam} (and see also the related
earlier works \cite{KuratowskiRyllNardzewski1965},
\cite{Castaing1967}).  In order to apply this result we use of a specific stochastic basis defined around the canonical Wiener space whose definition we recall next.

\subsubsection{{The Wiener measure and its filtration}}
\label{sec:specificfiltration}

We recall the canonical Wiener space as follows; see \cite{KaratzasShreve} for further details.  Let 
  \begin{equation*}
  \Omega  = \mathcal C ([0, T]; \mathfrak U_0),
  \end{equation*}
 equipped with the Borel $\sigma$-algebra denoted as $\mathcal{G}$.  We equip
 $(\Omega, \mathcal{G})$ with the Wiener measure $\mathbb P $.\footnote{Using the orthonormal   basis $\{e_k\}_{k \geq 1}$ of $\mathfrak{U}$,  
  $\mathbb{P}$ is obtained as the product of the independent Wiener measures each one defined on $\mathcal C ([0, T]; \mathbb{R})$.}
Then the evaluation map $W (\omega, t):= \omega (t)$, $\omega \in  \Omega  $, $t \in [0, T]$, is a cylindrical Wiener process on $\mathfrak U_0$.
The filtration is given by $\mathcal{G}_t$ defined as 
 \begin{align*}
 \mbox{the completion of the sigma algebra generated by the }  W (s)  \mbox{ for  }  s \in [0, t]   \mbox{ with respect to }\mathbb P .
  \end{align*}
 Combining these elements $\mathcal{S}_{\mathcal{G}} = (  \Omega , \mathcal{G},  \{\mathcal{G}_t \}_{t \geq 0},\Prb, W)$ gives a stochastic
 basis suitable for applying Theorem~\ref{thm:MeasureSel}.

\subsubsection{Existence of the $U^n_N$'s adapted to ${\mathcal{G}}_{t_n} $ }
\label{sec:ExistenceForImpEuler}

\begin{Prop}\label{thm:ExistenceSelectionThm}
Suppose that
\begin{align}
N \geq  N_{0} := 4 T c_{4},  \quad (\textrm{or equivalently that } 4 \deltaN  c_{4} < 1),
\label{eq:ExistenceStartingDiscretization}
\end{align}
where $c_{4}$ is the constant
arising in \eqref{eq:SublinearConditionsOnSDTerm}.
Consider the stochastic basis $\mathcal{S}_{\mathcal{G}}$  defined as in Section~\ref{sec:specificfiltration}, an $N \geq N_{0}$, and an element $U_N^0 \in L^2(\Omega;H)$
which is ${{\mathcal{G}}}_0$-measurable and a  process $\ell =\ell (t) \in L^{2}(\Omega; L^2(0,T;V')$ measurable with respect to the sigma algebra generated by the  $    {W} (s)$ for $ s \in [0, t]   $. Then there  exists a
sequence $\{U_N^n\}_{n =0}^N$ which is an admissible solution of the Euler scheme \eqref{eq:EulerSemiFnRep}
in the sense of Definition~\ref{def:AdminEulerSchemes}
\end{Prop}

The rest of this subsection is devoted to the proof of Proposition~\ref{thm:ExistenceSelectionThm}. Below we will  construct the sequence $\{U^{n}_{N}\}_{n = 0}^{N}$ iteratively
starting from $U_N^0$
but we first need to take the preliminary step of establishing the existence
of a certain Borel measurable map $\Gamma: [0,T] \times V' \rightarrow V$
which is used at the heart of this
construction.

We define the continuous map $\mathfrak{G}: [0,T] \times V \rightarrow V_{(2)}'$ according to
\begin{align}
\mathfrak{G}(t,U) = U + \deltaN \bigl(AU + B(U) + EU - \SDT(U,t)\bigr),
\end{align}
and, for each $t \in [0,T]$ and $F \in V'$ we set:
\begin{align}
	\Lambda(t,F) = \left\{ U \in V : \langle \mathfrak{G}(t,U)- F, {{\mathfrakV}} \rangle = 0,
	\forall  {{\mathfrakV}} \in V_{(2)}
	\textrm{ and }  |U|^{2}
 +  \deltaN c_{1} \| U \|^{2} \leq \langle F + \SDT(t,U)\deltaN, U \rangle
 \right\}.
 \label{eq:LambdaDef}
\end{align}
Using this family of sets defined by \eqref{eq:LambdaDef} we now
establish the following Lemma:
\begin{Lem}\label{lem:BorelMappingLemma}
There exists a map $\Gamma: (0,T) \times V' \rightarrow V$ {which is universally Radon measurable (Radon measurable for every Radon measure on $(0,T) \times V'$)},
such that for every $t \in (0, T)$ and every $F \in V'$,
$U := \Gamma(t,F) \in \Lambda(t,F)$.
\end{Lem}
\begin{proof}
We establish the existence of the desired $\Gamma$ by showing that $\Lambda$
satisfies the conditions of Theorem~\ref{thm:MeasureSel}.  More precisely we need
to verify that\footnote{To apply Theorem~\ref{thm:MeasureSel} we
actually would like to define $\Lambda$ on the Banach space $\mathbb{R} \times V'$.
For this purpose we may simply take $\Lambda(t,F) = \Lambda(T,F)$ when $t > T$
and when $t < 0$ we let $\Lambda(t,F) = \Lambda(0,F)$.}
\begin{itemize}
\item[(i)] for each $t \in [0,T]$, $F \in V'$, the set $\Lambda(t,F)$ is \emph{non-empty} and that
\item[(ii)] $\Lambda(t,F)$ is \emph{closed}.  In other words we need to show that,
given any sequences
\begin{align*}
t_{n} \rightarrow t, \quad F_{n} \rightarrow F \textrm{ in } V', \quad U_{n} \rightarrow U \in V
\end{align*}
such that, for every $n$,
\begin{align*}
\langle \mathfrak{G}(t_{n},U_{n})- F_{n}, {{\mathfrakV}} \rangle = 0,
	\textrm{ for every } {{\mathfrakV}} \in V_{(2)}
	\textrm{ and }  |U_{n}|^{2}
 +  \deltaN c_{1} \| U_{n} \|^{2} \leq \langle F_{n} + \SDT(t_{n},U_{n})  \deltaN, U_{n} \rangle,
\end{align*}
we have
\begin{align*}
\langle \mathfrak{G}(t,U)- F, {{\mathfrakV}} \rangle = 0,
	\textrm{ for every } {{\mathfrakV}} \in V_{(2)}
	\textrm{ and }  |U|^{2}
 +  \deltaN c_{1} \| U \|^{2} \leq \langle F + \SDT(t,U) \deltaN, U \rangle.
\end{align*}
\end{itemize}

The first item, (i) may be established with a Galerkin scheme and
the Brouwer fixed point theorem along standard arguments typically used to prove the existence of
solutions for nonlinear elliptic equations of the type of Navier-Stokes and primitive equations (see Lemma 2.3, Page 26  in \cite{PetcuTemamZiane}).
 Since some specifics are
different here we briefly sketch some details of this argument.
Fix any $t \in [0,T]$ and any $F \in V'$ and consider a family $\{\Psi_{k}\}_{k \geq 1} \subset V_{(2)}$ which is free and total
in $V$.  For each $m \geq 1$ we seek an element $U_{m} = \sum_{j=1}^{m} {\beta}_{jm} \Psi_{j}$
such that
\begin{align}
	\langle \mathfrak{G}(t,U_{m}) - F, \Psi_{k}\rangle = 0 \textrm{ for every } k = 1, \ldots, m.
	\label{eq:mthGalAprox}
\end{align}
Observe that, for any $U_{m}$ of this form, using \eqref{eq:AvarDef}, \eqref{eq:eCancelation},
\eqref{eq:BCan}, and \eqref{eq:SublinearConditionsOnSDTerm}
we estimate
\begin{align*}
	\langle \mathfrak{G}(t,U_{m}) - F, U_{m}\rangle
	=& |U_{m}|^{2} + \deltaN (a(U_{m},U_{m}) - (\SDT(t, U_{m}),U_{m})) - \langle F, U_{m} \rangle\\
	\geq& |U_{m}|^{2} + \deltaN (c_{1} \|U_{m}\|^{2} -2 c_{4}(1 + |U_{m}|^{2})) - | F|_{V'} \| U_{m} \|\\
	\geq& \frac{ \deltaN c_{1} }{2} \|U_{m}\|^{2} - \frac{1}{2}\left(1 + \frac{1}{ \deltaN c_{1} } | F|_{V'}^2 \right).
\end{align*}
The last inequality follows from the assumption \eqref{eq:ExistenceStartingDiscretization} which implies
 that $2c_{4} \deltaN \leq 1$.
The existence of solutions for \eqref{eq:mthGalAprox} for any given $t, F$ of the form
$U_{m} = \sum_{j=1}^{m} {\beta}_{jm} \Psi_{j}$ thus follows for each $m$
from the Brouwer fixed point theorem.

We next seek   bounds on the resulting sequence of $U_{m}$'s in $V$ independent of $m$.  Starting
from \eqref{eq:mthGalAprox} we find that
\begin{align}
	|U_{m}|^{2} + c_{1}\deltaN \|U_{m}\|^{2} \leq& \deltaN (\SDT(t, U_{m}),U_{m}) + \langle F, U_{m} \rangle
	\notag\\
		\leq& 2c_{4} \deltaN (1+ |U_{m}|^{2}) +  \frac{1}{2 \deltaN c_{1} } | F|_{V'}^2 + \frac{ \deltaN c_{1} }{2} \|U_{m}\|^{2}.
		\label{eq:UniformBndsGalerkinApprox}
\end{align}
Using once again the standing assumption \eqref{eq:ExistenceStartingDiscretization} we have that $U_{m}$ is   bounded
in $V$ independently of $m$.  Passing to a subsequence as needed and using that $V$ is compactly embedded
in $H$ we infer the existence of an element $U$ such that $U_{m} \rightarrow U$ weakly in $V$ and strongly in $H$.

Returning to \eqref{eq:UniformBndsGalerkinApprox} and using the
lower semicontinuity of weakly convergent sequences we obtain that
$ |U|^{2} + c_{1}\deltaN \|U\|^{2} \leq \langle \SDT(t,U) + F, U \rangle$.  To show
that $U$ satisfies $\langle \mathfrak{G}(t,U)- F, {{\mathfrakV}} \rangle = 0$ for every ${{\mathfrakV}} \in V_{(2)}$
we simply invoke \eqref{eq:WeakLimCond} for $B$ and the other continuity assumptions on $A$,
$E$ and $\SDT$ and obtain this identity for ${{\mathfrakV}} = \Psi_{k}$ for each $k \geq 1$.
By linearity and density we therefore infer the identify for arbitrary ${{\mathfrakV}} \in V_{(2)}$.
With this we now have established (i). The second item, (ii), to show that $\Lambda$ is closed,
follows immediately from the continuity of $\mathfrak{G}$ from $[0,T] \times V$ into $V_{(2)}'$
and the continuity of $\SDT$ from $[0,T] \times H$ into $H$.  The proof of Lemma~\ref{lem:BorelMappingLemma}
is therefore complete.
\end{proof}

\paragraph{{Construction of an Adapted Solution}}
\label{sec:Construction}

\paragraph{{Step 1.}}{We will build
the desired sequence $\{U^{n}_{N}\}_{n = 0}^{N}$ inductively as follows:
 \begin{equation}\label{eq:u n}
 U^{n}_N =f^{n}_N (W \big|_{[0, t_{n}]} ),
 \end{equation}
 with $f^{n}_N  :\mathcal C ([0, t_{n}]; \, \mathfrak U_0) \rightarrow V$   measurable for $V$ equipped with
  $\mathcal B (V)$ and $\mathcal C ([0, t_{n}]; \,\mathfrak U_0)$ equipped with ${\mathcal G}_{n}:= {\mathcal G}_{t_n}$ (defined as in Section \ref{sec:specificfiltration}). }

Suppose that we have obtained $U^{n-1}_N $ for some $n \geq 2$.
 Since  ${\mathcal G}_{n-1}$ is  the completion of $\mathcal B (\mathcal C ([0, t_{n-1}]; \,\mathfrak U))$ with respect to the Wiener measure ${\mathbb P}$\footnote{We observe that the sigma algebra generated by the  $W (s)$  for   $s \in (0, t)$ is just   $ \mathcal \phi _t ^{-1} ( \mathcal B (\mathcal C ([0, T]; \mathfrak U_0))$,    where $\mathcal \phi_t: $ $ \mathcal C ([0, T]; \mathfrak U_0) \rightarrow \mathcal C ([0, T]; \mathfrak U_0)$ is the mapping $(\mathcal \phi _t ^{-1} \omega) (s) = \omega (t \wedge s); $ $0\leq s\leq T$  (see \cite{KaratzasShreve}).}, $f^{n-1}_N $  is  ${\mathbb P}$-measurable.
{Now we define $\mathfrak D^{n}_N : V \times V^\prime \times  \mathcal C ([0, t_{n}]; \, \mathfrak U_0)   \rightarrow V^\prime$ by setting
\begin{equation}\label{defofdn}
 \mathfrak D^{n}_N (x, y, z)  =x + y \, \Delta t  +  \sigma _N (t^{n-1}, U )z .
\end{equation}
Then we can   define
\begin{equation}\label{eq:unn1}
\begin{split}
U^n_N &=   \Gamma \left (t^n,\,\, \mathfrak D^{n}_N  \left(  U^{n-1}_{N}, \,\ell^{n}_{N} , \,  \eta^{n}_{N}   \right)\right)\\
&: = \chi   \left (t^n, \, U^{n-1}_{N},\, \ell^{n}_{N} , \,  \eta^{n}_{N}   \right) .
 \end{split}
\end{equation}
Since $\sigma_N$ is a continuous map, clearly $  \mathfrak D^{n}_N   $    is a continuous map. Moreover    $\Gamma$ is universally Radon measurable thanks to Lemma  \ref{lem:BorelMappingLemma}, hence Corollary \ref{cor:cont} applies and we infer that  $\chi  $ is    universally Radon measurable from the
Borel sigma algebra on $V \times V^\prime \times \mathcal  C ([0, t_{n}]; \, \mathfrak U_0)$ to the Borel sigma algebra on $V$  .}

  {Since $\ell = \ell(t)$ is a process assumed to be measurable with respect to the sigma algebra generated by the  $   {W} (s)$ for $s \in [0, t]   $, $\ell^{n}_{N}$ is measurable with respect to
  the sigma algebra generated by  the $  {W} (s)$ for $ s \in [0, t_n]$ thanks to  \eqref{eq:EllIncrement}. Hence by Theorem   \ref{doob} in the Appendix  with $\mathcal X$ as $\Omega$, $(\mathcal Y, \mathcal M)$ as $ ( \mathcal C ([0, t_{n}]; \, \mathfrak U_0), \mathcal B ( \mathcal C ([0, t_{n}]; \, \mathfrak U_0)   ))$, $\psi$ as  $W \big|_{[0, t_n]}$,   $\mathcal H$ as $V$, we see that there exists a function $L^n_N :\,\,\, \mathcal C ([0, t_n]; \,\mathfrak U_0) \rightarrow V$ which is Borel measurable,  such that
\begin{equation}\label{eq:LnN}
\ell ^n _N  =L^n_N   ( W \big|_{[0, t_n]} )   .
\end{equation}
 From  \eqref{eq:unn1} and \eqref{eq:LnN} we infer
 \begin{equation}\label{eq:unn2}
 \begin{split}
U^n_N &=  \kappa  (t^n,\,f^{n-1}_N (W \big|_{[0, t_{n-1}]} ) ,\, L^n_N  ( W \big|_{[0, t_n]} ), \,\eta^{n}_{N} )\\
&:=  f^n_N (W \big|_{[0, t_n]} ).
\end{split}
\end{equation}
Since $L^n_N  $  and  $f^{n-1}_N$ are   $ {\mathbb P}$-measurable and $\kappa $  is universally Radon measurable,
Theorem \ref{lem:composition} applies and we infer that $f^n_N$ is $ {\mathbb P}$-measurable, that is, $f^n_N$ is measurable with  respect to   ${\mathcal G}_{n}$.}

\paragraph{{Step 2.}}
{We infer that  $U^n_N : \Omega \rightarrow V$ is measurable with respect to   $\mathcal{G}_{n}$ as desired.}\\

 Observe moreover that,
according to Lemma~\ref{lem:BorelMappingLemma} (cf. \eqref{eq:LambdaDef}),
$
\langle \mathfrak{G}(t_{n},U^{n}_{N}), {{\mathfrakV}} \rangle
=
{\langle  \mathfrak D^{n}_N  \left(  U^{n-1}_{N}, \ell^{n}_{N} ,   \eta^{n}_{N}   \right) ,\, \tilde U \rangle}$,
for every
$
{{\mathfrakV}} \in V_{(2)}
$
and
$
  |U^{n}_{N}|^{2}
 +  \delta t c_{1} \| U^{n}_{N} \|^{2} \leq  { \langle   \mathfrak D^{n}_N  \left(  U^{n-1}_{N}, \ell^{n}_{N} ,   \eta^{n}_{N}   \right), \, U^n _N   \rangle}
$
which is to say that $U^{n-1}_{N}$ and $U^{n}_{N}$ satisfy \eqref{eq:EulerSemiFnRep}
and \eqref{eq:DiscreteInequality1}.

It remains to show that $U^{n}_{N} \in L^{2}(\Omega;V)$. We start from
\eqref{eq:DiscreteInequality1}, now established for $U^{n}_{N}$
and $U^{n-1}_{N}$, and use the elementary identity
$2(U - {{\mathfrakV}} , U) = |U|^2 -|{{\mathfrakV}}|^2 + |U - {{\mathfrakV}}|^2$ and
 obtain,
\begin{align}
|U^n_N|^2 - |U^{n-1}_N|^2 + |U^n_N - U^{n-1}_N|^2
  &+  2\deltaN c_1\|U^n_N\|^2 \notag\\
  &\leq
  2 \deltaN \bigl(\ell^n_N(U^n_N) + \SDvar(t^{n},U^{n}_{N}, U_{N}^{n})\bigr)
  + 2g_{N}(U^{n-1}_N, U^n_N) \eta_N^n,
  \label{eq:AppOfElementaryId}
\end{align}
almost surely. To address the terms involving $\ell$ we have that (cf. \eqref{eq:EllIncrement})
\begin{align*}
	|2 \deltaN \ell^n_N(U^n_N)| \leq 2 \int_{(n-1)\deltaN}^{n\deltaN} \| \ell (t) \|_{V'} \|U^n_N\|_Vdt
	\leq c_1 \deltaN \|U^n_N\|^2  + c_1^{-1}\zeta^n_N
\end{align*}
where we define $\zeta^n_N$ according to
\begin{align}
	\zeta^n_N=  \int_{(n-1)\deltaN}^{n\deltaN} \| \ell \|_{V'}^2 dt.
		\label{eq:ExternalForcingEqnSum}
\end{align}
For the terms involving $\SDvar$ defined as in (\ref{def:s}) we simply infer
from \eqref{eq:SublinearConditionsOnSDTerm}
\begin{align}
  2 \deltaN |\SDvar(t^{n},U^{n}_{N},U^{n}_{N})| \leq 4 \deltaN c_{4} (1+ |U^{n}_{N}|^{2}).
  \label{eq:SDvarDirectEstimate}
\end{align}
 With  H\"{o}lder's inequality
 we find
\begin{equation}\label{eq:g}
  \begin{split}
    |2 g_{N}(U^{n-1}_N,U^{n}_N -  U^{n-1}_N) \eta_N^n|
    \leq& \frac{1}{2} |U^{n}_N -  U^{n-1}_N|^2 +
            2|\sigma_{N}(U^{n-1}_N) \eta_N^n|^2.
 \end{split}
\end{equation}  Then using that $g_{N}$ is linear in its second argument,
\begin{equation}\label{eq:g2}
\begin{split}
 g_{N}(U^{n-1}_N, U^n_N) \eta_N^n &= g_{N}(U^{n-1}_N, U^{n-1}_N) \eta_N^n
 +g_{N}(U^{n-1}_N, U^{n}_N -  U^{n-1}_N) \eta_N^n \\
 & \leq  g_{N}(U^{n-1}_N, U^{n-1}_N) \eta_N^n
 + \left |g_{N}(U^{n-1}_N, U^{n}_N -  U^{n-1}_N) \eta_N^n\right | \\
 & \leq (\mbox{thanks to (\ref{eq:g})})\\
 & \leq  g_{N}(U^{n-1}_N, U^{n-1}_N) \eta_N^n + \frac{1}{2} |U^{n}_N -  U^{n-1}_N|^2 +
            2|\sigma_{N}(U^{n-1}_N) \eta_N^n|^2.
            \end{split}
 \end{equation}
Using these observations for $g_N$, $\ell^{n}_{N}$ and $\SDvar$  we rearrange and infer that,
up to a set of measure zero,
\begin{align}
|U^n_N|^2 - |U^{n-1}_N|^2 + \frac{1}{2}|U^n_N &- U^{n-1}_N|^2
+  \deltaN c_1\|U^n_N\|^2
\notag\\
\leq& c_1^{-1} \zeta^n_N +
       4 \deltaN c_{4} (1+ |U^{n}_{N}|^{2})+
       2g_{N}(U^{n-1}_N, U^{n-1}_N) \eta_N^n +
      2|\sigma_{N}(U^{n-1}_N) \eta_N^n|^2.
      \label{eq:DiscreteInequality2}
\end{align}
Using \eqref{eq:SublinearCondH},
\eqref{eq:SigmaApproxCond2} and that $U^{n-1}_{N}$ is ${{\mathcal{G}}}_{n-1}$-measurable and in $L^{2}(\Omega;H)$ we have
that
\begin{gather*}
   \E g_{N}(U^{n-1}_N, U^{n-1}_N) \eta_N^n = 0,\quad
  \E|\sigma_{N}(U^{n-1}_N) \eta_N^n|^2 = \deltaN \E |\sigma_{N}(U^{n-1}_N)|^2_{L_2(\mathfrak{U}, H)}
    \leq
    2 \deltaN c_{3}^2 \E  (1 + |U^{n-1}_N|^2).
\end{gather*}
From this observation, \eqref{eq:DiscreteInequality2} and \eqref{eq:ExistenceStartingDiscretization} we infer
\begin{align*}
 \E \deltaN c_1\|U^n_N\|^2  \leq \E\left( (4 \deltaN c_{4}-1)|U^n_N|^2
   +   c( |U^{n-1}_N|^2  + \zeta^n_N + 1) \right)
     \leq c\E( |U^{n-1}_N|^2  + \zeta^n_N + 1),
\end{align*}
which implies that $U^{n}_{N} \in L^{2}(\Omega;V)$, as needed.

We have thus
established the iterative step in the construction of $\{U_{N}^{n}\}_{n=0}^{N}$.
The base case, $n =1$, is established in an identical fashion to the iterative steps.
The proof of Proposition~\ref{thm:ExistenceSelectionThm} is now
complete.

\begin{Rmk}\label{notthefiltrationyet}
{Although necessary for the establishment of the existence of the $U^n_N$'s   in Proposition \ref{thm:ExistenceSelectionThm}, it is not necessary
 to assume the underlying stochastic basis to be $\mathcal{S}_{\mathcal{G}}$ (defined in subsection~\ref{sec:specificfiltration})
 in the    results throughout  Section \ref{sec:UniformEstimates} to Section \ref{sec:CompArgs}. The reason is that
 these results   are   true whenever such $U^n_N$'s defined as in Definition \ref{def:AdminEulerSchemes} exist;  in other words they  are independent of
 the choice of the underlying stochastic basis. Similarly, it is not necessary at this point to assume that $U^0$ and $\ell$  have laws which coincide with those of the externally
 given  $\mu_{U^0}$ and $\mu_{\ell}$ for these results.}

{However, it is necessary that we resume these assumptions of $\mathcal{S}_{\mathcal{G}}$, $\mu_{U^0}$ and $\mu_{\ell}$    starting in Section \ref{sec:SBPassagetoTheLimit}.}
\end{Rmk}

\subsection{Uniform `Energy' Estimates for the $U^{n}_{N}$}
\label{sec:UniformEstimates}
Starting from \eqref{eq:DiscreteInequality1}
we next determine certain uniform bounds,
independent of $N$,
for (suitable) sequences $\{U_N^n\}_{n=1}^{N}$
satisfying \eqref{eq:EulerSemiFnRep}
as follows:
\begin{Prop}
\label{thm:UniformBnds}
Let
\begin{align}
  N_{1} := 12 T c_{5}, \textrm{ with } c_{5} :=8c_{4} + 80c_{3}^{2},
\label{eq:BndsStartingDiscretization}
\end{align}
where $c_{3}$ and $c_{4}$ are from \eqref{eq:SublinearCondH} and
\eqref{eq:SublinearConditionsOnSDTerm}.
   Let   ${\mathcal{S}= (\Omega,\mathcal{F}, \{\mathcal{F}_t\}_{t \geq 0}, \mathbb{P},
\{W^k\}_{k \geq 1})} $  be
the given stochastic basis and assume that $\ell = \ell(t)$$\in L^2(\Omega;L^{2}(0,T; V'))$
  is
 measurable with respect to  $\mathcal {F}_t$.
  For each $N \geq N_{1}$ we assume that  $U_N^0 \in L^2(\Omega, H)$,
 is ${{\mathcal{F}}}_0$ measurable
  and such that
\begin{align}
  \sup_{N \geq N_{1}} \E |U_N^0|^2 < \infty.
  \label{eq:UniformBndIntialDataForUniformEstimates}
\end{align}
 Then for each $N \geq N_{1}$, consider the sequences $\{U_N^n\}_{n =1}^{N} \subset L^{2}(\Omega; V)$
  which satisfy \eqref{eq:EulerSemiFnRep} starting from $U_N^0$ and relative to $\ell$ in the
  sense of Definition~\ref{def:AdminEulerSchemes}.
 Then
  \begin{align}
  	\sup_{N \geq N_{1}} \E
	\left( \max_{0 \leq l \leq N}  |U^l_N|^2
	 +\sum_{k = 1}^{N} \bigl( |U^k_N - U^{k-1}_N|^2
 +  \deltaN \|U^k_N\|^2 \bigr)  \right)< \infty.
\label{eq:UniformBndConclusion}
  \end{align}
\end{Prop}

\begin{proof}

The starting point for the estimates leading to \eqref{eq:UniformBndConclusion} is of course
\eqref{eq:DiscreteInequality1}
and from this inequality we can use the same proof as in Proposition~\ref{thm:ExistenceSelectionThm} to obtain
\eqref{eq:DiscreteInequality2}.
In order to make suitable estimates for the final
two terms in \eqref{eq:DiscreteInequality2} we need to take
advantage of some martingale structure in the terms involving $\sigma_{N}$.
For any $1 \leq m \leq n \leq N$ we define the stochastic processes
\begin{align}
  M^{m,n}_N &:= \sum_{k = m}^n  g_{N}(U^{k-1}_N, U^{k-1}_N) \eta_N^k,
  \quad
 Q^{m,n}_N  :=\sum_{k = m}^n  |\sigma_{N}(U^{k-1}_N) \eta_N^k|^2.
\label{eq:IncreasingProcess}
\end{align}
Summing \eqref{eq:DiscreteInequality2} for   $1 \leq m \leq n=k \leq l \leq N$
we find,
\begin{align}
  |U^l_N|^2
  + \sum_{k = m}^{l} ( \tfrac{1}{2} |U^k_N - U^{k-1}_N|^2
 &+  \deltaN c_1\|U^k_N\|^2 ) \notag\\
\leq& |U^{m-1}_N|^2 +    \sum_{k = m}^{l} (c_1^{-1} \zeta^k_N +  4 \deltaN c_{4} (1+ |U^{k}_{N}|^{2}))
     +2 M^{m,l}_N + 2 Q^{m,l}_N.
     \label{eq:DiscreteInequality3}
\end{align}
Since $\{U^{n}_{N}\}_{n=0}^{N} \subset L^{2}(\Omega;H)$ and is adapted to ${{\mathcal{{F}}}}_{n} :={{\mathcal{{F}}}}_{t^{n}}$,
it is easy to see that $\{M^{m,n}_{N}\}_{n= m}^{N}$ is a \emph{martingale}
relative to $\{\mathcal{{F}}_{n}\}_{n=m}^{N}$ with $M^{m,m}_{N} \equiv 0$.
We would like to apply a discrete version of the Burkholder-Davis-Gundy
inequality, recalled here as in
Lemma~\ref{thm:DiscreteBDG} to obtain estimates for
$\E \max_{ m \leq l \leq n} |M^{m,l}_N|$.  Unfortunately it is not
clear that $\{M^{m,n}_{N}\}_{n= m}^{N}$ is square integrable so
we have to apply a localization argument to make proper use of this
inequality.  For any $K > 0$ we define the stopping times
\begin{align*}
\tilde{n}_{K} = \min_{l \geq m} \{|U_{N}^{l-1}| \geq K \} \wedge N.
\end{align*}
Since $\{U^{n}_{N}\}_{n=0}^{N} \subset L^{2}(\Omega;H)$ we have that $\tilde{n}_{K} \uparrow N$
almost surely as $K \uparrow \infty$.  Clearly $\{M^{m,n\wedge \tilde{n}_{K}}_{N}\}_{n= m}^{N}$ is a square-integrable
martingale. For the moment let us recall  a discrete analogue of the Burkholder-Davis-Gundy Inequality.
This result  and other related  martingale inequalities can be found in e.g. \cite{Durrett2010}.
\begin{Lem}\label{thm:DiscreteBDG}
Assume that $\{M^{n}\}_{n \geq 0}$ is
a (discrete) martingale on a Hilbert space $\mathcal{H}$ (with
norm $| \cdot |$), relative to a given filtration
$\{ \mathcal{F}_{n}\}_{n \geq 0}$.
We assume, additionally that $M_{0} \equiv 0$
and that $\E |M_{n}|^{2} < \infty$, for all $n \geq 0$.
Then, for any $q \geq 1$ and any $n \geq 1$
\begin{align}
	\E \max_{ 1 \leq m \leq n} |M^{n}|^{q}
	 \leq c_{q} \E (A^{n})^{q/2},
	 \label{eq:BDGDiscreteInequality}
\end{align}
where $c_{q}$ is a universal positive constant
depending only on $q$\footnote{We may often determine $c_{q}$ in \eqref{eq:BDGDiscreteInequality} explicitly
and in particular we have that $c_{1} = 3$.} (which is independent of $n$
and $\{ M^{m}\}_{m\geq 0}$) and $A^{n}$ is the
\emph{quadratic variation} defined by
\begin{equation}\label{eq:QuadVar}
	A^{n} = \sum_{m= 1} ^{n} \E( |M^{m} - M^{m-1}|^{2} | \mathcal{F}_{m-1}).
\end{equation}
\end{Lem}
Hence with the observation that $\indFn{\tilde{n}_{K} \geq k}$ is ${{\mathcal{{F}}}}_{k-1}$-measurable
we compute the \emph{quadratic variation} of $\{M^{m,n\wedge \tilde{n}_{K}}_{N}\}_{n= m}^{N}$
 in view of \eqref{eq:QuadVar} as follows
\begin{align*}
  A^{m,n}_N =& \sum_{k = m}^n
  \E (|M^{m,k\wedge \tilde{n}_{K}}_{N} - M^{m,(k-1)\wedge \tilde{n}_{K}}_{N}|^{2}
  | {{\mathcal{{F}}}}_{k-1})
     = \sum_{k = m}^n \E (\indFn{\tilde{n}_{K} \geq k}|g_{N}(U^{k-1}_N, U^{k-1}_N) \eta_N^k|^2  | {{\mathcal{{F}}}}_{k-1})\\
     =& \sum_{k = m}^{\tilde{n}_{K}\wedge n} |g_{N}(U^{k-1}_N, U^{k-1}_N)|^2 \deltaN;
\end{align*}
Thus, by
Lemma~\ref{thm:DiscreteBDG}, \eqref{eq:SublinearCondH} and \eqref{eq:SigmaApproxCond2} we infer
\begin{align*}
   \E \max_{ m \leq l \leq n}& |M^{m,l\wedge \tilde{n}_{K}}_N|
    \leq 3  \E \left( \sum_{k = m}^{n \wedge \tilde{n}_{K}}  |g_{N}(U^{k-1}_N, U^{k-1}_N)|^2_{L_2}
                \deltaN \right)^{1/2}
    \leq 3  \E \left( \sum_{k = m}^n  |\sigma_{N}(U^{k-1}_N)|^2_{L_2(\mathfrak{U}, H)} | U^{k-1}_N|^2
                \deltaN \right)^{1/2}
                \notag\\
    \leq& 3c_3  \E \left( \sum_{k = m}^n 2(1 + |U^{k-1}_N|^2)| U^{k-1}_N|^2
                \deltaN \right)^{1/2}
   \leq \frac{1}{4}  \E \max_{ m \leq k \leq n} |U^{k-1}_N|^2
          +  18 c_3^2\E \sum_{k = m}^n  (1 + |U^{k-1}_N|^2) \deltaN .
\end{align*}
Hence, letting $K \uparrow \infty$, we have, by the monotone convergence theorem,
\begin{align}
   \E \max_{ m \leq l \leq n} |M^{m,l}_N|
   \leq \frac{1}{4}  \E \max_{ m \leq k \leq n} |U^{k-1}_N|^2
          + 18 c_3^2 \deltaN \E \sum_{k = m}^n  (1 + |U^{k-1}_N|^2).
          \label{eq:MG1}
\end{align}
On the other hand since $U^{n}_{N}$ is adapted to ${{\mathcal{{F}}}}_{n}$, given
the condition \eqref{eq:SublinearCondH} on $\sigma$ and \eqref{eq:SigmaApproxCond2} we infer
that
\begin{align}
   \E Q^{m,n}_N =
    \sum_{k = m}^n \E |\sigma_N(U^{k-1}_N)|^2_{L_2(\mathfrak{U}, H)} \deltaN
    \leq
    2c_{3}^2\deltaN \E \sum_{k = m}^n  (1 + |U^{k-1}_N|^2).
    \label{eq:ItoTypeInequality}
\end{align}

We now use \eqref{eq:MG1}, \eqref{eq:ItoTypeInequality} with
\eqref{eq:DiscreteInequality3} and infer that
\begin{align*}
  \E \max_{m \leq l \leq n} |U^l_N|^2
    \leq&  \E \left( 2|U^{m-1}_N|^2 + \sum_{k = m}^{n}    (c_1^{-1} \zeta^k_N +  4  c_{4} \deltaN(1+ |U^{k}_{N}|^{2}))
                + 2 \max_{m \leq l \leq n}  |M^{m,l}_N| + 2Q^{m,n}_N
              \right)\\
   \leq&  \E \left( 2|U^{m-1}_N|^2 +  \sum_{k = m}^{n}   (c_1^{-1} \zeta^k_N +  4  c_{4} \deltaN(1+ |U^{k}_{N}|^{2}))
         + 40 c_3^2 \deltaN \sum_{k = m}^n  (1 + |U^{k-1}_N|^2)
             \right)\\
             &\quad+ \frac{1}{2} \E \max_{ m \leq k \leq n} |U^{k-1}_N|^2.
\end{align*}
Rearranging we find that
\begin{align}
\E \max_{m \leq l \leq n}|U^l_N|^2
\leq&  \E \left( 2|U^{m-1}_N|^2 +   2 c_1^{-1}\sum_{k = m}^{n}   \zeta^k_N +
        c_5 \deltaN \E \sum_{k = m}^{n+1}  (1 + |U^{k-1}_N|^2)
          \right)
          \notag\\
\leq&  \E \left( 2 |U^{m-1}_N|^2 +   2c_1^{-1} \sum_{k = m}^{n}   \zeta^k_N
       + c_5 \deltaN (n -m+2) (1+ \E \max_{m \leq k \leq n+1} |U^{k-1}_N|^2)
          \right),
          \label{eq:DiscreteInequality4}
\end{align}
for the constant $c_5 = 8 c_{4} +80 c_3^2$ which in particular depends only on $c_3, c_{4}$.
Thus, subject to the condition:
\begin{equation}\label{eq:GronwallCondSmallTime}
  c_5 \deltaN  (n - m+2)\leq \frac{1}{2}, \quad  \textrm{ i.e.} \quad   \frac{ n -m+2}{N} \leq
  \frac{1}{2c_5 T},
\end{equation}
we have
\begin{align}
\E \max_{m \leq l \leq n}|U^l_N|^2
\leq c_{6} \E \left(  |U^{m-1}_N|^2 +    \sum_{k = m}^{n}\zeta^k_n
 + 1
        \right),
        \label{eq:DiscreteInequality5}
\end{align}
where $c_{6} = \max\{4  c_1^{-1},7\}$.
Thus, by iterating this inequality
and noting, cf. \eqref{eq:ExternalForcingEqnSum}, that $\sum_{k=1}^{N} \zeta^k_n = \|\ell\|_{L^{2}(0,T;V')}^{2}$,  we finally conclude that,
\begin{align}
  \E \max_{1 \leq l \leq N}|U^l_N|^2
  \leq c_{7} \E \left( |U_{N}^{0} |^2 +   \|\ell\|_{L^{2}(0,T;V')}^{2}
  + 1
     \right), \quad \textrm{ for all $N \geq N_1$.}
     \label{eq:FinalDiscreteInequalityPartOne}
\end{align}
Note carefully that, in view of \eqref{eq:GronwallCondSmallTime},
we need not iterate \eqref{eq:DiscreteInequality5} more
than, say, $\ulcorner 16c_5T \urcorner$ times to obtain \eqref{eq:FinalDiscreteInequalityPartOne}.\footnote{
 Indeed, for $N \geq N_{1}$, let $\mathfrak{N}(N)$ be the minimum number of iterations
of \eqref{eq:DiscreteInequality5}, subject to the constraint
\eqref{eq:GronwallCondSmallTime}, which are needed to establish
\eqref{eq:FinalDiscreteInequalityPartOne}.
Take $\mathfrak{F}(N)$ to be the `fraction of the time interval that can be covered
at each step', namely,
\begin{align*}
	\mathfrak{F}(N)
	:= \max_{n \in \mathbb{N}}
	\left\{ \frac{n}{N} :  n +2 \leq   \frac{N}{2c_{5} T} \right\}
	> \frac{1}{2c_{5}T} - \frac{3}{N} \geq \frac{1}{4c_{5}T},
\end{align*}
where the last inequality follows from the standing assumption \eqref{eq:BndsStartingDiscretization}.
Since $\mathfrak{N}(N) \mathfrak{F}(N) \leq 2$ we finally estimate:
\begin{align*}
	\mathfrak{N}(N) \leq \frac{2}{\mathfrak{F}(N)} \leq 16c_{5} T.
\end{align*}
Here $\ulcorner  p\urcorner=$ the smallest integer that is larger than or equal to $p$.
}
  As such we may take $c_{7} = (1 +c_{6})^{16c_{5}T}
  = (1+\max\{4c_{1}^{-1},7\})^{16T(8c_{4} + 80 c_{3}^{2})}$
which, crucially, is independent of $N$.

We now return to \eqref{eq:DiscreteInequality3}.
With \eqref{eq:ItoTypeInequality} we infer,
\begin{align}
  \E \sum_{k = 1}^{N} ( |U^k_N - U^{k-1}_N|^2
 +  2 c_1 \deltaN \|U^k_N\|^2 )
  \leq& \E \left( |U_{N}^{0}|^2 + \sum_{k = 1}^{N}  (c_1^{-1} \zeta^k_N +  4  c_{4}\deltaN (1+ |U^{k}_{N}|^{2}))
	+ 4 c_3^{2} \deltaN \sum_{k =1}^N (1 + |U^{k-1}_N|^2)
     \right)
     \notag\\
     \leq& c_{8}
     \E \left( |U_{N}^{0}|^2  + \max_{1 \leq l \leq N}|U^l_N|^2    +  \|\ell\|_{L^{2}(0,T; V')}^{2} + 1 \right),
     \label{eq:FinalDiscreteInequalityPartTwo}
\end{align}
where we can take, $c_{8} = \max\{1,c_{1}^{-1},4T(c_{3}^{2} + c_{4})\}$.
As such,  \eqref{eq:FinalDiscreteInequalityPartOne}
and \eqref{eq:FinalDiscreteInequalityPartTwo} with \eqref{eq:UniformBndIntialDataForUniformEstimates} imply
\eqref{eq:UniformBndConclusion},
completing the proof of Proposition~\ref{thm:UniformBnds}.
\end{proof}

\section{Continuous Time Approximations and  Uniform Bounds}
\label{sec:ContTimeProcesses}

In this section we detail how the sequences $\{U^{n}_{N}\}_{n=0}^{N}$
defined  in the sense of Definition~\ref{def:AdminEulerSchemes}
 may be used to define continuous time processes that approximate \eqref{eq:PEsAbstractFormulation}.
The details of establishing the compactness of the associated sequences of probability laws and of the passage
to the limit are given further on in Section~\ref{sec:TightnessAndLimit}.

We now fix sequences $\{U^{n}_{N}\}_{n=0}^{N}$ satisfying \eqref{eq:EulerSemiFnRep} in the sense of
Definition~\ref{def:AdminEulerSchemes}.  For $N \geq N_{1}$, with $N_{1}$ as in \eqref{eq:BndsStartingDiscretization}, let:
\begin{align}
U_{N}(t) &=
\begin{cases}
U^{0}_{N} &\textrm{ for } t \in [0, t^{1}],\\
U^{n}_{N}
	&\textrm{ for } t \in (t^{n}, t^{n+1}], \;
	n = 1, \ldots, N-1.
\end{cases} \label{eq:StepProcessesN}
\end{align}
Of course we do not have any time derivatives of the $U_N$'s
(even fractional in time) as are typically needed for compactness.
Furthermore we would like to be able to associate an approximate
stochastic equation for \eqref{eq:PEsAbstractFormulation} with these
$\{U_N^n\}_{n =0}^{N}$'s.
For these dual concerns we introduce further stochastic processes and consider:
\begin{align}
\bar{U}_{N}(t) &=
\begin{cases}
U^{0}_{N} &\textrm{ for } t \in [0, t^{1}]\\
U^{n-1}_{N} + \frac{U^{n}_{N} - U^{n-1}_{N}}{\deltaN} (t - t^{n})
	&\textrm{ for } t \in (t^{n}, t^{n+1}], \;
	n = 1,  \ldots, N-1.
\end{cases}
\label{eq:InterpProcessN}
\end{align}
\begin{Rmk}
The processes $U_{N}$ and $\bar{U}_{N}$
are slightly different than those typically used in the deterministic
case. See, e.g. \cite{Temam1}.
Actually, these processes are essentially
their deterministic analogues evaluated at time $t$
by their value at time $t -\deltaN$.  With this choice
we crucially obtain processes which are adapted to $\{\mathcal{{F}}_{t}\}_{t\geq0}$.
 Not surprisingly however the present definitions of
$U_{N}, \bar{U}_{N}$ leads to bothersome error terms
in \eqref{eq:InEQPlusError} below.  In turn these error terms
dictate the additional convergences in $\sigma$ and $U^0$
when we initially defined the discrete scheme \eqref{eq:EulerSemiFnRep};
cf. \eqref{eq:SigmaApproxCond1}--\eqref{eq:SigmaApproxCond3}
and Remark~\ref{rmk:SomeCrazyApproximation} above.    These error
terms also complicate compactness arguments further on in Section~\ref{sec:TightnessAndLimit}
and see Remark~\ref{rmk:errorsAreABitch}.
\end{Rmk}

The rest of this section is now devoted to proving the following desirable properties of
$U_{N}$ and $\bar{U}_{N}$:
\begin{Prop}\label{prop:BasicPropertiesOfConProcesses}
 Let  ${\mathcal{S}= (\Omega, \mathcal{F},\{\mathcal{F}_t\}_{t \geq 0}, \mathbb{P},
\{W^k\}_{k \geq 1})}$  be a  stochastic basis,
  and let $N_{1}$ be as in
  \eqref{eq:BndsStartingDiscretization} in
  Proposition~\ref{thm:UniformBnds}.
  Consider a sequence $\{U_N^0\}_{N \geq N_{1}}$   bounded in $L^2(\Omega, H)$ independly of $N$,
  with $U^0_N$ ${{\mathcal{F}}}_0$-measurable for each $N$ and such that
\begin{align}
	\E \left( (1+\|U_N^0\|^2) (1+ \|U_N^0\|^2_{V_{(2)}}) \right)
	 \leq c \deltaN^{-1} = c N,
	\label{eq:InitialDataBnd}
\end{align}
for a constant $c >0$, independent of $N$.\footnote{The constraint \eqref{eq:InitialDataBnd} is
necessary for \eqref{eq:UniformBoundAssertionContProcesses},\eqref{eq:GoodConvL2HErrors}.
This is not a serious restriction when we pass to the limit in Section~\ref{sec:TightnessAndLimit}; as we described above in
Remark~\ref{rmk:SomeCrazyApproximation}, for any given $U^{0} \in L^{2}(\Omega; H)$ we may obtain a sequence $U_{N}^{0}$
approximating $U^{0}$ which maintains \eqref{eq:InitialDataBnd}.}
  Suppose we also have defined a process  $\ell =\ell (t)\in L^2(\Omega; L^{2}(0,T; V'))$ adapted to $\{{\mathcal {F}_t}\}_{t\geq0}$.

For each $N \geq N_{1}$, we consider sequences $\{U_N^n\}_{n =1}^{N}$
  which satisfy \eqref{eq:EulerSemiFnRep} starting from $U_N^0$ in the
  sense of Definition~\ref{def:AdminEulerSchemes}.
Once these sequences $\{U_N^n\}_{n =0}^{N}$ exists, then
   we define the continuous time
  processes $\{U_{N}\}_{N \geq 1}$ and $\{\bar{U}_{N}\}_{N \geq 1}$
  according to \eqref{eq:StepProcessesN}
  and \eqref{eq:InterpProcessN} respectively.  Then,
  \begin{itemize}
  \item[(i)] for each $N \geq N_{1}$, $U_{N}$ and $\bar{U}_{N}$
  are $\{\mathcal{F}_{t}\}_{t\geq0}$-adapted and
  \begin{align}
  \{U_{N}\}_{N  \geq N_{1}} \textrm{ and } \{\bar{U}_{N}\}_{N \geq N_{1}}
  \textrm{ are
  bounded in }
  L^2(\Omega;L^{2}( 0, T ;V) \cap L^{\infty}( 0, T ;H)).
  \label{eq:UniformBoundAssertionContProcesses}
  \end{align}
  Moreover we have that
  \begin{align}
  \lim_{N \uparrow \infty} \E \int_{0}^{T} |U_{N} - \bar{U}_{N}|^{2} dt =0.
    \label{eq:ConvTogetherUUbar}
  \end{align}
  \item[(ii)] $U_{N}$ and $\bar{U}_{N}$ satisfy a.s. and  for every $t \geq 0$,
  \begin{align}
       	\bar{U}_{N}(t) = U^{0}_{N}  + \int_{0}^{t} (\mathcal{N}(U_{N}) + \ell_{N})ds
		      + \int_{0}^{t} \sigma_{N}(U_{N}) dW + \mathcal{E}_{N}^{D}(t) + \mathcal{E}_N^{S}(t),
  \label{eq:InEQPlusError}
  \end{align}
  subject to error terms $\mathcal{E}_{N}^{D}(t) \in L^2(\Omega;L^2( 0, T ;V'))$, $\mathcal{E}_N^{S}(t)
  \in L^2(\Omega;L^2( 0, T ;H))$ which are defined explicitly
  in \eqref{eq:ErrorTerm1}, \eqref{eq:ErrorTerm2} below. 
  \item[(iii)] These error terms  $\mathcal{E}_{N}^{D}(t)$, $\mathcal{E}_N^{S}(t)$
  satisfy
  \begin{align}
    \lim_{N \uparrow \infty} \E \|\mathcal{E}_{N}^{D}\|_{L^{2}( 0, T ;V')}^{2} =0,
    \label{eq:GoodConvL2HErrors}\\
    \lim_{N \uparrow \infty} \E \|\mathcal{E}_N^{S}\|_{L^{2}( 0, T ;H)}^{2} = 0,
    \label{eq:GoodConvL2HErrorsStochastic}
  \end{align}
  and moreover
  \begin{align}
    \sup_{N \geq N_{1}} \E \| \mathcal{E}_{N}^{S}\|_{L^{\infty}( 0, T ;H) \cap L^{2}( 0, T ;V)}^{2} < \infty.
  \label{eq:UniformBndEStoch}
  \end{align}
  \end{itemize}
\end{Prop}

We proceed to prove Proposition~\ref{prop:BasicPropertiesOfConProcesses} in a series of subsections below.
The proof of (i) is essentially a direct application of Proposition~\ref{thm:UniformBnds} and we provide the details
in the subsection immediately following.  In Subsection~\ref{sec:ApproxEvolution} we provide the details of the derivation
of \eqref{eq:InEQPlusError} and in particular explain the origin of the error terms $\mathcal{E}_{N}^{D},\mathcal{E}_{N}^{S}$.  The
final Subsection~\ref{sec:ErrorTermEstimates} provides details of the estimates for these error terms which lead to \eqref{eq:GoodConvL2HErrors}--\eqref{eq:UniformBndEStoch}.

\begin{Rmk}\label{rmk:errorsAreABitch}
It is not straightforward  to obtain  fractional in time estimates for
$\bar{U}_{N}$ from \eqref{eq:InEQPlusError} in view of the error terms
which have a rather complicated structure (see \eqref{eq:ErrorTerm1}, \eqref{eq:ErrorTerm2} below).
As such, we can not
establish sufficient compactness for the sequence $\bar{U}_{N}$ directly to facilitate the
passage to the limit.    For this reason we   choose  to introduce additional continuous time
processes in Section~\ref{sec:TightnessAndLimit} below.
An alternate approach will be presented later on in the related work \cite{GlattTemam2011}.
\end{Rmk}

\subsection{Uniform Bounds and Clustering}
\label{sec:UniformBndsConProcesses}
It is clear from \eqref{eq:StepProcessesN} that $U_{N}$ is $\{\mathcal{F}_{t}\}_{t\geq0}$-adapted and
that
\begin{align*}
    \E \left( \sup_{t \in [0,T]} |U_{N}|^2 + \int_0^T \|U_{N}\|^2 dt \right)
  =  \E \left( \max_{0 \leq m \leq N-1}  |U_{N}^{m}|^{2}  +   \sum_{m =0}^{N-1}  \deltaN  \|U_{N}^{m}\|^{2}\right).
\end{align*}
Thus, since \eqref{eq:UniformBndIntialDataForUniformEstimates} holds we have the uniform bound \eqref{eq:UniformBndConclusion} from Proposition~\ref{thm:UniformBnds} and we immediately infer that
\begin{align}
  \sup_{N \geq N_1} \E \left( \sup_{t \in [0,T]} |U_{N}|^2 + \int_0^T \|U_{N}\|^2 dt \right)
 	< \infty,
	\label{eq:UniformBoundUProcess}
\end{align}
with $N_1$ the integer appearing in \eqref{eq:BndsStartingDiscretization}.

As with the $U_{N}$ above, it is easy to see  from \eqref{eq:InterpProcessN}   that
$\bar{U}_{N}$ is adapted to $\{\mathcal{F}_{t}\}_{t\geq0}$
and  that $\{U_N^n\}_{n =1}^{N}$ is adapted to ${{\mathcal{F}}}_n (= {{\mathcal{F}}}_{t^n})$.
Furthermore, direct calculations show that:
\begin{align}
U_{N} - \bar{U}_{N}(t) &=
\begin{cases}
0  &\textrm{ for } t \in [0, t^{1}],\\
  \frac{U^{n}_{N} - U^{n-1}_{N}}{\deltaN} (t^{n+1} - t)
	&\textrm{ for } t \in (t^{n}, t^{n+1}], \;
	n = 1,  \ldots, N-1.
\end{cases}
\label{eq:DiffProcess}
\end{align}
Using \eqref{eq:DiffProcess} we compute, similarly to
e.g. \cite{Temam1}, that
\begin{align*}
	\E \int_{0}^{T} |U_{N} - \bar{U}_{N}|^{2} dt =
	\sum_{n = 1}^{N-1} \E \left|
	U^{n}_{N} - U^{n-1}_{N}\right|^2  \int_{t^n}^{t^{n+1}} \left(\frac{t^{n+1} - t}{\deltaN}\right)^2 dt
	= \frac{\deltaN}{3} \E \sum_{n =1}^{N} |U^{n}_{N} - U^{n-1}_{N}|^{2}.
\end{align*}
We thus infer \eqref{eq:ConvTogetherUUbar} directly from this observation and  \eqref{eq:UniformBndConclusion}.
Based on similar considerations we also have
\begin{align*}
\E \int_{0}^{T} \| \bar{U}_{N} \|^{2}  \leq c\deltaN \E \sum_{n =0}^{N} \|U^{n}_N\|^{2} = c\deltaN \E  \|U^{0}_N\|^{2} + \deltaN \E \sum_{n =1}^{N} \|U^{n}_N\|^{2}.
\end{align*}
Thus, once again due to \eqref{eq:InitialDataBnd} and \eqref{eq:UniformBndConclusion}, we finally have
\begin{align}
  \sup_{N \geq N_{1}} \E \left( \sup_{t \in [0,T]} |\bar{U}_N|^2 + \int_0^T \|\bar{U}_N\|^2 dt \right)
 	< \infty.
	\label{eq:UniformBoundsOnbarUProcess}
\end{align}
With \eqref{eq:UniformBoundUProcess} and \eqref{eq:UniformBoundsOnbarUProcess}
we have now established the first item in Proposition~\ref{prop:BasicPropertiesOfConProcesses}.

\subsection{The Approximate Stochastic Evolution Systems}
\label{sec:ApproxEvolution}
We next derive the equation \eqref{eq:InEQPlusError} relating
$U_{N}$ and $\bar{U}_N$ giving explicit expressions for
$\mathcal{E}_N^D$, $\mathcal{E}_N^S$.
We observe that, almost surely and for almost every $t \geq 0$ (in fact for every $t \not \in \{ t_{0}, t_{1},  \ldots, t_{N} \} $)
\begin{align}
  \frac{d}{dt} \bar{U}_{N}(t) = \sum_{n=1}^{N -1} \frac{U^{n}_{N} - U^{n-1}_{N}}{\deltaN} \chi_{(t^{n}, t^{n+1})}(t),
  \label{eq:barUDerivative}
\end{align}
where $\chi (t_1, t_2)$ denotes the indicator function of $(t_1, t_2)$.
 Recall that $\eta^{n}_{N} = W(t^{n}) - W(t^{n-1})$ and let $N^{t}_{*} := \min \{ n : t^{n} \geq t \}$
in other words we take $N^{t}_{*}$ such that
$$
	N_{*}^{t} \deltaN \leq t < (N_{*}^{t} + 1) \deltaN.
$$
Working from \eqref{eq:barUDerivative} and \eqref{eq:EulerSemiFnRep} we therefore compute
\begin{align}
  \bar{U}_{N}(t) &= U^{0}_{N} + \int_{0}^{t} \sum_{n=1}^{N-1} \frac{U^{n}_{N} - U^{n-1}_{N}}{\deltaN} \chi_{(t^{n}, t^{n+1})}(s) ds
  \notag\\
  	&= U^{0}_{N}  + \int_{0}^{t} \sum_{n=1}^{N-1}  (\mathcal{N}(U^{n}_{N}) + \ell^n_N)\chi_{(t^{n}, t^{n+1})}(s) ds
		      + \int_{0}^{t} \sum_{n=1}^{N-1} \sigma_N(U^{n-1}_{N}) \frac{\eta^{n}_{N}}{\deltaN} \chi_{(t_{n}, t_{n+1})}(s) ds
		      \notag\\
        	&= U^{0}_{N}  + \int_{0}^{t} (\mathcal{N}(U_{N}) + \ell_{N})ds
		      + \int_{0}^{t} \sigma_N(U_{N}) dW + \mathcal{E}_{N}^{D}(t) + \mathcal{E}_N^{S}(t),
		      \label{eq:ApproxEvolutionComputation}
\end{align}
where the `error terms',  $\mathcal{E}^{D}_{N}(t)$ and $\mathcal{E}^{S}_{N}(t)$, are defined as:
\begin{equation}\label{eq:ErrorTerm1}
\begin{split}
  \mathcal{E}^{D}_{N}(t) := - \mathcal{N}(U^{0}_N) \deltaN \wedge t - \left(\int_{t^{N_*^t-1}}^t \ell_{N} ds +  \ell^{N_*^t - 1}_N (t^{N_*^t} - t)
  \chi_{t > t^{1}} \right)
  = \mathcal{E}^{D,1}_N(t) + \mathcal{E}^{D,2}_N(t),
\end{split}
\end{equation}
and
\begin{equation}\label{eq:ErrorTerm2}
  \mathcal{E}^{S}_{N}(t) :=
  -  \sigma_N(U^{N_{*}^{t}-2}_{N}) \frac{\eta^{N_{*}^{t} -1}_{N}}{\deltaN} (t^{N_{*}^{t}} - t) \chi_{t > t^{1}}
  - \int_{t^{N_{*}^{t}-1}}^{t} \sigma_N(U_{N}) dW
  := \mathcal{E}^{S,1}_N(t) +  \mathcal{E}^{S,2}_N(t).
\end{equation}
To understand the origin of these error terms we observe that
\begin{align*}
\int_{0}^{t} \sum_{n=1}^{N-1}  \mathcal{N}(U^{n}_{N}) \chi_{(t^{n}, t^{n+1})}(s) ds
   &= \int_{0}^{t} \sum_{n=0}^{N-1}  \mathcal{N}(U^{n}_{N})\chi_{(t^{n}, t^{n+1})}(s) ds
   	- \mathcal{N}(U^{0}_N) \deltaN \wedge t\\
  &= \int_{0}^{t} \mathcal{N}(U_{N}) ds
         + \mathcal{E}^{D,1}_{N}(t).
\end{align*}
Moreover, using the definition of the $\ell^{n}_{N}$'s in \eqref{eq:EllIncrement}, we have
\begin{align*}
 \int_{0}^{t} \sum_{n=1}^{N-1} \ell^n_N \chi_{(t^{n}, t^{n+1})}(s) ds
 =&  \int_{0}^{t^{N_*^t}} \sum_{n=1}^{N-1} \ell^n_N \chi_{(t^{n}, t^{n+1})}(s) ds
  + \left (\int_{t^{N_*^t}}^{t} \ell^{N_*^t - 1}_N ds \right )\chi_{t > t^{1}}\\
 =&  \sum_{n=1}^{N_*^t-1} \ell^n_N \deltaN + \ell^{N_*^t - 1}_N ( t -t^{N_*^t}  )\chi_{t > t^{1}}
 =  \int_0^{t^{N_*^t-1}} \ell ds +  \ell^{N_*^t - 1}_N ( t -t^{N_*^t} )\chi_{t > t^{1}}\\
 =&\int_0^{t} \ell ds- \int_{t^{N_*^t-1}}^t \ell ds + \ell^{N_*^t - 1}_N (t -t^{N_*^t}  )\chi_{t > t^{1}}.\\
 \end{align*}
On the other hand for the error terms $\mathcal{E}^{S}_{N}(t)$ involving $\sigma_N$ in \eqref{eq:ErrorTerm2},
we compute,
\begin{align*}
	\int_{0}^{t} \sum_{n=1}^{N-1} \sigma_N(U^{n-1}_{N}) \frac{\eta^{n}_{N}}{\deltaN} \chi_{(t^{n}, t^{n+1})}(s) ds
    &= \int_{0}^{t^{N_{*}^{t}}} \sum_{n=1}^{N_{*}^{t}-1} \! \! \sigma_N(U^{n-1}_{N}) \frac{\eta^{n}_{N}}{\deltaN} \chi_{(t^{n}, t^{n+1})}(s) ds
         -  \! \! \int_{t}^{t^{N_{*}^{t}}} \! \! \! \! \! \sigma_N(U^{N_{*}^{t}-2}_{N}) \frac{\eta^{N_{*}^{t} -1}_{N}}{\deltaN} ds\chi_{t > t^{1}}\\
    &= \sum_{n=1}^{N_{*}^{t}-1} \sigma_N(U^{n-1}_{N}) \eta^{n}_{N}
         -  \sigma_N(U^{N_{*}^{t}-2}_{N}) \frac{\eta^{N_{*}^{t} -1}_{N}}{\deltaN} (t^{N_{*}^{t}} - t)\chi_{t > t^{1}}\\
    &= \int_{0}^{t^{N_{*}^{t}-1}} \sigma_N(U_{N}) dW
         -  \sigma_N(U^{N_{*}^{t}-2}_{N}) \frac{\eta^{N_{*}^{t} -1}_{N}}{\deltaN} (t^{N_{*}^{t}} - t)\chi_{t > t^{1}}\\
    &= \int_{0}^{t} \sigma_N(U_{N}) dW + \mathcal{E}^{S}_{N}(t).
\end{align*}

\subsection{The Estimates for the Error Terms}
\label{sec:ErrorTermEstimates}

We next proceed to make estimates on the error terms $\mathcal{E}^D_N$ and $\mathcal{E}^S_N$
as desired in \eqref{eq:GoodConvL2HErrors}, \eqref{eq:UniformBndEStoch}.
Perusing \eqref{eq:ErrorTerm1} we begin with estimates for $\mathcal{E}^{D,1}_N$.
Invoking the bounds provided by \eqref{eq:BSizeHNorm} along with the
continuity properties of the other operators making up $\mathcal{N}$
in \eqref{eq:TheDriftPartTheWholePotatoe} defined in Section~\ref{sec:BasicOperators}
we have:
\begin{align*}
	\E \sup_{t \in [0,T]} \| \mathcal{E}^{D,1}_N(t)\|_{V'}^2
	\leq \deltaN^{2} \E \|\mathcal{N}(U_N^0)\|_{V'}^2
	\leq c\deltaN^{2} \E\left( (1 + \|U_N^0\|^2) (1 +\| U_N^0\|^2_{V_{(2)}})\right).
\end{align*}
As such, in view of the standing condition \eqref{eq:InitialDataBnd} (cf. Remark~\ref{rmk:SomeCrazyApproximation})
we conclude that
\begin{align}
	\lim_{N \uparrow \infty}\E \|\mathcal{E}^{D,1}_N\|_{L^2( 0, T ; V')}^{2}
	=\lim_{N \uparrow \infty}\E \|\mathcal{E}^{D,1}_N\|_{L^\infty( 0, T ; V')}^{2} =0.
	\label{eq:DetErrorDecalLinfH}
\end{align}
For $\mathcal{E}^{D,2}_N$ we estimate in $L^2( 0, T ; V')$
\begin{align*}
  \int_0^T\left\| \int_{t^{N_*^t-1}}^t \ell ds \right\|_{V'}^2 dt
\leq& \int_0^T \int_{t^{N_*^t-1}}^t \left\| \ell  \right\|_{V'}^2 ds (t- t^{N_*^t-1}) dt
= \sum_{k =1}^{N-1} \int_{t^{k-1}}^{t^k} \int_{t^{k-1}}^t \|\ell \|_{V'}^2 ds (t - t^{k-1})dt\\
\leq& c \deltaN^2  \int_{0}^{T} \| \ell \|_{V'}^{2} dt,
\end{align*}
and
\begin{align*}
  \int_0^T\left\|\ell^{N_*^t - 1}_N (t^{N_*^t} - t) \chi_{t > t^{1}}  \right\|_{V'}^2 dt
  =& \sum_{k =1}^{N-1}   \left\| \ell^{k}_N \right\|_{V'}^2  \int_{t^{k}}^{t^{k+1}}(t^{k+1} - t)^2   dt\\
  \leq& \frac{\deltaN}{3} \sum_{k =1}^N \left\| \int_{t^{k}}^{t^{k+1}}   \ell ds \right\|_{V'}^{2}
  \leq \frac{\deltaN^{2}}{3} \int_{0}^{T} \| \ell \|_{V'}^{2} dt.
\end{align*}
In summary we have
\begin{align}
	\lim_{N \uparrow \infty}\E \|\mathcal{E}^{D,2}_N\|_{L^2( 0, T ; V')}^{2} =0
	\label{eq:DetErrorAuxDecalL2H}
\end{align}
and so we conclude \eqref{eq:GoodConvL2HErrors} from \eqref{eq:DetErrorDecalLinfH} and \eqref{eq:DetErrorAuxDecalL2H}.

We next turn to make estimates for $\mathcal{E}^S_N$.    We begin with
estimates in  $L^2( 0, T ; H)$.  For $\mathcal{E}^{S,1}_N$ we observe with \eqref{eq:SublinearCondH}
and \eqref{eq:SigmaApproxCond2}
(cf. \eqref{eq:ItoTypeInequality}) that
\begin{align*}
  \E \int_0^T |\mathcal{E}^{S,1}_N|^2 dt
  =& \sum_{k =1}^{N-1} \E  \left|  \sigma_N(U^{k-1}_{N}) \frac{\eta^{k}_{N}}{\deltaN} \right|^2 \int_{t^k}^{t^{k+1}} (t^{k+1} - t)^2 dt
  =  \frac{\deltaN}{3} \sum_{k =1}^{N-1}\E  \left|  \sigma_N(U^{k-1}_{N}) \eta^{k}_{N} \right|^2 \notag\\
  =&  \frac{\deltaN}{3} \sum_{k =1}^{N-1}\E  \left|  \sigma_N(U^{k-1}_{N}) \right|^2_{L^{2}(\mathfrak{U}, H)} \deltaN
  \leq c\deltaN \sum_{k =1}^{N-1}\E  (1+ \left| U^{k-1}_{N} \right|^2) \deltaN,
\end{align*}
and infer from \eqref{eq:UniformBndConclusion} in Proposition~\ref{thm:UniformBnds} that
\begin{align}
\lim_{N \uparrow \infty}\E \|\mathcal{E}^{S,1}_N\|_{L^2( 0, T ; H)}^{2} =0.
	\label{eq:StocError1L2H}
\end{align}
On the other hand, with the It\=o isometry and another application of \eqref{eq:SublinearCondH} and \eqref{eq:SigmaApproxCond2}
we have
\begin{align*}
   \E \int_0^T |\mathcal{E}^{S,2}_N|^2 dt
  =& \sum_{k =1}^{N-1} \E  \int_{t^k}^{t^{k+1}} \left|  \int_{t^{k}}^{t} \sigma_N(U_{N}) dW \right|^2 dt
  =\sum_{k =1}^{N-1} \E  \int_{t^k}^{t^{k+1}}  \int_{t^{k}}^{t}  \left|\sigma_N(U_{N}) \right|^2_{L^{2}(\mathfrak{U}, H)} ds dt
  \notag\\
  =&\sum_{k =1}^{N-1} \E   \left|\sigma_N(U_{N}^{k-1}) \right|^2_{L^{2}(\mathfrak{U}, H)}  \int_{t^k}^{t^{k+1}} (t - t_k) dt
  \leq c\deltaN \sum_{k =1}^{N-1} \E   (1+ \left|U_{N}^{k-1}\right|^2) \deltaN,
\end{align*}
so that
\begin{align}
\lim_{N \uparrow \infty}\E \|\mathcal{E}^{S,2}_N\|_{L^2( 0, T ; H)}^{2} =0.
	\label{eq:StocError2L2H}
\end{align}
By combining now  \eqref{eq:StocError1L2H} and \eqref{eq:StocError2L2H}
we obtain \eqref{eq:GoodConvL2HErrorsStochastic}.

We turn now to establishing the
uniform bounds announced in \eqref{eq:UniformBndEStoch}.
Estimates similar to those leading to \eqref{eq:StocError1L2H}, \eqref{eq:StocError2L2H} but
which instead make use of the condition \eqref{eq:SigmaApproxCond1} yield bounds in $L^2( 0, T ;V)$ namely,
\begin{align*}
  \E \int_0^T \|\mathcal{E}^{S,1}_N\|^2 dt
  =\frac{\deltaN}{3} \sum_{k =1}^{N-1}\E  \left\|  \sigma_{N}(U^{k-1}_{N}) \right\|^2_{L^{2}(\mathfrak{U}, V)} \deltaN
  \leq \frac{T}{3} \sum_{k =1}^{N-1}\E  \left|  \sigma(U^{k-1}_{N}) \right|^2_{L^{2}(\mathfrak{U}, V)} \deltaN
  \leq c \sum_{k =1}^{N-1}\E  (1+ \left| U^{k-1}_{N} \right|^2) \deltaN,
\end{align*}
and similarly
\begin{align*}
   \E \int_0^T \|\mathcal{E}^{S,2}_N\|^2 dt
  =&\sum_{k =1}^{N-1} \E   \left\|\sigma_{N}(U_{N}^{k-1}) \right\|^2_{L^{2}(\mathfrak{U}, V)}  \int_{t^k}^{t^{k+1}} (t - t_k) dt
  \leq c\sum_{k =1}^{N-1} \E   (1+ \left|U_{N}^{k-1}\right|^2) \deltaN,
\end{align*}
so that, taken together we infer that:
\begin{align}
	\sup_{N \geq N_{1}}\E \|\mathcal{E}^{S}_N\|_{L^2( 0, T ; V)} < \infty.
	\label{eq:StocErrorL2V}
\end{align}
Finally we supply a bound for $\mathcal{E}^{S}_N$ in $L^\infty( 0, T ;H)$.  For
$\mathcal{E}^{S,1}_N$ we observe with \eqref{eq:SublinearCondH}, \eqref{eq:SigmaApproxCond2} that
\begin{align*}
  \E \sup_{t \in [0,T]} |\mathcal{E}^{S,1}_N|^2
  \leq \sum_{k=1}^{N-1}  \E \sup_{t \in [t^k,t^{k+1}]}  |\mathcal{E}^{S,1}_N|^2
  \leq \sum_{k=1}^{N-1}  \E  |\sigma_N(U^{k-1}_{N}) \eta^{k}|^2
  \leq  c\sum_{k=1}^{N-1}  \E  (1 +|U^{k-1}_{N}|^2) \deltaN.
\end{align*}
To estimate $\mathcal{E}^{S,2}_N$ we use Doob's inequality and \eqref{eq:SublinearCondH} to infer
\begin{align*}
  \E \sup_{t \in [0,T]} |\mathcal{E}^{S,2}_N|^2
  \leq& \sum_{k=1}^{N-1}  \E \sup_{t \in [t^k,t^{k+1}]}  |\mathcal{E}^{S,2}_N|^2
  = \sum_{k=1}^{N-1}  \E \sup_{t \in [t^k,t^{k+1}]}  \left|\int_{t^{k}}^{t} \sigma_N(U_{N}) dW\right|^2
  \notag\\
  \leq&  \sum_{k=1}^{N-1}  \E \int_{t^{k}}^{t^{k+1}} \left|\sigma_N(U_{N}) \right|^2_{L^{2}(\mathfrak{U}, H)} ds
  \leq c \sum_{k=1}^{N-1}  \E  (1 +|U^{k-1}_{N}|^2) \deltaN.
\end{align*}
With these bounds and \eqref{eq:UniformBndConclusion} we conclude that
\begin{align}
	\sup_{N \geq N_{1}}\E \|\mathcal{E}^{S}_N\|_{L^\infty( 0, T ; H)}^{2} < \infty.
	\label{eq:StocErrorLinfH}
\end{align}
In turn, \eqref{eq:StocErrorL2V}, \eqref{eq:StocErrorLinfH} directly imply \eqref{eq:UniformBndEStoch}
and so the proof of Proposition~\ref{prop:BasicPropertiesOfConProcesses} is now complete.

\section{Compactness and The Passage to the Limit}
\label{sec:TightnessAndLimit}

In this section we detail the compactness arguments that
we use to prove the existence of Martingale solutions of
\eqref{eq:PEsAbstractFormulation} using the processes
$U_{N}$ and $\bar{U}_{N}$ defined in the previous section.  As it is not
clear how to obtain compactness directly from $\bar{U}_{N}$,
(cf. Remark~\ref{rmk:errorsAreABitch}) we must introduce further
processes  to achieve this end.

Recalling \eqref{eq:StepProcessesN}, \eqref{eq:InterpProcessN},
\eqref{eq:ErrorTerm1}, \eqref{eq:ErrorTerm2} we define
\begin{align}
	U^*_{N} = \bar{U}_{N} - \mathcal{E}^S_{N},\quad
	U^{**}_{N} =U^{*}_{N} - \mathcal{E}^{D}_{N},
	\label{eq:additionalStaredProcesses}
\end{align}
and then consider the associated probability measures
\begin{align}
	\mu_{N}(\cdot) := \Prb( U_{N} \in \cdot),
	\quad
	\mu^*_{N}(\cdot) := \Prb( U^*_{N} \in \cdot),
	\quad
	\mu^{**}_{N}(\cdot) := \Prb( U^{**}_{N} \in \cdot).
	\label{eq:BorelMeasuresForUUstarUstarstar}
\end{align}
Notice that, due to Proposition~\ref{prop:BasicPropertiesOfConProcesses},
$\mu_N, \mu_N^*$ are defined on the space $\mathcal{X} := L^{2}( 0, T ;H)$.
Regarding the elements $\mu^{**}_N$ we observe that,
as a consequence of \eqref{eq:InEQPlusError}
\begin{align}
  U^{**}_{N}(t) = U_{N}^{0} + \int_{0}^{t} (\mathcal{N}(U_{N}) + \ell) dt + \int_{0}^{t} \sigma_{N}(U_{N}) dW.
  \label{eq:ustaruRelationship}
\end{align}
As a result of this identity and Proposition~\ref{prop:BasicPropertiesOfConProcesses}, the elements
$\mu_{N}^{**}$ may be regarded as  measures on the
space $\mathcal{Y} := L^{2}( 0, T ;V') \cap \mathcal C([0,T]; V_{(3)}')$.

We will show below that  $\mu_{N}$ and $\mu_{N}^{**}$ converge weakly to a common measure
$\mu$ and then make careful usage of the Skorohod embedding theorem to pass to the limit in
\eqref{eq:ustaruRelationship} on a new stochastic basis.  The former compactness arguments, which
rely on the intermediate measures $\mu^{*}_{N}$, will be carried out in the next subsection and the
details of the Skorohod embedding will be discussed in Subsection~\ref{sec:SBPassagetoTheLimit}
further on.

\subsection{Tightness Arguments}
\label{sec:CompArgs}

In this section we will establish the following compactness properties of the $\{\mu_{N}\}_{N \geq N_{1}}$
and $\{\mu_{N}^{**}\}_{N \geq N_{1}}$
\begin{Prop}\label{prop:Tightness}
The assumptions are precisely those in Proposition~\ref{prop:BasicPropertiesOfConProcesses}.
Define $\{U_{N}\}_{N \geq N_{1}}$ and $\{U_{N}^{**}\}_{N \geq N_{1}}$ according to \eqref{eq:StepProcessesN}
and \eqref{eq:additionalStaredProcesses} and where $N_{1}$ is as in \eqref{eq:BndsStartingDiscretization}.
Let $\{\mu_{N}\}_{N \geq N_{1}}$, $\{\mu_{N}^{**}\}_{N\geq N_{1}}$
be the associated Borel measures on
$$
 \mathcal{X} := L^{2}( 0, T ;H), \quad  \mathcal{Y} := L^{2}( 0, T ;V') \cap \mathcal C([0,T]; V_{(3)}'),
$$
defined according to \eqref{eq:BorelMeasuresForUUstarUstarstar}.
Then, there exists a Borel measure $\mu$ on $L^{2}( 0, T ;H) \cap \mathcal C([0,T]; V_{(3)}')$
such that, up to a subsequence\footnote{We recall the notion of weak compactness
of probability measures along with the equivalent notion of \emph{tightness}
in the Appendix, Section~\ref{sec:LetsConvergeToaMeasure}
below.}
\begin{align}
	\mu_{N} \rightharpoonup \mu, \quad\textrm{(weakly) on } \mathcal{X},
	\label{eq:MuNConvL2H}
\end{align}
and
\begin{align}
	\mu_{N}^{**} \rightharpoonup  \mu, \quad \textrm{(weakly) on } \mathcal{Y}.
	\label{eq:MuNstarstarConvCV3}
\end{align}
\end{Prop}

The rest of this subsection is devoted to the proof of Proposition~\ref{prop:Tightness}.  We
proceed as follows: First we show that $\{\mu_{N}^{*}\}_{N \geq N_{1}}$ is tight (cf. Appendix~\ref{sec:LetsConvergeToaMeasure}) in
$L^{2}( 0, T ;H)$ by employing a suitable variant of the Aubin-Lions compactness theorem
which we establish in Proposition~\ref{thm:MotherOfAllCompEmbeddings} below.
We next show that $\{\mu_{N}^{**}\}_{N \geq N_{1}}$ is tight in
$ \mathcal C([0,T];V_{(3)}')$ via an Arzel\'a-Ascoli type compact embedding from \cite{FlandoliGatarek1} and \cite{Temam4}.
We finally employ the estimates
\eqref{eq:ConvTogetherUUbar}, \eqref{eq:GoodConvL2HErrors}
along with the general convergence results recalled in Lemma~\ref{thm:ConvTogether}
to finally infer \eqref{eq:MuNConvL2H} and \eqref{eq:MuNstarstarConvCV3}.

\subsubsection{Tightness for $\mu_{N}^{*}$ in $L^{2}(0,T; H)$}

With the aid of Proposition~\ref{thm:MotherOfAllCompEmbeddings}
we identify some compact subsets of $\mathcal{X} = L^{2}( 0, T ;H)$ that, in conjunction with
suitable estimates (see \eqref{eq:SplitOfTDTypeEstimateIntoDetandStochParts}--\eqref{eq:StocPortDerivatPreEst}
immediately below) are used to establish the tightness of $\{ \mu^*_N\}_{N \geq N_{1}}$
in $\mathcal{X}$.
For $U \in \mathcal{X}$, $n > 0$, define
\begin{align}
[U]_{j} :=
     \left( j \sup_{ 0 \leq \theta \leq j^{-{6}}} \int_{0}^{T- \theta} \|U(t+\theta) - U(t) \|_{V_{(2)}'}^{4/3} dt \right)^{3/4},
     \label{eq:CompSpSemiNorm}
\end{align}
and, for each $R > 0$, consider
\begin{align}
  B_R := \left\{ U \in \mathcal{X}:  \|U\|_{L^2( 0, T ,V)}
	+\|U\|_{L^\infty ( 0, T , H)}
	+  \sup_{j \geq 1} [U]_{j} \leq R \right\}.
\end{align}
It is not hard to show that each set $B_R$ is a closed subset of $\mathcal{X}$.
\homework{
Suppose that $U_N \rightarrow U$ in $\mathcal{X} = L^2(0,T;H)$ with
$U_N \in B_R$.   As in, e.g. \cite[Section 7.1]{DebusscheGlattTemam1}
we have that
\begin{align*}
    \|U\|_{L^2( 0, T ,V)}  +\|U\|_{L^\infty ( 0, T , H)}  \leq R.
\end{align*}
Take
\begin{align*}
\| \cdot \|_n := \|\cdot\|_{L^2( 0, T ,V)}  +\|\cdot\|_{L^\infty ( 0, T , H)} + [\cdot]_n
\end{align*}
We would like to show that, for each $\epsilon > 0$,   $\|U\|_n \leq R + \epsilon$
for all $n$ so that $U \in B_R$.  Observe that
\begin{align*}
[U- U_N]_n  \leq 2 n \|U- U_N\|_{L^{4/3}( 0, T ;V_{(2)}')} \leq c(n, T) \|U- U_N\|_{L^{2}( 0, T ;H)}
\end{align*}
Thus, for every $N$ sufficiently large,
\begin{align*}
  \| U \|_n \leq \|U- U_N\|_n + \|U_N\|_n \leq \epsilon + R,
\end{align*}
as desired.
}Perusing \eqref{eq:CompSpSemiNorm} it is clear that the condition  \eqref{eq:UniformTmShiftConvFixedT}
holds  uniformly for elements in $B_{R}$. Thus, as a consequence of Proposition~\ref{thm:MotherOfAllCompEmbeddings}, (ii)
these sets $B_R$ are compact in $\mathcal{X} = L^2(0,T;H)$ for each $R > 0$.

Now, for each $R >0$, we have:
\begin{align}
	\mu^*_N( B_R^c)
	\leq&
	\Prb \left( \|U^*_N\|_{L^2( 0, T ,V)}
	+\|U^*_N\|_{L^\infty ( 0, T , H)}  > R/2\right)
	+\Prb \left( \sup_{j \geq 1} [U^*_N]_{j} > R/2 \right).
	\label{eq:TightnessSplit1}
\end{align}
As a consequence of \eqref{eq:UniformBoundAssertionContProcesses},
\eqref{eq:UniformBndEStoch} and \eqref{eq:additionalStaredProcesses} we have
\begin{align}
	\Prb \bigl( \|U^*_N\|_{L^2( 0, T ,V)}  + \|U^*_N\|_{L^\infty ( 0, T , H)}  > R/2 \bigr)
	\leq \frac{c}{R^2},
	\label{eq:BasicBndsTightness}
\end{align}
for some constant $c$ independent of $N$.

Next we need to establish suitable uniform estimates for $\sup_{j \geq 1} [U^*_N]_{j}$
(cf. \eqref{eq:CompSpSemiNorm}).
To this end we observe with \eqref{eq:ApproxEvolutionComputation}
and \eqref{eq:additionalStaredProcesses} that for any $\theta > 0$,
\begin{align}   \label{eq:SplitOfTDTypeEstimateIntoDetandStochParts}
\int_{0}^{T-\theta} \|U^*_N(t+\theta) - U^*_N(t) \|_{V_{(2)}^{'}}^{4/3} dt
\leq     I_{N}^{D}(\theta) + I_{N}^{S}(\theta),\end{align}
with
\begin{equation*}
I_{N}^{D}(\theta) = c \int_{0}^{T-\theta} \left\| \int_{t}^{t+ \theta} \sum_{n=1}^{N-1}  (\mathcal{N}(U^{n}_{N}) + \ell^n_N)\chi_{(t^{n}, t^{n+1})}(s) ds
 \right\|_{V_{(2)}^{'}}^{4/3} dt,
 \notag
 \end{equation*}
 \begin{equation*}
  I_{N}^{S}(\theta) = c\int_{0}^{T-\theta} \left\|
		        \int_{t}^{t + \theta} \sigma_N(U_{N}) dW \right\|_{V_{(2)}^{'}}^{4/3} dt
		        \notag.
\end{equation*}
To address $I_{N}^{D}(\theta)$ we observe, with \eqref{eq:BSizeV2Prime}
and the standing assumptions on the operators that make up $\mathcal{N}$
in \eqref{eq:TheDriftPartTheWholePotatoe}, that  for any $U \in V$,
\begin{align}
  \|\mathcal{N}(U)\|_{V_{(2)}'}^{4/3} \leq c(|U|^{2/3} + 1)(\|U\|^{2}+1).
  \label{eq:NonlinearTermEstimateTheWholeDeal}
\end{align}
Furthermore it is clear from \eqref{eq:EllIncrement} and H\"older's inequality that,  a.s.
\begin{align*}
	\int_{0}^{T} \sum_{n=1}^{N-1} \| \ell^n_N\|_{V_{(2)}^{'}}^{4/3}\chi_{(t^{n}, t^{n+1})}(s) dt
	\leq& \int_{0}^{T} \sum_{n=1}^{N-1} \left( \frac{1}{\deltaN} \int_{(n-1)\deltaN}^{n\deltaN} \| \ell\|_{V_{(2)}^{'}}^{4/3} ds \right) \chi_{(t^{n}, t^{n+1})}(s) dt
	\leq \int_{0}^{T}  \| \ell\|_{V_{(2)}^{'}}^{4/3} dt.
\end{align*}
Combining these observations we infer that, a.s.
\begin{align}
   I_{N}^{D}(\theta) \leq&
   c \theta^{1/3} \int_{0}^{T-\theta}
   \int_{t}^{t+ \theta} \sum_{n=1}^{N-1} \| \mathcal{N}(U^{n}_{N}) + \ell^n_N\|_{V_{(2)}^{'}}^{4/3}\chi_{(t^{n}, t^{n+1})}(s) ds  dt
   \notag\\
   \leq& c \theta^{1/3} \int_{0}^{T}  \sum_{n=1}^{N-1} \left( (|U^{n}_{N}|^{2/3} + 1)(\|U^{n}_{N}\|^{2} +1 )
   	+ \| \ell\|_{V_{(2)}^{'}}^{4/3}\right)\chi_{(t^{n}, t^{n+1})}(s)  ds
   \notag\\
   \leq& c \theta^{1/3} \left(\max_{0 \leq l \leq N}(1+  |U^l_N|^{2/3})\sum_{j = 1}^{N}  \deltaN (\|U^j_N\|^2 +1 )
   +  \int_{0}^{T}  \| \ell\|_{V_{(2)}^{'}}^{4/3} dt\right)
   \notag\\
   \leq& c \theta^{1/3} \left( \max_{0 \leq l \leq N}(1+  |U^l_N|^{2}) \sum_{j = 1}^{N}  \deltaN (\|U^j_N\|^2 +1) +  \int_{0}^{T} (1+ \| \ell\|_{V^{'}}^{2} )dt\right).
   \label{eq:DetPortDerivatPreEst}
\end{align}
For the term $I_{N}^{S}$ we estimate, for $0 \leq \theta \leq \delta$,
\begin{align}
   \E \left( \sup_{0 \leq \theta \leq  \delta} I_{N}^{S}(\theta)\right)
   \leq& c\int_{0}^{T}  \left( \E  \sup_{0 \leq \theta \leq  \delta} \left\|
		        \int_{t}^{(t + \theta)\wedge T } \sigma_N(U_{N}) dW \right\|_{V_{(2)}^{'}}^{2}   \right)^{2/3} dt
		       \notag \\
   \leq& c \delta^{2/3} \E \sup_{t \in [0,T]} (1 + |U_{N}|^2),		
      \label{eq:StocPortDerivatPreEst}
\end{align}
where the second line follows from Doob's inequality and the standing
assumptions \eqref{eq:SublinearCondH} on $\sigma$  and \eqref{eq:SigmaApproxCond2} on $\sigma_N$:
\begin{align*}
 \E  \sup_{0 \leq \theta \leq  \delta} \left\|
		        \int_{t}^{(t + \theta) \wedge T} \sigma_N(U_{N}) dW \right\|_{V_{(2)}^{'}}^{2}
	\leq  c \E \int_{t}^{(t + \delta) \wedge T} \|\sigma_N(U_{N})\|_{V_{(2)}'}^2 ds
	\leq  c \delta  \E \sup_{t \in [0,T]} (1 + |U_{N}|^2).
\end{align*}

The estimates \eqref{eq:DetPortDerivatPreEst}, \eqref{eq:StocPortDerivatPreEst}
allow the second term
in \eqref{eq:TightnessSplit1} to be treated as follows. Observe that according to
\eqref{eq:CompSpSemiNorm}, \eqref{eq:SplitOfTDTypeEstimateIntoDetandStochParts}
we have
\begin{align*}
	\sup_{j \geq 1} [U^*_N]_{j}^{4/3}
	\leq \sup_{j \geq 1}
	\left(j \sup_{ |\theta| \leq j^{-6}}I_{N}^{D}(\theta)\right) +
	\sup_{j \geq 1}
    \left(j\sup_{ |\theta| \leq j^{-6}}I_{N}^{S}(\theta) \right).
\end{align*}
For the first term we observe with \eqref{eq:DetPortDerivatPreEst}  that
\begin{align}
    \sup_{j \geq 1}  \left(j \sup_{ |\theta| \leq j^{-6}}I_{N}^{D}(\theta)\right)  \leq&
      c \left( \max_{0 \leq l \leq N}(1+  |U^l_N|^{2})
            \sum_{r = 1}^{N}  \deltaN (\|U^r_N\|^2 +1)
             +  \int_{0}^{T}  (\| \ell\|_{V^{'}}^{2} +1) dt \right)
       \notag\\
      :=& c ( T_1^N T_2^N + T_3^N).
      \label{eq:TheCrazySplitUp}
\end{align}
Regarding the second term we simply bound
\begin{align*}
	\sup_{j \geq 1}
	    \left(j \sup_{ |\theta| \leq j^{-6}}I_{N}^{S}(\theta)\right)
	    \leq \sum_{j \geq 1} j \sup_{ |\theta| \leq j^{-6}}I_{N}^{S}(\theta)
	    := T_4^N
\end{align*}
so that for $\rho > 0$, sufficiently large,
\begin{align}
  \Prb \left( \sup_{j \geq 1} [U^*_N]_{j}^{4/3} > \rho \right)
	\leq &
	\Prb ( c(T_1^N T_2^N + T_3^N) + T_4^N > \rho)
	\leq
	\Prb ( cT_1^N T_2^N > \rho/2 ) +
	\Prb (cT_3^N +T_4^N > \rho/2) \notag\\
	\leq &	\Prb \left(  \left\{T_1^N > \sqrt{\rho/(2c)} \right\} \cup \left\{T_2^N > \sqrt{\rho/(2c)} \right\} \right) +
	\Prb (cT_3^N +T_4^N > \rho/2) \notag\\
	\leq &\Prb \left(  T_1^N > \sqrt{\rho/(2c)}\right)  + \Prb\left(T_2^N > \sqrt{\rho/(2c)} \right) +
	\Prb (cT_3^N +T_4^N > \rho/2) \notag\\
	\leq& \frac{c}{\sqrt{\rho}}  \E (T_1^N + T_2^N + T_3^N + T_4^N).
	\label{eq:tightnessNoMomenttrick1}
\end{align}
In view of the uniform bound \eqref{eq:UniformBndConclusion}
established in Proposition~\ref{thm:UniformBnds},
$\sup_N \E T_1^N$ and $\sup_N \E T_2^N$ are both finite.  The term $\sup_N\E T_3^N$, which is
independent of $N$, is finite due to the standing assumption on
$\ell$ (cf. \eqref{eq:DataMomentAssumptions}). For $T_4^N$ we refer
back to \eqref{eq:StocPortDerivatPreEst} and apply the monotone
convergence theorem to infer:
\begin{align*}
  \E T_4^N \leq \sum_{n \geq 1} n^{-3} \E \sup_{t \in [0,T]} (1 + |U_{N}|^2) < \infty.
\end{align*}
We finally conclude that
\begin{align}
 \Prb \left( \sup_{j \geq 1} [U^*_N]_{j} > R/2\right) =
 \Prb \left( \sup_{j \geq 1} [U^*_N]_{j}^{4/3} > (R/2)^{4/3}\right)
 \leq \frac{c}{R^{2/3}}.
 \label{eq:FupFinalTightnessTerm}
\end{align}
Combining \eqref{eq:TightnessSplit1}, \eqref{eq:BasicBndsTightness}
and \eqref{eq:FupFinalTightnessTerm} we now conclude that (cf. Appendix~\ref{sec:LetsConvergeToaMeasure})
\begin{align}
\{\mu^*_N\}_{N \geq 1} \textrm{ is tight in } \mathcal{X} = L^2( 0, T ;H).
\label{eq:tightnessInX}
\end{align}

\subsubsection{Tightness for $\mu^{**}_{N}$ in $ \mathcal C([0,T];V_{(3)}')$}

We next show that $\mu_N^{**}$ is tight in $ \mathcal C([0,T],V_{(3)}')$.    For this purpose
we make appropriate usage of a compact embedding from \cite{FlandoliGatarek1} (see also \cite{Temam4}).
Let us fix any $p \in (2,\infty), \alpha \in (0,1/2)$ such that $\alpha p > 1$.  According
to \cite{FlandoliGatarek1}:
\begin{align}
W^{1,4/3}( 0, T ; V_{(2)}') \subset \subset \mathcal C([0,T]; V_{(3)}'),
\quad W^{\alpha,p}( 0, T ; V_{(2)}') \subset \subset \mathcal C([0,T]; V_{(3)}'),
\label{eq:FlandoliGatImbeddings}
\end{align}
that is, the embeddings are continuous and compact.
We now define
\begin{align*}
B_R :=& \left\{ X \in \mathcal C([0,T]; V_{(3)}'): \|X\|_{W^{1,4/3}( 0, T ; V_{(2)}')} \leq R\right\}
            +\left\{ Y \in \mathcal C([0,T]; V_{(3)}'): \|Y\|_{W^{\alpha,p}( 0, T ; V_{(2)}')} \leq R\right\} \\
            :=& B_{R}^{D} + B_{R}^{S}
\end{align*}
for any $R > 0$.  With \eqref{eq:FlandoliGatImbeddings}, it is clear that $B_{R}$ is compact
in $ \mathcal C([0,T]; V_{(3)}')$ for every $R>0$.
Observe moreover that, in view of \eqref{eq:ustaruRelationship}
\begin{align*}
  \{U_{N}^{**} \in B_{R}\} \supseteq
  \left\{ U^{0}_{N}  + \int_{0}^{\cdot} (\mathcal{N}(U_{N}) + \ell)ds  \in B_{R}^{D}\right\}
  \cap \left\{ \int_{0}^{\cdot} \sigma_N(U_{N}) dW \in B_{R}^{S} \right\},
\end{align*}
and thus that
\begin{align}
	\mu_{N}^{**}(B_{R}^{C})
	\leq& \Prb\left( \left\|   U^{0}_{N}  + \int_{0}^{\cdot} (\mathcal{N}(U_{N}) + \ell)ds
		      \right\|_{W^{1,4/3}( 0, T ; V_{(2)}')} > R \right)
		      +
	      \Prb\left( \left\| \int_{0}^{\cdot} \sigma_N(U_{N}) dW\right\|_{W^{\alpha,p}( 0, T ; V_{(2)}')} > R
	      \right)
	      \notag\\
	      :=& S_{N}^R + T_{N}^R.
\end{align}
Hence we will infer that $\{\mu_N^{**}\}$ is tight in $ \mathcal C([0,T],V_{(3)}')$ if we can show that
 $T_N^R, S_N^R$ converge uniformly in $N$ to zero as $R \uparrow \infty$.

For $T_N^R$ we estimate, with \eqref{eq:NonlinearTermEstimateTheWholeDeal}
\begin{align*}
 &\left\|   U^{0}_{N}  + \int_{0}^{\cdot} (\mathcal{N}(U_{N}) + \ell)ds
		      \right\|_{W^{1,4/3}( 0, T ; V_{(2)}')}^{4/3}
		      \notag\\
	&\quad \quad \quad \quad
	\leq c(1+ |U_{N}^{0}|^{2})
	+ c \int_{0}^{T} \|\mathcal{N}(U_{N}) + \ell\|_{V_{(2)}'}^{4/3}dt
	\notag \\
	&\quad \quad \quad \quad
	\leq
	c (1+|U_{N}^{0}|^{2})
	+ c  \int_{0}^{T} \left( (|U_{N}|^{2/3} + 1)(\|U_{N}\|^{2}+1) + \|\ell\|_{V'}^{2}\right)dt,
        \notag\\
        &\quad \quad \quad \quad
        	\leq c\sup_{t \in [0,T]} (1+ |U_{N}|^{2}) \cdot \left(
	\int_{0}^{T} \left( 1 + \|U_{N}\|^{2} + \|\ell\|_{V'}^{2}\right)dt + 1\right).
\end{align*}
Thus we find, cf. \eqref{eq:tightnessNoMomenttrick1}:
\begin{align}
	 T_N^R
	 &\leq	 \Prb\left(c\sup_{t \in [0,T]} (1+ |U_{N}|^{2}) \cdot
	\left(\int_{0}^{T} \left( 1 + \|U_{N}\|^{2} + \|\ell\|_{V'}^{2}\right)dt + 1\right)
	> R \right)
	\notag\\
	&\leq
	\Prb\left(c\sup_{t \in [0,T]} (1+ |U_{N}|^{2}) > R^{1/2} \right) +\Prb\left(
	\int_{0}^{T} \left( 1 + \|U_{N}\|^{2} + \|\ell\|_{V'}^{2}\right)dt +1 > R^{1/2} \right)
	\notag\\
	  &\leq
	  \frac{c}{R^{1/2}} \E\sup_{t \in [0,T]} (1+ |U_{N}|^{2}) +  \frac{1}{R^{1/2}}
	\E \left(\int_{0}^{T} \left( 1 + \|U_{N}\|^{2} + \|\ell\|_{V'}^{2}\right)dt +1\right).
	 \label{eq:DetChebEstimateforCV3primeTightness1}
\end{align}	

We turn to $S_N^R$.  For this purpose let us define for any $R > 0$ the stopping times
\begin{align*}
	\tau_{R} := \inf_{t \geq 0} \left\{ \sup_{s \in [0,t]} |U_{N}| \geq R \right\} \wedge T
		     = \sup_{t \geq 0} \left\{ \sup_{s \in [0,t]} |U_{N}| < R \right\} \wedge T.
\end{align*}
Using $\tau_{R}$ we now estimate with the Chebyshev inequality that
\begin{align}
	      S_N^R \leq&   \Prb\left( \left\| \int_{0}^{\cdot \wedge \tau_{R}} \sigma_N(U_{N}) dW\right\|_{W^{\alpha,p}( 0, T ; V_{(2)}')} > R,
	      \tau_{R} \geq T
	      \right)
	      + \Prb( \tau_{R} < T) \notag\\
	      \leq&
	       \Prb\left( \left\| \int_{0}^{\cdot \wedge \tau_{R}} \sigma_N(U_{N}) dW\right\|_{W^{\alpha,p}( 0, T ; V_{(2)}')} > R
	      \right)
	      + \Prb \left( \sup_{s \in [0,T]} |U_{N}| \geq R \right)\notag\\
	      \leq&
	      \frac{1}{R^{p}}
	      \E \left(
                \left\| \int_{0}^{\cdot \wedge \tau_{R}} \sigma_N(U_{N}) dW\right\|_{W^{\alpha,p}( 0, T ; V_{(2)}')}^{p}
                \right)
                + \frac{1}{R^{2}} \E \sup_{s \in [0,T]} |U_{N}|^{2}.
                \label{eq:StochasticChebEstimateforCV3primeTightness1}
\end{align}
Now in order to treat this final stochastic integral term we recall the following
generalization of the Burkholder-Davis-Gundy inequality from e.g. \cite{FlandoliGatarek1}:
for a given Hilbert space $X$, $p \geq 2$  and $\alpha \in [0,1/2)$ we have
for all $X$-valued predictable $G \in L^{p}(\Omega; L^{p}_{loc}(0,
\infty,L_{2}(\mathfrak{U}, X)))$
\begin{align*}
  \E \left(
     \left\| \int_0^{\cdot} G dW \right\|_{W^{\alpha, p}( 0, T ;X)}^p \right)
     \leq c \E \left(
    \int_0^T |G|_{L_2(\mathfrak{U}, X)}^p dt \right),
\end{align*}
which holds with a constant $c$ depending only on $\alpha$ and $p$.   Continuing now from \eqref{eq:StochasticChebEstimateforCV3primeTightness1}
we have
\begin{align}
	S_N^R
               \leq& \frac{c}{R^{p}}  \E  \int_0^{T \wedge \tau_{R}} |\sigma_N(U_{N})|_{L_2(\mathfrak{U}, H)}^p dt
                + \frac{1}{R^{2}} \E \sup_{s \in [0,T]} |U_{N}|^{2}
                 \leq \frac{c}{R^{p}}  \E  \sup_{s \in [0,T \wedge \tau_{R}]} (1 + |U_{N}|^{p})
                + \frac{1}{R^{2}} \E \sup_{s \in [0,T]} |U_{N}|^{2}
                \notag\\
                \leq& \frac{c(1+R^{p-2})}{R^{p}}  \E  \sup_{s \in [0,T]} (1+ |U_{N}|^{2})
                + \frac{1}{R^{2}} \E \sup_{s \in [0,T]} |U_{N}|^{2}
                \leq\frac{c}{R^{2}} \E \sup_{s \in [0,T]} (1+|U_{N}|^{2}).
                \label{eq:StochasticChebEstimateforCV3primeTightness2}
\end{align}
Combining the estimates \eqref{eq:DetChebEstimateforCV3primeTightness1},
\eqref{eq:StochasticChebEstimateforCV3primeTightness2}
with \eqref{eq:UniformBoundAssertionContProcesses} we finally conclude
\begin{align*}
	\sup_{N \geq N_{1}} \mu^{**}_N (B_R) \geq 1 - \frac{c}{R^{1/2}}
\end{align*}
and hence infer
\begin{align}
\{\mu^{**}_N\}_{N \geq N_{1}} \textrm{ is tight in }\mathcal C([0,T]; V_{(3)}').
\label{eq:tightnessInY}
\end{align}

\begin{Rmk}\label{rmk:LetsBragAboutAvoidingMoments}
Let us observe that the tightness bounds for $\mu^{**}_N$
and $\mu^{*}_N$ could be carried out differently if we
had available, for example, the uniform bounds on `higher moments' like
\begin{align}
     \sup_{N \geq 1} \E \left( \max_{0 \leq k \leq N}|U^k_N|^{4} +
      \left(\sum_{k = 1}^{N}  \deltaN \|U^k_N\|^2 \right)^2 \right)
      < \infty
      \label{eq:LetgetHigh_er_moments_Discrete}
\end{align}
or equivalently that
\begin{align}
     \sup_{N \geq 1} \E \left( \sup_{t \in [0,T]} |U_N|^4 +
     \left(\int_0^T \|U_N\|^2 dt
     \right)^2 \right)
      < \infty.
      \label{eq:LetgetHigh_er_moments}
\end{align}
Indeed, in numerous other previous works related
to stochastic fluids equations (see e.g.
\cite{Bensoussan1,FlandoliGatarek1,MenaldiSritharan,
DebusscheGlattTemam1,GlattVicol2011})
estimates analogous to \eqref{eq:LetgetHigh_er_moments}
are established essentially via Ito's lemma in order to achieve
tightness in the probability laws associated to a regularization
scheme.

 In the current situation, instead due to the way we carry out the estimates
in \eqref{eq:tightnessNoMomenttrick1},
\eqref{eq:DetChebEstimateforCV3primeTightness1}
and
\eqref{eq:StochasticChebEstimateforCV3primeTightness1}--\eqref{eq:StochasticChebEstimateforCV3primeTightness2}, we have adopted a different approach, namely,    we   establish tightness (compactness)
estimates without recourse to such higher moment estimates.

A different method using higher moments will be shown in   the related work  \cite{GlattTemam2011}.

\end{Rmk}

\subsubsection{Cauchy Arguments and Conclusions}
\label{cauchy}

With \eqref{eq:tightnessInX} and \eqref{eq:tightnessInY} now in hand it is then simply a matter of
collecting the various convergences above to complete the proof of Proposition~\ref{prop:Tightness}

By making use of Prohorov's theorem (cf. Section~\ref{sec:LetsConvergeToaMeasure} in the Appendix) with \eqref{eq:tightnessInX} we infer the existence of a probability measure
$\mu$ such that, up to a subsequence,
\begin{align*}
  \mu^*_N \rightharpoonup \mu.
   \quad  \textrm{ on } \mathcal{X} = L^2( 0, T ;H)
    \textrm{ (and also on } L^{2}( 0, T ; V')).
\end{align*}
Due to \eqref{eq:additionalStaredProcesses} with \eqref{eq:ConvTogetherUUbar} and \eqref{eq:GoodConvL2HErrorsStochastic}
it is clear that $U^*_{N}- U_N$ converges to zero
in $\mathcal{X} = L^2( 0, T ;H)$ and hence in $L^{2}( 0, T ;V')$ a.s..  Hence, by now invoking  \eqref{eq:GoodConvL2HErrors} and referring
back once more to  \eqref{eq:additionalStaredProcesses},  we have that $U^*_N - U^{**}_N$
converges to zero in $L^{2}( 0, T ;V')$ a.s.. Thus, invoking Lemma~\ref{thm:ConvTogether},
we conclude, again up to a subsequence, that:
\begin{align}\label{useofconv}
  \mu^{**}_N \rightharpoonup \mu
   \textrm{ on } L^2( 0, T ;V')
  \textrm{ and } \mu_N \rightharpoonup \mu
    \textrm{ on } \mathcal{X} = L^2( 0, T ;H).
\end{align}
In particular this is the first desired convergence for $\{\mu_{N}\}_{N \geq N_{1}}$, \eqref{eq:MuNConvL2H}.  On the other hand invoking
Prohorov's theorem with \eqref{eq:tightnessInY} and the convergence just established
for $\{\mu_{N}^{**}\}_{N \geq N_{1}}$ in $L^{2}( 0, T ;V')$ we see that  $\mu_{N}^{**}$ is tight
in $\mathcal{Y} = L^2( 0, T ;V') \cap \mathcal C([0,T]; V_{(3)}')$.  By Prohorov's theorem
in the other direction and passing to a further subsequence as needed we have
$$
  \mu^{**}_N \rightharpoonup \tilde{\mu}
  \textrm{ on } \mathcal{Y} = L^2( 0, T ;V') \cap \mathcal C([0,T], V_{(3)}').
$$
Since, clearly, $\tilde{\mu} = \mu$ this yields the second desired item \eqref{eq:MuNstarstarConvCV3}.
The proof of Proposition~\ref{prop:Tightness} is therefore complete.

\subsection{Proof of Theorem~\ref{thm:MainExistenceMGSol} \\Conclusion:Almost Sure convergence and the Passage to the Limit on the Skorokhod Basis }
\label{sec:SBPassagetoTheLimit}

We now have all of the ingredients to finally prove one the main results of this article, namely
Theorem~\ref{thm:MainExistenceMGSol}.  Suppose that
we are given $\mu_{U_{0}} \in Pr(H)$ and $\mu_{\ell} \in Pr( L^{2}_{loc}(0,\infty; V'))$ according to the conditions
specified in Definition~\ref{def:MGSol}.   As mentioned  in Remark  \ref{notthefiltrationyet}  now it is necessary to introduce  the stochastic
basis $\mathcal{S}_{\mathcal{G} }$ (defined as in subsection~\ref{sec:specificfiltration}),
 an element $U^{0} $  which is ${{\mathcal{G}}}_{0}$ measurable
 and a process $ \ell =\ell(t) $  measurable with respect to  the sigma algebra generated by the $  {W} (s)$ for $ s \in [0, t]$\footnote{Note that since the
  sigma algebra generated by the $  {W} (s)$ for $ s \in [0, t])$ is the smallest respect to which $W(t)$ is measurable, so $\ell(t)$ is
  adapted to $\{\mathcal {F}_t\}_{t\geq0}$, and hence all the previous results applies.
  },whose laws coincide  with those of
$\mu_{U_{0}}, \mu_{\ell}$.   Thus Proposition \ref{thm:ExistenceSelectionThm} applies and we obtain the existence of the $U^n_N$'s adapted to $ {\mathcal{G}}_{t_n}$.

We then approximate $U^{0} \in L^{2}(\Omega;H)$ with a
sequences of elements   $\{U^{0}_{N}\}_{N \geq 1} \subseteq L^{2}(\Omega,V_{(2)})$,
which maintains the bound \eqref{eq:InitialDataBnd} as described in Remark~\ref{rmk:SomeCrazyApproximation}
above.
Proposition~\ref{prop:BasicPropertiesOfConProcesses} applies and hence we can use this sequence $\{U^{0}_{N}\}_{N \geq N_{1}}$,  the
process $\ell$, and the sequence $U^n_N$
to define processes $\{U_{N}\}_{N \geq N_{1}}$, $\{U^{**}_{N}\}_{N \geq N_{1}}$ according to \eqref{eq:StepProcessesN}
and \eqref{eq:additionalStaredProcesses} respectively ($N_{1}$
is given by \eqref{eq:BndsStartingDiscretization}).  In order to pass to  the limit in the associated evolution
equation \eqref{eq:ustaruRelationship}, we consider the product measures:
\begin{align*}
	\nu_{N}(\cdot) := \Prb( (U^{**}_{N}, U_{N},\ell,W) \in \cdot)
\end{align*}
which are defined on the space
\begin{equation}\label{spacez}
\mathcal{Z} = \mathcal{Y} \times \mathcal{X} \times L^{2}( 0, T ;V') \times \mathcal C([0,T]; \mathfrak{U}_0).
\end{equation}
where, as above, $\mathcal{Y} = L^2( 0, T ;V') \cap \mathcal C([0,T], V_{(3)}')$, $\mathcal{X} = L^2( 0, T ; H)$,
and $\mathfrak{U}_{0}$ is defined as in Section~\ref{sec:SomeStochasticAnal}, \eqref{eq:AuxSpace}.
By invoking Proposition~\ref{prop:Tightness} we have that (passing to a subsequences as needed)
$\mu_{N} \rightharpoonup \mu$  on $\mathcal{X}$  and $\mu^{**}_{N} \rightharpoonup \mu$
on $\mathcal{Y}$, where $\mu_{N}$ and $\mu_{N}^{**}$ are defined as in \eqref{eq:BorelMeasuresForUUstarUstarstar}.  It follows, again up to passing to a subsequence, that $\nu_{N}$ converges weakly to
a measure $\nu$ on $\mathcal Z$ (defined in \eqref{spacez}).  Furthermore, recalling \eqref{eq:additionalStaredProcesses}
and making use of \eqref{eq:ConvTogetherUUbar}, \eqref{eq:GoodConvL2HErrors}, \eqref{eq:GoodConvL2HErrorsStochastic}
it is not hard to see that:
\begin{align*}
  \nu(  \{ (U^{**}, U,  \ell, W) \in \mathcal{Z} : U \not = U^{**} \}) = 0.
\end{align*}
Thus, by making use of the Skorokhod embedding theorem (see Section \ref{sec:LetsConvergeToaMeasure}) we obtain, relative
to a new probability space $(\tilde{\Omega}, \tilde{\mathcal{F}}, \tilde{\Prb})$, a sequence of
random variables
\begin{align}
	(\tilde{U}^{**}_{N}, \tilde{U}_{N}, \tilde{\ell}_{N}, \tilde{W}_{N}) \rightarrow (\tilde{U}, \tilde{U},  \tilde{\ell}, \tilde{W}) \quad \tilde{\Omega} \textrm{ a.s. in } \mathcal{Z}.
	\label{eq:jointConvNewBasis}
\end{align}
Moreover, the uniform bounds for $\{U_{N}\}_{N \geq N_{1}}$ in
$L^{2}(\Omega;L^{2}( 0, T ;V) \cap L^{\infty}( 0, T ;H))$ from
Proposition~\ref{prop:BasicPropertiesOfConProcesses}, \eqref{eq:UniformBoundAssertionContProcesses} imply that
in addition to \eqref{eq:jointConvNewBasis}
we also have
\begin{align}
   \tilde{U}_{N} \rightharpoonup \tilde{U}
   \quad
   \textrm{ weakly in } L^{2}(\Omega; L^{2}( 0, T ;V) )
   \textrm{ and weakly-star in }
   L^{2}(\Omega; L^{\infty}( 0, T ;H)).
   \label{eq:MoreWeakConvOnNewBasis}
\end{align}

Following a procedure very similar to \cite{Bensoussan1} we may now show that
$\tilde{W}_N$ is a cylindrical Brownian motion relative to the filtration
$\tilde{\mathcal{F}}^N_t $ defined as the sigma algebra generated by the $  (\tilde{U}^{**}_{N}(s), \tilde{U}_{N}(s), \tilde{\ell}_{N}(s), \tilde{W}_{N}(s))$ for $ s \leq t$
and that $(\tilde{U}^{**}_{N}, \tilde{U}_{N}, \tilde{\ell}_{N},\tilde{W}_{N})$ satisfies
\eqref{eq:ustaruRelationship} on the `Skorokhod space' $(\tilde{\Omega}, \tilde{\mathcal{F}}, \tilde{\Prb})$ viz.
\begin{align}
  \tilde{U}^{**}_{N}(t)
        	&= \tilde{U}^{**}_{N}(0)  + \int_{0}^{t} (\mathcal{N}(\tilde{U}_{N}) + \tilde{\ell}_{N})ds
		      + \int_{0}^{t} \sigma_{N}(\tilde{U}_{N}) d\tilde{W}_{N}.
		      \label{eq:StochasticApproxEvolution}
\end{align}
Using the convergences in \eqref{eq:jointConvNewBasis}--\eqref{eq:MoreWeakConvOnNewBasis}
with \eqref{eq:StochasticApproxEvolution} it is standard\footnote{Note that,
in particular, the stochastic terms involving $\sigma_N(U_N)$ converge due to \eqref{eq:SigmaApproxCond3}.
}
 to show that
$\tilde{U}$ satisfies  \eqref{eq:solRegularity}--\eqref{eq:InitialDataDist}
relative to the stochastic basis ${\tilde{\mathcal{S}}}:=(\tilde{\Omega}, \tilde{\mathcal{F}}, \{\tilde{\mathcal{F}}_t\}_{t \geq 0}, \tilde{\Prb},
\{\tilde{W}_k\}_{k \geq 1})$
where $\{\tilde{\mathcal{F}}_t\}_{t \geq 0}$ is defined as the sigma algebra generated by the   $ (\tilde{U}(s),  \tilde{\ell}(s),$ $\tilde{W}(s))$ for $s \leq t $
and $\tilde{W}_k = (\tilde{W},e_k)_{\mathfrak{U}}$.
Therefore
$(\tilde{S},\tilde{U}, \tilde{\ell})$ is a Martingale solution of
\eqref{eq:PEsAbstractFormulation} relative  to $\mu_{U_{0}}, \mu_{\ell}$
in the sense of Definition~\ref{def:MGSol} and the
proof of Theorem~\ref{thm:MainExistenceMGSol}
is complete.

\section{Convergence of the Euler Scheme}

We conclude  by  reinterpreting   from the point of view of numerical analysis, the study above as a result of convergence  for the Euler scheme (\ref{eq:EulerSemiFnRep}).
\begin{Thm}\label{convergenceof}
We assume given $\mu_{U_{0}} \in Pr(H)$ and $\mu_{\ell} \in Pr( L^{2}_{loc}(0,\infty; V'))$ according to Definition~\ref{def:MGSol}. We also
 assume given the stochastic basis $\mathcal{S}_{\mathcal{G} }$ (defined as in subsection~\ref{sec:specificfiltration}),
an element $U^{0}$  which is ${{\mathcal{G}}}_{0}$ measurable
 and a process $\ell=\ell(t)$  measurable with respect to the sigma algebra generated by the $  {W} (s)$ for $ s \in [0, t]$,   whose laws coincide with those of
$\mu_{U_{0}}, \mu_{\ell}$.  Let a
sequences of elements   $\{U^{0}_{N}\}_{N \geq 1} \subseteq L^{2}(\Omega,V_{(2)})$  approximate $U^{0} \in L^{2}(\Omega;H)$  as described in Remark~\ref{rmk:SomeCrazyApproximation}.
Then the processes $\{U_{N}\}_{N \geq N_{1}}$ defined   according to \eqref{eq:StepProcessesN}
 ($N_{1}$
is given by \eqref{eq:BndsStartingDiscretization}) adapted to $\{\mathcal{G}_t\}_{t \geq 0}$   exist.

Moreover the family  $\{ \mu_{N}\}$  of   probability laws of $\{U_{N}\}$, is weakly
 compact over the phase space $L^2 (0, T; \, H)\cap \mathcal C([0,T], V_{(3)}')$ and hence converges weakly to a probability measure $\mu$ on  the same phase space up to a subsequence.
Furthermore, there exists a probability space $(\tilde \Omega, \tilde F, \tilde P)$ and a subsequence of random vectors $(\tilde U_{N_k},   \tilde {\ell}_{N_k}, \tilde W_{N_k})$ with
 values in $\mathcal {Z}_1:=L^2 (0, T; \, H) \cap\mathcal C([0,T], V_{(3)}') \times L^2 (0, T; V^\prime) \times  \mathcal \mathcal C([0,T]; \mathfrak{U}_0)$ such that

(i) $(\tilde {U}_{N_k},  \tilde {\ell}_{N_k}, \tilde {W}_{N_k} )$ have the same probability distribution as $( U_{{N_k}},  \ell,  W)$.

(ii) $(\tilde {U}_{{N_k}}, \tilde {\ell}_{N_k},   \tilde {W}_{{N_k}})$ converges almost surely as ${N_k}\rightarrow \infty$, in the topology of $\mathcal {Z}_1$, to an element $( \tilde{U},  \tilde{\ell}, \tilde{W})$. Particularly,
\begin{equation}\label{strongconv1}
\tilde {U}_{{N_k}} \rightarrow \tilde U \mbox{ strongly  in } L^2 (0, T; \, H) \cap \mathcal C([0,T], V_{(3)}')\,\,a.s.,
\end{equation}
where $  \tilde{U} $ has the probability distribution $\mu$.
\begin{proof}
The existence of   $\{U_{{N_k}}\}_{N \geq N_{1}}$  follows directly from the existence of  the $U^n_N$'s proven in Proposition \ref{thm:ExistenceSelectionThm}.
(i) and (ii) follow from the Skorokhod embedding theorem (see Section \ref{sec:LetsConvergeToaMeasure})  as shown in Section \ref{sec:SBPassagetoTheLimit}.
\end{proof}
\end{Thm}

\section{Applications for Equations in Geophysical Fluid Dynamics}
\label{sec:GeophysicsExample}

In this section we apply the above framework culminating
in Theorem~\ref{thm:MainExistenceMGSol} and Theorem~\ref{convergenceof}
to a stochastic version of the Primitive Equations.
Our presentation here will focus on the case of
the equations of the oceans.  Note however that the abstract setting
introduced above is equally well suited to derive results for
analogous systems for the atmosphere or for the
coupled oceanic-atmospheric system (COA).\footnote{Via a suitable change of variables,
the dynamical equations for the compressible gases which constitute the earth's atmosphere may be shown to take
a mathematical form essentially similar to the incompressible equations for the oceans.}
We refer the interested reader to \cite{PetcuTemamZiane} for
further details on these other interesting situations.

\subsection{The Oceans Equations}
\label{sec:TheEquations}

The stochastic primitive equations of the Oceans take the form:
\begin{subequations}\label{eq:PE3DBasic}
  \begin{gather}
    \begin{split}
    \pd{t} \mathbf{v}
    + \nabla_{\mathbf{v}} \mathbf{v} + w \pd{z}\mathbf{v}
    + \frac{1}{\rho_0} \nabla p
    +& f \mathbf{k} \times \mathbf{v}
    - \mu_{\mathbf{v}} \Delta \mathbf{v}
    - \nu_{\mathbf{v}} \pd{zz} \mathbf{v}
    = F_{\mathbf{v}} + \sigma_{\mathbf{v}}(\mathbf{v},T,S) \dot{W}_1,
    \end{split}
    \label{eq:MomentumPE}\\
    \pd{z} p = - \rho g,
    \label{eq:HydroStaticPE}\\
    \nabla \cdot \mathbf{v} + \pd{z} w = 0
    \label{eq:divFreeTypeCondPE}\\
    \pd{t} T +\nabla_{\mathbf{v}} T
             + w \pd{z} T
             - \mu_{T} \Delta T
             - \nu_{T} \pd{zz} T
             = F_{T} + \sigma_{T}(\mathbf{v},T,S) \dot{W}_2,
    \label{eq:diffEqnTempPE}\\
    \pd{t} S + \nabla_{\mathbf{v}} S
             + w \pd{z} S
             - \mu_{S} \Delta S
             - \nu_{S} \pd{zz} S
             = F_{S} + \sigma_{S}(\mathbf{v},T,S) \dot{W}_3,
    \label{eq:diffEqnSaltPE}\\
    \rho = \rho_0 ( 1 + \beta_T( T - T_r) + \beta_S(S- S_r)).
    \label{eq:linearDensityDependence}
  \end{gather}
\end{subequations}
Here, $U := (\mathbf{v},T,S)= (u,v, T, S)$, $p$, $\rho$ represent the
horizontal velocity, temperature, salinity, pressure and density of the fluid under
consideration; $\mu_{\mathbf{v}}$, $\nu_{\mathbf{v}}$, $\mu_{T}$, $\nu_{T}$,
$\mu_{S}$, $\nu_{S}$ are positive coefficients which account for
the eddy and molecular diffusivities (viscosity) in the equations for
$\mathbf{v}$, $T$ and $S$.
The terms $F_{\mathbf{v}}, F_{T}, F_{S}$ are volumic sources of momentum,
heat and salt which are zero in idealized situations but which we consider to be
random in general.

The state dependent stochastic terms are driven by independent Gaussian white noise
processes $\dot{W}_{j}$, $j = 1,2,3$ which are formally delta correlated in time.
The stochastic terms may be written in the expansion
\begin{align}
       \sigma_{U}(U) \dot{W} = \left(
      \begin{array}{c}
    \sigma_{\mathbf{v}}(U) \dot{W}_1(t,\spX)\\
    \sigma_{T}(U)\dot{W}_2(t,\spX)\\
    \sigma_{S}(U)\dot{W}_3(t,\spX)\\
  \end{array}
      \right)
      =        \sum_{k \geq 1}  \left(
      \begin{array}{c}
    \sigma_{\mathbf{v}}^k(U)(t,\spX) \dot{W}_1^k(t)\\
    \sigma_{T}^k(U)(t,\spX)\dot{W}_2^k(t)\\
    \sigma_{S}^k(U)(t,\spX)\dot{W}_3^k(t)\\
  \end{array}
      \right),
      \label{eq:expStochFormal}
\end{align}
where the elements $\dot{W}_j^k$ are independent 1-D white (in time) noise
processes. We may interpret the multiplication in \eqref{eq:expStochFormal} in either the It\={o}
or the Stratonovich sense;  as we detail in one example below
the classical correspondence between the It\={o} and Stratonovich systems
allows us to treat both situations within the framework of
the It\={o} evolution \eqref{eq:PEsAbstractFormulation}.
We will describe
some physically interesting configurations of these `stochastic terms' in
detail below in Subsection~\ref{sec:StochasticForcings}.

The operators $\Delta = \pd{xx} + \pd{yy}$ and $\nabla = (\pd{x}, \pd{y})$
are the horizontal laplacian and gradient operator.  Here the operator $\nabla_{\mathbf{v}}$ captures
part of the convective (material) derivative and is defined according to
\begin{align}
 \nabla_{\mathbf{v}} := \mathbf{v} \cdot \nabla =  u \pd{x} + v \pd{y}.
\end{align}
\begin{Rmk}\label{rmk:SphericalGeo}
As given, the model \eqref{eq:PE3DBasic}, expresses the equations for
Oceanic flows in the `beta-plane approximation', that is to say we make use of the
fact that the earth is locally flat.
This setting is suitable for regional studies and we will focus on this case for
the simplicity of presentation. With suitable adjustments to the definition of the
operators $\Delta$, $\nabla$, $\nabla_{\mathbf{v}}$ and to the domain
introduced below we could consider the evolutions in the full spherical geometry
of the earth.  We refer to \cite{LionsTemamWang1} (and also to \cite{PetcuTemamZiane})
for further details on how to cast  a global circulation model in the form of e.g. (\ref{eq:PEsAbstractFormulation}).
\end{Rmk}
\subsubsection{Domain and Boundary Conditions}

The evolution \eqref{eq:PE3DBasic}  takes place on a bounded domain $\mathcal{M} \subset \mathbb{R}^{3}$
which we define as follows.  Fix a bounded, open domain $\Gamma_{i} \subset \RR^{2}$ with sufficiently smooth
boundary ($\mathcal{C}^{3}$, say);  $\Gamma_{i}$ represents the
surface of the ocean in the region under consideration.
We suppose we have defined a `depth' function $h = h(x,y): \Gamma_{i} \rightarrow \mathbb{R}$
which is at least $\mathcal{C}^{2}$ and is subject to the restriction
  $0 <\underline{h} \leq h(x,y) \leq \bar{h}$.
With these ingredients we then let
\begin{align*}
  \mathcal{M} := \left\{ \spX := (x,y,z) \in \mathbb{R}^{3}: (x,y) \in \Gamma_{i}, z \in (-h(x,y), 0)
                          \right\}.
\end{align*}
The boundary $\partial \mathcal{M}$ of $\mathcal{M}$,  is divided into its top $\Gamma_{i}$
lateral $\Gamma_{l}$ and bottom $\Gamma_{b}$ boundaries.
We denote  the outward unit normal to $\partial \mathcal{M}$
by $\mathbf{n}$ and the normal to $\Gamma_{l}$ in $\RR^2$
by $\mathbf{n}_H$.

We next prescribe the following, physically realistic boundary conditions
for equation \eqref{eq:PE3DBasic} considered in $\mathcal{M}$.  See
e.g.   \cite{PetcuTemamZiane}  for further details.  On $\Gamma_i$
we suppose
\begin{equation}\label{eq:3dPEBCTop}
	\pd{z} \mathbf{v} + \alpha_{\mathbf{v}} (\mathbf{v} - \mathbf{v}^{a})= \tau_{\mathbf{v}}, \quad w = 0,\quad
         \pd{z} T + \alpha_{T} (T - T^a) = 0,\quad
    \pd{z} S = 0,
\end{equation}
where $\alpha_{\vel}$, $\alpha_{T}$ are fixed positive constants
and $\tau_{\vel}$, $\vel^{a}$, $T^{a}$ are in general random
and non-constant in space and time.  Physically speaking, the first two equations in
\eqref{eq:3dPEBCTop} account for a boundary layer model
where $\vel^{a}$, $T^{a}$ represent the values for
velocity and temperature of the atmosphere at the surface
of the oceans; $\tau_{\vel}$ accounts for the shear of the
wind.

At the bottom of the ocean $\Gamma_b$ we take
\begin{equation}\label{eq:3dPEBCBottom}
  \begin{split}
  \mathbf{v} = 0, \quad w = 0, \quad
  \pd{\mathbf{n}} T = 0, \quad \pd{\mathbf{n}} S = 0.\\
  \end{split}
\end{equation}
Finally for the lateral boundary $\Gamma_l$
\begin{equation}\label{eq:3dPEBCSide}
  \mathbf{v} = 0, \quad \pd{\mathbf{n}}T = 0, \quad \pd{\mathbf{n}}S = 0.
\end{equation}
Note that, in view of the Neumann (no-flux) boundary conditions imposed on $S$
in \eqref{eq:3dPEBCTop}--\eqref{eq:3dPEBCSide}, there is no loss in generality
in assuming
\begin{align}
   \int_{\mathcal{M}} S d \mathcal{M} = 0=  \int_{\mathcal{M}} F_{S} d\mathcal{M}.
\label{eq:MeanZeroSaltState}
\end{align}
See \cite{PetcuTemamZiane} for further details.  Finally
\eqref{eq:PE3DBasic}--\eqref{eq:MeanZeroSaltState} are supplemented with
 initial conditions for $\mathbf{v}$, $T$ and $S$, that is
\begin{equation}\label{eq:InitialCondPEs}
   \mathbf{v} = \mathbf{v}_0, \quad
   T = T_0, \quad
   S = S_0, \quad
   \textrm{ at } t = 0.
\end{equation}
\subsubsection{A Reformulation of the Equations}
\label{sec:Reformulation}

Starting from the incompressibility condition, \eqref{eq:divFreeTypeCondPE} and the
hydrostatic equation \eqref{eq:HydroStaticPE} we may derive an equivalent
form for \eqref{eq:PE3DBasic} as follows.
\begin{subequations}\label{eq:PE3DBasicReform}
  \begin{gather}
    \begin{split}
    \pd{t} \mathbf{v}
    + \nabla_{\mathbf{v}} \mathbf{v} + w(\vel) \pd{z}\mathbf{v}
    + \frac{1}{\rho_0} \nabla p_{s}
    +& f \mathbf{k} \times \mathbf{v}
    - \mu_{\mathbf{v}} \Delta \mathbf{v}
    - \nu_{\mathbf{v}} \pd{zz} \mathbf{v}
    = F_{\mathbf{v}} - \nabla P +
     \sigma_{\mathbf{v}}(\mathbf{v},T,S) \dot{W}_1,
    \end{split}
    \label{eq:MomentumPERF}\\
    \pd{t} T +\nabla_{\mathbf{v}} T
             + w(\mathbf{v}) \pd{z} T
             - \mu_{T} \Delta T
             - \nu_{T} \pd{zz} T
             = F_{T} + \sigma_{T}(\mathbf{v},T,S) \dot{W}_2,
    \label{eq:diffEqnTempPERF}\\
    \pd{t} S + \nabla_{\mathbf{v}} S
             + w(\mathbf{v}) \pd{z} S
             - \mu_{S} \Delta S
             - \nu_{S} \pd{zz} S
             = F_{S} + \sigma_{S}(\mathbf{v},T,S) \dot{W}_3,
    \label{eq:diffEqnSaltPERF}\\
    \rho = \rho_0 ( 1 - \beta_T( T - T_r) + \beta_S(S- S_r)),
    \quad P = P(S,T) = g \int_{z}^{0} \rho d\bar{z},
    \label{eq:linearDensityDependenceRF}\\
     w(\mathbf{v})(\cdot,z) = \int^0_z \nabla \cdot \mathbf{v} d\bar{z},  \quad
    \nabla \cdot \int_{-h}^0  \mathbf{v} d\bar{z} = 0.
    \label{eq:divFreeTypeCondPERF}
  \end{gather}
\end{subequations}
This reformulation is desirable as, in particular,
it is more suitable for the typical functional setting of the equations which we describe next.
The unknowns and parameters in the equations are precisely those given above
immediately after \eqref{eq:PE3DBasic}.  Of course \eqref{eq:PE3DBasicReform} is subject
to the same initial and boundary conditions as in \eqref{eq:PE3DBasic}, namely
\eqref{eq:3dPEBCTop}--\eqref{eq:InitialCondPEs}.  For further details concerning the equivalence of
\eqref{eq:PE3DBasicReform} and \eqref{eq:PE3DBasic} see \cite{PetcuTemamZiane}.

\subsection{The Functional Setting and Connections with the Abstract Framework}
\label{sec:PrimEqFn}
We now proceed to introduce the basic function spaces associated with the Primitive
equations \eqref{eq:PE3DBasicReform} (equivalently \eqref{eq:PE3DBasic})
and then introduce and explain the
variational formulation of the various terms in equation connecting them
with the abstract assumptions laid out above in Section~\ref{sec:setup}.

\subsubsection{Basic Function Spaces}
To begin we define the smooth test functions
\begin{align*}
 \mathcal{V} := \mathcal{V}_{1} \times \mathcal{V}_{2}
 	= \left\{ \mathbf{v} \in \mathcal C^{\infty}(\bar{\mathcal{M}})^{2}
		: \nabla \cdot \int_{-h}^{0} \mathbf{v} dz = 0, \mathbf{v}_{| \Gamma_{l} \cap \Gamma_{b}} = 0
		\right\} \times
		 \left\{ (T,S) \in \mathcal C^{\infty}(\bar{\mathcal{M}})^{2}
		: \int_{\mathcal{M}} S d\mathcal{M} = 0
        \right\}.
\end{align*}
We now take $H$ to be the closure of $\mathcal{V}$ in $L^{2}(\mathcal{M})^{4}$ or, equivalently,
$H := H_{1} \times H_{2}$ where,
\begin{align}
    \left\{ \mathbf{v} \in L^2(\mathcal{M})^2:
           \nabla \cdot  \int_{-h}^0 \mathbf{v} \,dz = 0,
            n_H \cdot \int_{-h}^0 \mathbf{v} \,dz = 0 \textrm{ on } \partial \Gamma_{i} \right\}
   	\!         \times \!
             \left\{ (T,S) \in L^2(\mathcal{M})^2:
       \int_{\mathcal{M}} S \,d\mathcal{M} = 0 \right\}.
\label{eq:L2velFields}
\end{align}
On $H$ it is convenient to define the inner product and norm according to:
\begin{align*}
  (U,\tilde{U})_H := \int_{\mathcal{M}}
   (\mathbf{v} \cdot \tilde{\mathbf{v}} +
   K_{T}T\tilde{T}  + K_{S} S \tilde{S} ) d\mathcal{M}
      , \quad
  |U| := (U,U)_H^{1/2}.
\end{align*}
The constants $K_{T}, K_{S}>0$, which are
introduced for coercivity in the principal linear
terms in the equations, are chosen
in order to fulfill \eqref{eq:AvarDef} for \eqref{eq:AvariationalDef} below. We define $\LP$
to be the orthogonal (Leray-type) projection from $L^{2}(\mathcal{M})^{4}$
onto $H$.

We shall next define the $H^{1}$ type space $V = V_{1} \times V_{2}$ where
\begin{align}
 \left\{ \vel \in H^1(\mathcal{M})^2:
            \int_{-h}^0 \nabla \cdot  \mathbf{v} \,dz = 0, \,
            \vel = 0 \textrm{ on } \Gamma_l \cup \Gamma_b
        \right\} \times
         \left\{ (T,S) \in H^1(\mathcal{M})^2:
            \int_{\mathcal{M}} S \,d\mathcal{M} = 0
        \right\},
\label{eq:H1vectFields}
\end{align}
We endow $V$ with the inner product and norm
\begin{align}
  ((U,\tilde{U}))_V := ((U,\tilde{U}))_{\mathbf{v}} + K_{T}(( U,\tilde{U} ))_{T} + K_{S}(( U,\tilde{U} ))_{S},
  \quad \|U\| := ((U,U))^{1/2}.
  \label{eq:Vinnerprod}
\end{align}
where
\begin{align*}
 ((U,\tilde{U}))_{\mathbf{v}} &:= \int_{\mathcal{M}} (
  \mu_{\mathbf{v}} \nabla \mathbf{v} \cdot \nabla \tilde{\mathbf{v}} +
  \nu_{\mathbf{v}}\pd{z}  \mathbf{v} \cdot \pd{z} \tilde{\mathbf{v}})
    \,d\mathcal{M}   + \alpha_{\vel} \int_{\Gamma_{i}} \vel \cdot \tilde{\vel} d \Gamma_{i},\\
 (( U,\tilde{U} ))_{T} &:=  \int_{\mathcal{M}} ( \mu_{T} \nabla T \cdot \nabla \tilde{T}+
    \nu_{T} \pd{z}  T \cdot \pd{z} \tilde{T})
  \,d\mathcal{M} + \alpha_{T}\int_{\Gamma_{i}} T \tilde{T} d \Gamma_{i},\\
   (( U,\tilde{U} ))_{S} &:=  \int_{\mathcal{M}} (   \mu_{S} \nabla S \cdot \nabla \tilde{S}+
      \nu_{S}  \pd{z}  S \cdot \pd{z} \tilde{S})
  \,d\mathcal{M}.
 \end{align*}
From \eqref{eq:H1vectFields}--\eqref{eq:Vinnerprod} we may deduce the Poincar\'{e} type inequality
  $|U| \leq c \|U\|,$
  for every $U \in V$.
This justifies taking $\| \cdot \|$ as the norm for $V$ (which is equivalent to the $H^{1}$ norm).  Finally we
define:
\begin{align}
  V_{(2)}, V_{(3)}  \textrm{ are the closures of } \mathcal{V} \textrm{ in }
   H^2(\mathcal{M})^4, H^{3}(\mathcal{M})^{4} \textrm{ norms respectively}
\label{eq:HigherOrderSpaces}
\end{align}
and simply endow $V_{(2)}$ and $V_{(3)}$ with, respectively, the $H^{2}(\mathcal{M})$ and $H^{3}(\mathcal{M})$ norms.
Let $V'$ (resp. $V_{(2)}'$, $V_{(3)}'$) be the dual of $V$ (resp. $V_{(2)}$, $V_{(3)}$)
relative to the $H$ inner product.

It is clear with the Rellich-Kondrachov theorem and standard
facts about Hilbert spaces that the spaces introduced in \eqref{eq:L2velFields}--\eqref{eq:HigherOrderSpaces}
provide a suitable Gelfand-Lions inclusion as desired for \eqref{eq:BasicSpaces}.  On this functional basis
we now turn to describe the  variational form of \eqref{eq:PE3DBasicReform}.

\subsubsection{The Variational Form of the Equations}
To capture most of the linear structure in \eqref{eq:PE3DBasicReform}
we define the operator $A$ as a continuous linear map from $V$
to $V'$ via the bilinear form:
\begin{equation}
  a(U, \tilde{U}) := ((U,\tilde{U}))_{V}
  - \int_{\mathcal{M}} \left( g  \int_z^0 (\beta_S S - \beta_T T ) d\bar{z} \right)\nabla\cdot \tilde{\vel} d\mathcal{M}.
  \label{eq:AvariationalDef}
\end{equation}
We observe that if $K_T$, $K_S$ in \eqref{eq:Vinnerprod} are chosen sufficiently
large then, $a$ is \emph{coercive}, namely it satisfies the condition
required by  \eqref{eq:AvarDef}.

We next define the main nonlinear portion of \eqref{eq:PE3DBasicReform}.
Motivated by \eqref{eq:divFreeTypeCondPERF} we take
$w= w(U) := \int_z^0 \nabla \cdot \vel \,d\bar{z}$
and then define a bilinear form $B: V\times V \rightarrow V_{(2)}'$ via the trilinear form
\begin{equation}\label{eq:nonlinearDualParing}
  b(U, \tilde{U}, U^*)
  := b_\vel(U, \tilde{U}, U^*)
     + K_{T} \cdot b_T(U, \tilde{U}, U^*) + K_{S} \cdot b_{S}(U, \tilde{U}, U^*),
\end{equation}
where
\begin{align*}
   b_\vel(U, \tilde{U}, U^*)
    :=&  \int_{\mathcal{M}}\left( (\vel \cdot \nabla_2) \tilde{\vel}
        + w(U) \pd{z} \tilde{\vel} \right) \cdot \vel^{*} \,d \mathcal{M},\\
    b_T(U, \tilde{U}, U^*) :=&
    \int_{\mathcal{M}} \left((\vel \cdot \nabla) \tilde{T}
        + w(U) \pd{z} \tilde{T}\right) T^{*} \,d \mathcal{M},\\
    b_S(U, \tilde{U}, U^*) :=&
    \int_{\mathcal{M}} \left((\vel \cdot \nabla) \tilde{S}
        + w(U) \pd{z} \tilde{S}\right) S^{*} \,d \mathcal{M}.
\end{align*}
To capture the rotation (Coriolis) term in \eqref{eq:MomentumPERF} we
define $E: H \rightarrow H$ via:
\begin{align}
	e(U, \tilde{U}) = \int_{\mathcal{M}} (2 f \mathbf{k} \times \vel) \cdot \tilde{\vel} d \mathcal{M}.
         \label{eq:rotationalTerm}
\end{align}
Note carefully that $a$, $e$ and $b$ satisfy the conditions imposed in Section~\ref{sec:BasicOperators}
which we used in the abstract result Theorem~\ref{thm:MainExistenceMGSol}.
The inhomogenous terms in \eqref{eq:PE3DBasicReform} are given by the
element $\ell$ defined according to
\begin{align}
 \ell(\tilde{U}) =& \int_{\mathcal{M}} (F_\vel \tilde{v} + K_TF_T \tilde{T}+ K_S F_S \tilde{S}) d\mathcal{M}
        +\int_{\mathcal{M}} \left(g  \int_z^0 (1 + \beta_T T_r- \beta_S S_r) dz \right) \nabla \cdot \tilde{\vel} d\mathcal{M}
        \notag\\
        &+\int_{\Gamma_i} [ (\tau_\vel + \alpha_\vel \vel^a)\cdot \tilde{v} + \alpha_T T^a\tilde{T} ]d \Gamma_i.
\label{eq:inhomoConcrete}
\end{align}
Note that $\vel^a, \tau_\vel, T^a$, which represent the velocity, shear force of the wind and the temperature
at the surface of ocean are have significant uncertainties and should thus be considered to have a random
component in practice.

\subsection{Some Stochastic Forcing Regimes}
\label{sec:StochasticForcings}

It remains to complete the
connection between \eqref{eq:PE3DBasic} and \eqref{eq:PEsAbstractFormulation} by describing
various physically interesting scenarios for $\sigma(U)\dot{W}$.  We connect these `concrete descriptions'
with the terms $\sigma$ and $\SDT$ in the abstract equation \eqref{eq:PEsAbstractFormulation}
(or equivalently to $g$, $\SDvar$ in \eqref{eq:solEqWeakForm}).  We consider three situations in
detail below.  In each case we describe how to define $\sigma_U$ appearing in \eqref{eq:PE3DBasicReform}
and we then take $\sigma(\cdot) = \Pi \sigma_U(\cdot)$.

\subsubsection{Additive Noise}

The most classical case is to consider an \emph{additive noise} where we suppose that $\sigma_U$ is
independent of $U = (\vel,T,S)$.  In other words $\sigma_{U}: [0,\infty) \times \mathcal{M} \rightarrow (L_{2}(\mathfrak{U}, L^2(\mathcal{M})))^{4}$.
In order to satisfy \eqref{eq:SublinearCondH} we would require that
\begin{align}
	\sup_{t \geq 0} \sum_{k \geq 1} |\sigma^{k}_U(t)|^{2} =
	\sup_{t \geq 0} |\sigma_U(t)|_{L_{2}(\mathfrak{U}, H)}^{2}< \infty.
\label{eq:AdditiveNoiseAssumption}
\end{align}
Note that since, the It\=o and Stratonovich interpretations of \eqref{eq:expStochFormal} coincide in the additive case
we may take $\SDT \equiv 0$ so that \eqref{eq:SublinearConditionsOnSDTerm}
is automatically satisfied.

We also observe that in this case we may give an explicit (if formal) characterization
of the space-time correlation structure of the noise
\begin{align}
  \E \left[ \sigma_{U}(t,\spX)\dot{W}(t, \spX) \cdot \sigma_{U}(s,\mathbf{y})\dot{W}(s,\mathbf{y})\right]
  = K(t,s,\spX,\mathbf{y})\delta_{t-s}
  \label{eq:STCorStructure}
\end{align}
where the correlation kernel $K$ is given by
\begin{align*}
   K(t,s,\spX,\mathbf{y}) = \sum_{k \geq 1} \sigma^{k}_{U}(t,\spX) \cdot \sigma^{k}_{U}(s,\mathbf{y}).
\end{align*}
\begin{Rmk}
Given the condition \eqref{eq:AdditiveNoiseAssumption} the case of space-time white noise is rule
out under our framework.  Of course such a space-time white noise is very degenerate in
space (not even defined in $L^{2}_{x}$) and so such a situation is far from reach due to the highly
nonlinear character of the PEs.   Similar remarks apply to the 3D stochastic Navier-Stokes equations but
see \cite{DaPratoDebussche2003} for the 2-D case.
\end{Rmk}

\subsubsection{Nemytskii Type Operators}

We next consider stochastic forcings of transformations of the unknown $U$
as follows.  Let $\Psi = (\Psi_{\vel}, \Psi_{T},\Psi_{S}): \RR^{4} \rightarrow \RR^{4}$
and suppose, for simplicity,
that $\Psi$ is smooth.  We denote the partial derivatives of $\Psi$ with respect to the
$\vel$, $T$, $S$ variables by $\pd{\vel} \Psi$, $\pd{T} \Psi$, $\pd{S}\Psi$
and the gradient by $\nabla_U \Psi$.
Take a sequence of smooth functions
$\alpha^{k} = \alpha^{k}(\spX): \mathcal{M} \rightarrow \RR$ and
define
\begin{align}
	\sigma^{k}_{U}(U,t, \spX) = \Psi(U) \alpha^{k}(\spX).
\label{eq:SigmaKNemytskii}
\end{align}
We may formally interpret $\sigma_{U}(U) \dot{W} = \Psi(U) \dot{\eta}$
where:
\begin{itemize}
\item $\dot{\eta}$ is a white in time Gaussian process with
the spatial-temporal correction structure
$
  \E (\dot{\eta}(t,\spX) \dot\eta(s,\mathbf{y})) = K(\spX,\mathbf{y})\delta_{t-s}
$
where
$
K(\spX,\mathbf{y}) = \sum_{k \geq 1} \alpha^{k}(\spX) \cdot \alpha^{k}(\mathbf{y})
$.
\item The `multiplication' $\Psi(U)$ and $\dot{\eta}$ may be taken in either the It\={o} or
the Stratonovich sense.
\end{itemize}

We now connect \eqref{eq:SigmaKNemytskii} to \eqref{eq:PEsAbstractFormulation}
in the It\={o} or
the Stratonovich situations in turn illustrating conditions on $\Psi$ and the $\alpha_k$'s
that guarantee that \eqref{eq:SublinearCondH} holds
and in the Stratonovich case that \eqref{eq:SublinearConditionsOnSDTerm} holds.

{\bf The It\={o} Case:}
Suppose that
\begin{align}
	|\Psi(U)|^{2}  \leq c_{\Psi}(1 + |U|^{2})
	\quad &\textrm{ for all } U \in \RR^{4},
	\label{eq:NMOPhiSubLin}
\end{align}
and for the elements $\alpha^{k}$ we suppose that
\begin{align}
	\sum_{k \geq 1} \|\alpha^{k}\|_{V_{(2)}}^{2} < \infty.
\label{eq:NMOalphakSumBnd}
\end{align}
Under \eqref{eq:NMOPhiSubLin}--\eqref{eq:NMOalphakSumBnd}
we have
\begin{align*}
	|\sigma(U)|^{2}_{L_2(\mathfrak{U},H)} \leq   \sum_{k \geq 1} |\Psi(U) \alpha^{k}|^{2}_{L^{2}}
	\leq c_{\Psi} \sum_{k \geq 1} \|\alpha^k\|_{L^{\infty}(\mathcal{M})}^{2} (1 + |U|^{2}_{H}).
	\leq c \sum_{k \geq 1} \|\alpha^k\|_{V_{(2)}}^{2} (1 + |U|^{2}_{H}),
\end{align*}
so that \eqref{eq:SublinearCondH} holds for constant $c_{3}$ that depends
on $c_{\Psi}, \sum_{k \geq 1} \|\alpha^{k}\|_{V_{(2)}}^{2}$ and the constant
in Agmon's inequality.
Note that, since we are considering the case of an It\=o noise, $\SDT \equiv 0$.

{\bf The Stratonovich Case:}
If we understand the multiplication $\Psi(U) \dot{\eta}$ in the Strantonovich sense
then we may convert back to an It\={o} type evolution according to:
\begin{align}
	\Psi(U) \dot{\eta} = \sum_{k \geq 1} \Psi(U) \alpha^k \circ d W^k
		= \SDT_U(U) + \sum_{k \geq 1} \Psi(U) \alpha^k  d W^k
		\label{eq:ConvFormulaNaminOpp}
\end{align}
where
\begin{align*}
	\SDT_U(U,x) =& \Psi(U) \cdot \nabla_U \Psi(U) \sum_{k =1} \alpha^k(x)^2
\end{align*}
See e.g. \cite{Arnaud1974, KloedenPlaten1992} for further details on this conversion formula.
Under the additional
assumption
\begin{align}
	|\nabla_U \Psi (U) | \leq c < \infty
	\quad &\textrm{ for all } U \in \RR^{4},
	\label{eq:NMOPhiBnd}
\end{align}	
we define $\SDT_U(U) := \Pi \SDT(U)$ for
any $U \in H$.
It is clear that
$\SDT$ satisfies
\eqref{eq:SublinearConditionsOnSDTerm}.

\begin{Rmk}\label{rmk:StratFromalSoFar}
We note here that the relationship \eqref{eq:ConvFormulaNaminOpp}
is, for now, only formal; we prove the existence of martingale solutions
for the system that results from a formal application of this conversion
formula (see e.g. \cite{Arnaud1974, KloedenPlaten1992}).   We leave the rigorous
justification of \eqref{eq:ConvFormulaNaminOpp} and the related issues of an
approximation of Wong-Zakai type (\cite{WongZakai1965}) of \eqref{eq:PEsAbstractFormulation} for future
work.  Note however that \eqref{eq:ConvFormulaNaminOpp} has already
been explored in \cite{GreckschSchmalfuss1996,Twardowska1996, ChueshovMillet2011}
in an infinite dimensional fluids context for pathwise solutions
and in \cite{TessitoreZabczyk2006} for martingale solutions of a class of abstract,
nonlinear, stochastic PDEs.
\end{Rmk}

\subsubsection{Stochastic Forcing of Functionals}
Finally we examine the case when we stochastically force
\emph{functionals of the unknown} i.e.
terms which have a non-local dependence on the solution $U$.  For example
consider, for $k \geq 1$ continuous (not necessarily linear)
$\phi^{k} :=  \phi^{k}(U): H \rightarrow \RR$,
and sufficiently smooth
$\alpha^{k} = \alpha^{k}(t, \spX): [0,\infty) \times \mathcal{M} \rightarrow \RR^{4}$.
We define
\begin{align}
  \sigma^{k}_{U}(U,t,\spX) = \phi^{k}(U) \alpha^{k}(t, \spX).
  \label{eq:functionalNoiseStructure}
\end{align}
Here, we interpret $\sigma_U(U) \dot{W}$ in the It\={o} sense.
Subject to, for example,
\begin{align}
  \sup_{k} |\phi^{k}(U)|^{2} \leq c(1+|U|^{2}), \quad
  \sup_{t \geq 0} \sum_{k \geq 1} \|\alpha^{k}(t)\|^{2} < \infty
  \label{eq:functionalNoiseStructureCond}
\end{align}
we obtain a $\sigma$ from \eqref{eq:functionalNoiseStructure} which
satisfies \eqref{eq:SublinearCondH}.
For a `concrete example' of a $\sigma$ of the form \eqref{eq:functionalNoiseStructure}
which satisfies \eqref{eq:functionalNoiseStructureCond}
let $\{\psi^{k}\}_{k \geq1}$
be a sequence of elements in $L^{2}(\mathcal{M})^{2}$ with $\sup_{k} |\psi^{k}|_{L^{2}(\mathcal{M})} <\infty$ and let
$\alpha^{k} \in V$ satisfying the sumability condition in \eqref{eq:functionalNoiseStructureCond}.
We take $\phi^{k}(U) = \int_{M}\vel(\spX) \cdot \psi^{k}(\spX) d\mathcal{M}$ and obtain
\begin{align}
	\sigma(U)\dot{W} = \sum_{k \geq 1} \left(\int_{\mathcal{M}}\vel(\spX) \cdot \psi^{k}(\spX) d\mathcal{M} \right)
	\alpha^{k}(t,\spX) dW^{k}(t).
	\label{eq:StochasticForceFnConcrete}
\end{align}

\appendix
\section{Appendix: Technical Complements}
\label{sec:Appendix}

We collect here, for the convenience of the reader, various  technical results
which have been used in the course of the analysis above.  While some of the material may
be considered to be 	somewhat `classical' by specialists we believe that the stochastic type
results will be useful to the non-probabilists and that the deterministic results will be helpful
for the probabilists.

\subsection{Some Convergence Properties of Measures}
\label{sec:LetsConvergeToaMeasure}

We next briefly review some basic notations of convergence
for collections of Borel probability measures.
In particular we highlight a certain abstract convergence lemma that
has been used in a crucial way in the passage to the limit several times above.
For further details concerning the general theory of convergence in
spaces of probability measures see e.g. \cite{Billingsley1} and \cite{RevuzYor}.

Let $(\mathcal{H},\rho)$ be a complete metric space and denote
by $Pr(\mathcal{H})$ the collection of Borel probability measures on $\mathcal{H}$.
We recall that a sequence $\{\mu_{n}\}_{n \geq 1} \subset Pr(\mathcal{H})$  is said to
\emph{converge weakly}
to a measure $\mu$ on $\mathcal{H}$ (denoted by $\mu_{n} \rightharpoonup \mu$)
if and only if
\begin{align}
   \lim_{n \rightarrow \infty}  \int  f(x)  d \mu_{n}(x)  = \int  f(x) d \mu(x)
   \textrm{ for every bounded continuous
function $f: \mathcal{H} \rightarrow \mathbb{R}$}.
\end{align}
We recall that a collection
$\Lambda \subset Pr(\mathcal{H})$ is said to be \emph{weakly relatively compact} if every
sequence  $\{\mu_{n}\}_{n \geq 1} \subset \Lambda$ possesses a weakly convergent
subsequence.  On the other hand we say that $\Lambda \subset Pr(\mathcal{H})$ is \emph{tight} if,
for every $\epsilon > 0$ there exists a compact set $K_\epsilon \subset \mathcal{H}$
such that $\mu(K_\epsilon) \geq 1 - \epsilon$, for each $\mu \in \Lambda$.  The
\emph{Prokhorov theorem} asserts that these two notions, namely tightness and weak compactness
of probability measures are equivalent.

We also make use of the \emph{Skorokhod embedding theorem} which
states that, whenever $\mu_{n} \rightharpoonup \mu$ on $\mathcal{H}$, then there exists a probability space $(\tilde \Omega, \tilde {\mathcal{F}}, \tilde \Prb)$ and
a sequence of random variables $X_{n}: \tilde \Omega \rightarrow \mathcal{H}$ such that
$\tilde \Prb(X_{n} \in \cdot) = \mu_{n}(\cdot)$ and which converges a.s. to a random variable
$X: \tilde \Omega \rightarrow \mathcal{H}$ with $\tilde \Prb(X \in \cdot) = \mu(\cdot)$.

The following convergence result, found in e.g. \cite{Billingsley1},
relates roughly speaking weak convergence and clustering
in probability, and was used to facilitate the proof of  \eqref{useofconv} in Section \ref{cauchy}:
\begin{Lem}\label{thm:ConvTogether}
Let $(\mathcal{H}, \rho)$ be an arbitrary metric space.
Suppose $X_{n}$ and $Y_{n}$ are $\mathcal{H}$-valued
random variables and let $\mu_{n}(\cdot) = \Prb(X_{n} \in \cdot)$
and $\nu_{n}(\cdot) = \Prb(Y_{n} \in \cdot)$ be the associated sequences
of the probability laws.
If the sequence $\{\mu_{n}\}_{n \geq 0}$ converges weakly to a probability measure $\mu$
and if, for all $\epsilon > 0$
\begin{displaymath}
	\lim_{n \rightarrow \infty} \Prb(\rho(X_{n}, Y_{n}) \geq \epsilon) = 0.
\end{displaymath}
Then $\nu_{n}$ also converges weakly to $\mu$.
\end{Lem}

\subsection{An extension of the Doob-Dynkin Lemma}
 We  extend the Doob-Dynkin Lemma (see e.g. \cite{Oksendal1}) to the case where the image space of the measurable functions are complete separable metric spaces. In order to achieve this goal, let us recall the following notions and results from \cite{dudley}.

 If $(\Omega, \mathcal F)$ is a measure space and $E\subset \Omega$, let $\mathcal F_{E} := \{ B \cap E: B \in \mathcal F \}$. Then $\mathcal F_{E}$ is a sigma algebra of subsets of $E$, and $\mathcal F_{E}$ will be called the $relative$ sigma algebra (of $\mathcal F$ on $E$).
\begin{Prop}\label{doob2}
Let $(\Omega, \mathcal F)$ be any measurable space and $E$ any subset of $\Omega$ (not necessarily in $\mathcal F$). Let $f$ be a function on $E$ with values in a Polish space $\mathcal H$ and  measurable with respect to $\mathcal {F}_E$. Then $f$ can be extended to a function on all of $\Omega$, measurable with respect to $\mathcal F$.
\end{Prop}
\begin{proof}
The proof is direct combining Theorem 4.2.5 and Proposition 4.2.6 in \cite{dudley}.
\end{proof}

Now let $(\mathcal Y, \mathcal M)$ be a measure space, $\mathcal X$ any set, and $\psi$ a function from $\mathcal X$ into $\mathcal Y$. Let $\psi^{-1}[\mathcal M]:= \{ \psi^{-1}(  M): \,\, M \in \mathcal M \}$. Then   $\psi^{-1}[\mathcal M]$ is a sigma algebra of subsets of $\mathcal X$.
\begin{Thm}\label{doob}
We are given a set $\mathcal X$, a measure space $(\mathcal Y, \mathcal M)$, and a function $\psi$ from $\mathcal X$ into $\mathcal Y$. If  a function $\ell$ on $\mathcal X$ with values in a Polish space $\mathcal H$   is  $\psi^{-1}[\mathcal M]$ measurable, then there exists an $\mathcal M$-measurable function $L$ on $\mathcal Y$ such that $\ell=L \circ \psi$.
\end{Thm}
\begin{proof}
Whenever $\psi(u)= \psi(v)$, we have $\ell(u)=\ell(v)$, for if not, let $B$ be a Borel set in $\mathcal H$ with $\ell(u)\in  B$ but $\ell(v) \notin   B$. Then $\ell^{-1} (B)=\psi^{-1} (C)$ for some $C\in \mathcal M$, with $\psi(u)\in C$ but $\psi(v)\notin C$, a contradiction. Thus, $\ell=L \circ \psi$ for some function $L$ from $D:=$ range $\psi$ into $\mathcal H$. For any Borel set $E \subset \mathcal H$, $\psi^{-1} (L^{-1} (E)) = \ell^{-1} (E)= \psi^{-1}(F)$ for some $F\in \mathcal M$, so $F\cap D= L^{-1} (E)$ and $L$ is $\mathcal M_D$ measurable. By Proposition \ref{doob2}, $L$ has a $\mathcal M$-measurable extension to all of $\mathcal Y$.
\end{proof}

\subsection{A Measurable Selection Theorem}
\label{sec:LetsJustBeinMeasurableOK?}

We turn now to restate the  \textit{measurable selection theorem} which was proven in \cite{BensoussanTemam}
and is based on the earlier works \cite{KuratowskiRyllNardzewski1965}, \cite{Castaing1967}.
We employed this result above to establish the existence
of adapted solutions of \eqref{eq:EulerSemi} in Proposition~\ref{thm:ExistenceSelectionThm}.

Firstly we recall the definition of a Radon measure. Let $X$ be a locally compact Hausdorff spaces
 and $\mathcal B (X) $ be the Borel sigma algebra on $X$. A Radon measure on $X$ is a measure defined
  on $\mathcal B (X) $ that is finite on all compact sets, outer regular on all Borel sets, and inner regular on all open sets (Page 212, \cite{Folland1}).
\begin{Thm}\label{thm:MeasureSel}
Let $X$ and $Y$ be separable Banach spaces and suppose that
$\Lambda$ is a `multivalued map' from $X$ into $Y$ i.e. a map from
$X$ into the  subsets of $Y$.  We assume that
$\Lambda$ takes values in closed, non-empty subsets of $Y$
and that its graph is closed viz.
$$
  \textrm{if } x_{n} \rightarrow x
  \textrm{ in $X$, and } y_{n} \rightarrow y
  \textrm{ in $Y$,  with } y_{n} \in \Lambda x_{n},
  \textrm{ then } y \in \Lambda x.
$$
Then, $\Lambda$ admits a \emph{universal Radon measurable section}, $\Gamma$,
that is there exists a map $\Gamma: X \rightarrow Y$ such that
$\Gamma x \in \Lambda x$ for every $x$, and such that $\Gamma$ is  Radon measurable for every Radon measure on $X$.

\end{Thm}
\begin{Rmk}
Note that since $X$ is a separable Banach space, any probability measure on $X$ is Radon; this is because any separable Banach space is a Polish space (separable and complete metric space) and that every Polish space is a Radon space (A Hausdorff space $X$ is called a Radon space if every finite Borel measure on $X$ is a Radon measure, i.e. is inner regular (see \cite{Schwartz}).
\end{Rmk}

The following results are from \cite{Schwartz} and \cite{dunfordschwartz}. The final  goal is to   establish Corollary \ref{cor:cont} below,  which we have employed
in   the article to  prove that the map $\chi  $  defined in (\ref{eq:unn1}) (Section \ref{sec:Construction})   is    universally Radon measurable.
 For that purpose, we need the to introduce the following results (Proposition \ref{lem:lusincontiunous} to Theorem \ref{lem:composition}).

\begin{Def}\label{Def:Lusin}(Lusin $\mu$-measurable)
 Let $X$ be a topological space. Let $\mu$ be a Radon measure on    $X$   and let $h$ maps $X$ into $ Y$ where $Y$ is a Hausdorff topological space. Then the mapping $h$ is said to be Lusin $\mu$-measurable if, for every compact set $K\subset X$ and every $\delta >0$, there exists a compact set $K_{\delta} \subset K$ with $\mu (K-K_{\delta}) \leq \delta $   such that $h$ restricted to $K_{\delta}$ is continuous.
\end{Def}
\begin{Prop}\label{lem:lusincontiunous}
A function whose restriction to every compact set is continuous, is Lusin measurable for every Radon measure (Page 25, \cite{Schwartz}).
\end{Prop}
\begin{Prop}\label{lem:lusin}
The assumptions are the same as in Definition \ref{Def:Lusin}.
If $h: \,X \rightarrow Y$ is Lusin $\mu$-measurable, then $h$ is  $\mu$-measurable, and conversely, if $Y$ is metrizable and separable, then
every  $\mu$-measurable function is also Lusin   $\mu$-measurable (Page 26, \cite{Schwartz}).
\end{Prop}

\begin{Thm}\label{lem:composition}
Let $X$, $Y$ and $Z$ to be separable Banach spaces and $\mu$ be a Radon measure on $X$.  Let $\varphi:X \rightarrow Y$ be a $\mu$-measurable mapping. Let $\Gamma : Y \rightarrow Z$ be universally Radon measurable.
 Then $G:=\Gamma \circ \varphi$ is  $\mu$-measurable on $X$.
\end{Thm}
\begin{proof}
From Proposition \ref{lem:lusin}, $\varphi$ is Lusin $\mu$-measurable. Then Theorem \ref{lem:composition} follows from  the proof of Theorem 3.2 in \cite{BensoussanTemam}.
\end{proof}

\begin{Cor}\label{cor:cont}
Let $X$, $Y$ and $Z$ to be separable Banach spaces and $\varphi:X \rightarrow Y$ be a continuous mapping. Let $\Gamma : Y \rightarrow Z$ be universally Radon measurable. Then $G:=\Gamma \circ  \varphi$ is universally Radon measurable.
\end{Cor}
\begin{proof}
This can be directly deduced from Proposition \ref{lem:lusincontiunous} and Theorem \ref{lem:composition}.
\end{proof}
\subsection{Compact Embedding Results}
\label{sec:CompactEmbRes}
In order to establish the compactness of a sequence of probability measures
associated with the solutions of \eqref{eq:EulerSemiFnRep} we made
use of the following compact embedding theorem which is close to that found in
\cite{Temamfirst} and of course generalizes the classical Aubin-Lions Compactness theorem
(see \cite{Aubin})

\begin{Prop}\label{thm:MotherOfAllCompEmbeddings}
Let $Z \subset \subset Y \subset X$ be a collection of three Banach spaces
with $Z$ compactly embedded in $Y$ and $Y$ continuously embedded in
$X$.
\begin{itemize}
\item[(i)] Suppose that $\mathfrak{G}$ is a bounded subset of
$L^p(\mathbb{R}, Z) \cap L^{\infty}(\mathbb{R}, Y)$, where $1 < p \leq \infty$, and assume that
for some $1 < q < \infty$
\begin{align}
	\int_{-\infty}^{\infty} |g(t + s) - g(s)|^{q}_{X} ds \rightarrow 0
	\quad \textrm{ as } t \rightarrow 0,
	\label{eq:UniformTmShiftConv}
\end{align}
uniformly for $g \in \mathfrak{G}$ and that there exists $L > 0$ such that
\begin{align}
	\mbox{supp}\{g\} \subset [-L, L], \textrm{ for every } g \in \mathfrak{G}.
	\label{eq:BndSupp}
\end{align}
Then, the set $\mathfrak{G}$ is relatively compact in $L^{p}(\mathbb{R}, Y)$.
\item[(ii)]  For $T > 0$ if $\mathfrak{G}$ is a bounded subset of
$L^p( 0, T , Z) \cap L^{\infty}( 0, T , Y)$ and
\begin{align}
	\int_{0}^{T-a} |g(t + s) - g(s)|^{q}_{X} ds \rightarrow 0
	\quad \textrm{ as } t \rightarrow 0,
	\label{eq:UniformTmShiftConvFixedT}
\end{align}
uniformly for elements in $\mathfrak{G}$, then $\mathfrak{G}$ is relatively compact in $L^{p}( 0, T , Y)$.
\end{itemize}
\end{Prop}

\begin{proof}
The proof is a fairly straightforward generalization of \cite[Theorem 13.2]{Temam4}.
Observe that if $q > p$ then \eqref{eq:UniformTmShiftConv} and \eqref{eq:BndSupp}
taken together imply that
\begin{align*}
	\int_{-\infty}^{\infty} |g(t + s) - g(s)|^{p}_{X} ds \rightarrow 0
	\quad \textrm{ as } t \rightarrow 0,
\end{align*}
uniformly for $g \in \mathfrak{G}$.  Therefore there is no loss of generality in supposing
that $q \leq p$ in what follows.

For $a >0$, define the \emph{averaging operator} $J_{a}$ according to
\begin{align*}
	(J_{a}f)(s) = \frac{1}{2a} \int_{s-a}^{s+a} f(t) dt = \frac{1}{2a}\int_{-a}^{a} f(s+t) dt.
\end{align*}
We take $\mathfrak{G}_{a} = \{ J_{a} g: g \in \mathfrak{G}\}$.
Arguing exactly as in \cite{Temamfirst} we have, for $a > 0$, that
$\mathfrak{G}_{a}$ is relatively compact in $L^{p}(\mathbb{R}; Y)$.

To show that $\mathfrak{G}$ is itself relatively compact in $L^{p}(\mathbb{R}; Y)$,
we prove that it is a
\emph{totally bounded} subset of $L^{p}(\mathbb{R}; Y)$; in other words we prove that, for every $\epsilon >0$,
there exists finitely many elements $g_{1}, \ldots, g_{N}$ in $L^{p}(\mathbb{R}, Y)$
such that $\mathfrak{G}$ is contained in the union of the $\epsilon$ balls centered
at these points.

Again, arguing exactly as in \cite{Temamfirst} we have that, as a consequence of
\eqref{eq:UniformTmShiftConv}, for every $\delta >0$ there exists $a = a(\delta) > 0$ such that
\begin{align}
	|J_{a} g - g|_{L^{q}(\mathbb{R}, X)} \leq  \delta,
	\quad \textrm{ for every } g \in \mathfrak{G}.
	\label{eq:DelsEps1ForComp}
\end{align}
On the other hand, from  \cite[Chapter 3, Lemma 2.1]{Temam1}  we infer that, for every $\eta > 0$, there exists $C_{\eta} >0$ such that,
for every $g \in L^{p}(\mathbb{R}, Z)$
\begin{align}
	|J_{a} g - g|_{L^{p}(\mathbb{R}, Y)} \leq
	C_{\eta} |J_{a} g - g|_{L^{p}(\mathbb{R}, X)}  + \eta
	|J_{a} g - g|_{L^{p}(\mathbb{R}, Z)}
	\leq
	C_{\eta} |J_{a} g - g|_{L^{p}(\mathbb{R}, X)}  + 2\eta
	|g|_{L^{p}(\mathbb{R}, Z)}.
	\label{eq:DelsEps2ForComp}
\end{align}
The last inequality follows from the fact that, $|J_{a}f|_{L^{p}(\mathbb{R}, Z)}
\leq |f|_{L^{p}(\mathbb{R}, Z)}$, for all $f \in L^{p}(\mathbb{R}, Z)$.  Now,
on the other hand we have
\begin{align*}
 |J_{a} g - g|_{L^{p}(\mathbb{R}, X)}
 \leq |J_{a}g - g|_{L^{\infty}(\mathbb{R}, X)}^{(p-q)/p}  |J_{a} g - g|_{L^{q}(\mathbb{R}, X)}^{q/p}
 \leq  (2 |g|_{L^{\infty}(\mathbb{R}, X)})^{(p-q)/p} |J_{a} g - g|_{L^{q}(\mathbb{R}, X)}^{q/p}.
\end{align*}
So, for a constant $\kappa$ depending only on $\sup_{g \in \mathfrak{G}}|g|_{L^{\infty}(\mathbb{R}, Y)}$, $p$,
$q$ and the constant associated with the continuous embedding of $Y$ into $X$ we find
\begin{align}
 |J_{a} g - g|_{L^{p}(\mathbb{R}, X)} \leq
 \kappa |J_{a} g - g|_{L^{q}(\mathbb{R}, X)}^{q/p}.
	\label{eq:DelsEps3ForComp}
\end{align}
Fix $\epsilon > 0$. Let  $|f|_{L^{p}(\mathbb{R}, Z)}\leq \kappa,\,\, \forall f \in \mathfrak{G}$ and let  $\eta = \epsilon / (6 \kappa)$, and
pick $a >0$, sufficiently small, so that
\eqref{eq:DelsEps1ForComp} holds for $\delta := \left(\frac{\epsilon}{3 C_{\eta} \kappa}\right)^{p/q}$,
where $C_{\eta}$ is the constant corresponding to $\eta$ in \eqref{eq:DelsEps2ForComp}.
Using that $\mathfrak{G}_{a}$ is precompact in $L^{p}(\mathbb{R}, Y)$, we next
choose a finite collection $\mathfrak{F} = \{g_{1}, \ldots, g_{n}\} \subset \mathfrak{G}$, such that the $L^{p}(\mathbb{R}, Y)$
$\epsilon/3$-balls centered at $J_{a}g_{k}$ cover $\mathfrak{G}_{a}$.
Now, with these various choices, we have that for
any $g \in \mathfrak{G}$, there exists $g_{k} \in \mathfrak{F}$ such that $|J_{a} g_{k} - J_{a}g|_{L^{p}(\mathbb{R}, Y)} \leq \epsilon/3$.
As such, we  employ \eqref{eq:DelsEps2ForComp} with
$\eta = \epsilon / (6 \kappa)$, followed by
\eqref{eq:DelsEps3ForComp} and estimate
\begin{align*}
	|J_{a} g_{k} - g|_{L^{p}(\mathbb{R}, Y)}
	\leq&	|J_{a} g_{k} - J_{a}g|_{L^{p}(\mathbb{R}, Y)} + |J_{a} g - g|_{L^{p}(\mathbb{R}, Y)}
	\leq \frac{\epsilon}{3} + C_{\eta} |J_{a} g - g|_{L^{p}(\mathbb{R}, X)}  + 2\eta
		\sup_{f \in \mathfrak{G}}|f|_{L^{p}(\mathbb{R}, Z)}\\
	\leq& \frac{2\epsilon}{3} + C_{\eta}\kappa |J_{a} g - g|_{L^{q}(\mathbb{R}, X)}^{q/p}  	
		\leq \epsilon.
\end{align*}
Since, $\epsilon > 0$ was arbitrary to begin with, this shows
that $\mathfrak{G}$ is a totally bounded subset of $L^{p}(\mathbb{R}; Y)$
and we thus infer $(i)$.
The second item $(ii)$ follows directly from $(i)$ as in \cite{Temamfirst}.
The proof of Proposition~\ref{thm:MotherOfAllCompEmbeddings}
is therefore complete.
\end{proof}

\section*{Acknowledgments}
This work was partially supported by the National Science Foundation (NSF)
under the grants  DMS 1206438 and by the Research Fund of
Indiana University. NEGH work has been partially supported under the grants NSF-DMS-1004638 and NSF-DMS-1313272.
We have benefited from the hospitality of the Department of Mathematics Virginia Tech and
from the Newton Institute for Mathematical Sciences, University of Cambridge where the final stage
of the writing was completed. The authors wish to thank Arnaud Debussche
for his help with the use of the universally Radon measurable selection theorem.

\begin{footnotesize}
\bibliographystyle{amsalpha}
\bibliography{ref}
\end{footnotesize}

\vspace{.3in}

\noindent Nathan Glatt-Holtz\\ {\footnotesize
Department of Mathematics\\
Virginia Polytechnic Institute and State University\\
Web: \url{http://www.math.vt.edu/people/negh/}\\
 Email: \url{negh@vt.edu}} \\[.3cm]

\noindent Roger Temam\\ {\footnotesize
Department of Mathematics\\
Indiana University\\
Web: \url{http://mypage.iu.edu/~temam/}\\
 Email: \url{temam@indiana.edu}} \\[.3cm]

\noindent Chuntian Wang\\ {\footnotesize
Department of Mathematics\\
Indiana University\\
Web: \url{http://mypage.iu.edu/~wang211/}\\
 Email: \url{wang211@umail.iu.edu}}

\end{document}